%% file: main.tex
\documentclass[12pt]{article}
\usepackage[utf8]{inputenc}

\usepackage{amsmath}
\usepackage{amsthm}

\usepackage{algorithm}
\usepackage{algorithmic}
\usepackage{amssymb}
\usepackage{braket}
\usepackage{float}
\usepackage{gensymb}
\usepackage{tikz}
\usepackage{tocloft}
\usepackage{booktabs}
\usepackage{tabularx,makecell,multirow}
\usepackage{graphics,pstricks}

\usepackage{pifont}
\usepackage{float}
\usepackage{subfig}
\usepackage{tikz}
\usetikzlibrary{decorations.pathmorphing,knots}
\usepackage[matrix,arrow]{xy}
\tikzstyle WL=[line width=2pt,opacity=1.0]
\tikzstyle WLL=[line width=5pt,opacity=1.0]

\usepackage[paperheight=11in,
             paperwidth=8.5in,
             outer=1in,
             inner=1in,
             bottom=1in,
             top=1in]{geometry}

\usetikzlibrary{shapes}
\newtheorem{defi}{Definition}[section]
\newtheorem{rmk}[defi]{Remark}
\newtheorem{prop}[defi]{Proposition}
\newtheorem{thm}[defi]{Theorem}
\newtheorem{eg}[defi]{Example}
\newtheorem{coro}[defi]{Corollary}
\newtheorem{lem}[defi]{Lemma}

\usepackage{tikz}
\usetikzlibrary{calc}
\usepackage{tikz}
\usetikzlibrary{calc} 

\usetikzlibrary{calc}

\usepackage{ifthen}
\usepackage{ulem}

\usepackage[colorlinks,linkcolor=blue]{hyperref}
\usepackage{bbm}

\newtheorem*{maintheorem}{Main Theorem}
\newtheorem*{comparison}{Comparison Theorem}

\usepackage{appendix}

\usepackage[nottoc]{tocbibind}
\usepackage{float}

\begin{document}

\title{G-equivariant Alterfold TQFT in any dimension}
\author{%
  Zixuan Feng\textsuperscript{1}\qquad Zhengwei Liu\textsuperscript{1,2}\\[0.6em]
  {\small\textsuperscript{1}Tsinghua University, Beijing, China}\\
  {\small\textsuperscript{2}Beijing Institute of Mathematical Sciences and Applications, Beijing, China}\\[0.4em]
  {\small Email (Zixuan Feng): \texttt{fzx22@mails.tsinghua.edu.cn}}\\[0.4em]
  {\small Email (Zhengwei Liu): \texttt{zhengweiliu@mail.tsinghua.edu.cn}}
}

\date{}

\maketitle
\begin{abstract}
For a discrete group $G$, we construct an $(n+1)$-dimensional $G$-equivariant Alterfold topological quantum field theory from a $G$-equivariant $S^n$-algebra satisfying strong semisimplicity, complete finiteness, and Morita invariance. 
We use state sum on triangulation to define a partition function and prove that it is invariant under local surgery relations and is invariant under gauge transformations relative to the boundary. 
If the input is given a reflection-positive structure, the theory is unitary. 
For $n=2$ and input reconstructed from a strict faithful spherical $G$-fusion category, the restriction to ordinary $G$-cobordisms is naturally isomorphic to the Turaev--Virelizier state-sum HQFT \cite{turaev20123}.
A Tambara-Yamagami example recovers the Quon language \cite{liu2017quon} as a $\mathbb{Z}_2$-equivariant Alterfold TQFT.
\end{abstract}
\tableofcontents

\section{Introduction}

A topological quantum field theory is a symmetric monoidal functor from a cobordism category to a suitable category of vector spaces \cite{atiyah1988topological}. 
State sums construct such theories from local weights on triangulations, as in the work of Turaev and Viro \cite{turaev1992state} and Barrett and Westbury \cite{barrett1996invariants}. 
Equivariant theories incorporate a background principal $G$-bundle into the cobordism category.

Group extensions appear in several related forms. For a fusion category $\mathcal C$, a $G$-extension is a faithfully $G$-graded fusion category with neutral component equivalent to $\mathcal C$. For finite $G$, Etingof, Nikshych, and Ostrik describe these extensions through the Brauer--Picard $3$-groupoid and obtain obstruction and classification results \cite{etingof2005fusion,etingof2010fusion}. In this formulation, the graded components are invertible bimodule categories, and constructing an extension requires coherent multiplication and associativity data.

In the study of topological phases, symmetry-protected topological (SPT) phases and symmetry-enriched topological (SET) phases bring complementary aspects of this structure into view. Chen, Gu, Liu, and Wen construct bosonic SPT models from group cohomology \cite{chen2013symmetry}. For $2+1$-dimensional topological phases, Barkeshli, Bonderson, Cheng, and Wang develop the theory of symmetry fractionalization, symmetry defects, and gauging using $G$-crossed braided categories \cite{barkeshli2019symmetry}. The grading records defect sectors, while the crossed structure encodes their interaction with the symmetry.

The gauge-theoretic and orbifold viewpoints also distinguish a specified background from a sum over backgrounds. Dijkgraaf and Witten construct topological gauge theories from group cohomology \cite{dijkgraaf1990topological}; for finite gauge groups, Freed and Quinn give a rigorous quantization \cite{freed1993chern}. Twisted sectors are central to the classical orbifold constructions of Dijkgraaf, Vafa, Verlinde, and Verlinde \cite{dijkgraaf1989operator}. In the TQFT setting, Schweigert and Woike construct pushforwards along homomorphisms of finite groups, with orbifoldization given by pushforward to the trivial group \cite{schweigert2019orbifold}. Gauging and orbifoldization combine sectors; a theory with fixed background retains them separately.

For TQFTs, the fixed-background formulation is provided by Turaev's homotopy quantum field theory (HQFT), whose manifolds carry maps to a target space \cite{turaev1999homotopy}. For a discrete group, the relevant target is $K(G,1)$. Crossed group-categories supply three-dimensional examples \cite{turaev2000homotopy}, and Turaev and Virelizier construct state-sum HQFTs from spherical $G$-fusion categories \cite{turaev20123}. These developments motivate the present problem: to construct equivariant state sums in arbitrary dimension directly from local algebraic data. 

The Alterfold approach began with the construction by Liu, Ming, Wang, and Wu of quantum invariants of $3$-manifolds equipped with decorated separating surfaces \cite{liu20233}.
They subsequently constructed the associated $(2+1)$-dimensional TQFT, placing the Turaev--Viro theory of a spherical fusion category and the Reshetikhin--Turaev theory of its Drinfeld center in a common framework \cite{liu2026alterfold}.
Their further work applies alterfold methods to modular invariance and $\alpha$-induction in Morita contexts \cite{liu2026alterfold4}.
Liu developed a hypersphere-functional construction of Alterfold TQFTs in arbitrary dimension \cite{liu2024functional}.
We develop a $G$-equivariant version of this construction with both space and time boundaries.

Fix a discrete group $G$ and an integer $n\ge 1$. 
Our algebraic input is a $G$-equivariant $S^n$-algebra $\mathcal V=(V,Z)$. 
Its disk operations describe the insertion of labeled regions, and its sphere functional $Z$ evaluates closed labeled $n$-spheres. 
Idempotent completion emerge a spherical $n$-category, organizes the local labels into morphisms and supplies the traces and global dimensions used in the state sum.

We impose strong semisimplicity (SS), complete finiteness (CF), and Morita invariance (MI). The first two is used in the construction of \cite{liu2024functional} which provide semisimple local algebras, nonzero global dimensions, and finite summations in each prescribed sector. 
We demand an extra condition \textup{(MI)}, which supplies the normalization needed to compare the group sectors in the lowest adjacent relation.
The role of \textup{(MI)} is specific to the group background. 
In a Move~$1$, the bulk determines the transport along the core interval, and hence the sector over which one may sum. For an irreducible $0$-morphism $\alpha$, the required identity is
\[
\sum_{\substack{\beta\text{ an irreducible }\alpha\text{--}\alpha\text{ bimodule}\\
                  \operatorname{gr}(\beta)=g}}
\frac{\operatorname{Tr}(\beta)^2}{\mu(\beta)}=
\sum_{\substack{\beta\text{ an irreducible }\alpha\text{--}\alpha\text{ bimodule}\\
                  \operatorname{gr}(\beta)=h}}
\frac{\operatorname{Tr}(\beta)^2}{\mu(\beta)}
,
\qquad \forall g,h\in G,
\]
where the sum is over the chosen representatives. This is the \textup{(MI)} condition. It provides the normalization in the $0$--$1$ adjacent identity. For the higher adjacent identities, the fixed link boundary already specifies the admissible bimodule summation. In particular, none of these identities is obtained by summing over all group backgrounds.

The geometric input is a stratified PL manifold whose cooriented codimension-one strata carry elements of $G$, with a flatness condition at codimension two. The wall labels record parallel transport. Relative isotopies, refinements with identity labels on new walls, and changes of chamber trivialization describe gauge equivalence while preserving the boundary data.

Before stating the construction, we describe its source category. The objects of the \textbf{$G$-equivariant spacetime cobordism category} are compact oriented $G$-labeled PL $n$-manifolds $\mathcal F=(F,P_F)$, possibly with boundary. A morphism from $\mathcal F_-$ to $\mathcal F_+$ is represented by a compact oriented $G$-labeled PL $(n+1)$-manifold $B$ with a boundary decomposition
\[
\partial B=M\cup\iota_-(\overline{F_-})\cup\iota_+(F_+).
\]
Here $M$ is the space boundary and carries a $\mathcal V$-labeling. The incoming and outgoing time boundaries are identified with $\overline{F_-}$ and $F_+$ by the specified embeddings $\iota_-$ and $\iota_+$; they carry the prescribed $G$-labeling but no $\mathcal V$-labels. The time boundaries meet $M$ along their boundaries and admit $G$-preserving stratified product collars. Morphisms are taken modulo PL isomorphisms preserving these identifications and decorations, together with gauge transformations relative to the entire boundary. Composition is gluing along the identified time boundaries, and the monoidal product is disjoint union. Thus time boundaries may themselves have boundary, with the space boundary completing the boundary of the cobordism. Section~5 gives the full definitions.

Our main construction can be stated as follows.

\begin{maintheorem}
Let $\mathcal V$ be a $G$-equivariant $S^n$-algebra satisfying \textup{(SS)}, \textup{(CF)}, and \textup{(MI)}. 
There is a well defined partition function on bulk vectors which is independent of the chosen full transverse triangulation, invariant under gauge transformations relative to the boundary, and satisfies the global $S^{k-1}$-relations. 
Using the universal reconstruction procedure this partition function determines a symmetric monoidal functor from the associated $G$-equivariant spacetime cobordism category to finite-dimensional complex vector spaces. 
If $\mathcal V$ is equipped with a compatible reflection-positive structure, the functor is unitary.
\end{maintheorem}

This combines Theorems~6.23, 6.26, 6.29, and 6.38. The state space is obtained by taking the span of compatible cobordisms and quotienting by the radical of the gluing pairing, following the universal-construction principle of Blanchet, Habegger, Masbaum, and Vogel \cite{blanchet1995topological}. Our surgery relations give finite-dimensional state spaces, and collar representatives establish compatibility with the monoidal product. In particular, applying the relations to an identity cylinder and splitting each relation into two halves yields an embedding into a finite direct sum of copies of $\mathbb C$. With reflection-positive data, the two halves are adjoints and the embedding is isometric.

The partition function is computed by applying surgery relations along a full transverse triangulation, processing faces before the simplices which contain them. Fullness means that a simplex whose vertices all lie on the boundary is itself contained in the boundary. Together with the product collars, this condition allows the surgeries to continue in the prescribed order. We prove subdivision invariance from the adjacent identities, evaluate the remaining sum in a shellable ball, and then establish isotopy invariance. Common subdivisions give triangulation independence. Gauge invariance is proved separately: after a relative refinement, chamber changes can be localized in collared balls and compared without altering the exterior. The shellable refinements use the derived-subdivision results of Adiprasito and Benedetti \cite{adiprasito2017subdivisions}. The relative PL results needed for these arguments are stated in Section~2 and proved in the appendix.

In dimension three, the construction has a direct comparison with the categorical state sum. A strict faithful spherical $G$-fusion category $\mathcal C=\bigoplus_{g\in G}\mathcal C_g$ defines an input $\mathcal V_{\mathcal C}$ by replacing a wall of degree $g$ with the direct sum of the simple objects in $\mathcal C_g$ and using the graphical calculus of $\mathcal C$ for the local operations.

\begin{comparison}
Let $\mathcal C$ be a strict faithful spherical $G$-fusion category over $\mathbb C$ with nonzero neutral-component dimension. For the same $G$-background and boundary identifications, the alterfold TQFT associated with $\mathcal V_{\mathcal C}$, restricted to ordinary $G$-cobordisms without space boundary, is symmetrically monoidally naturally isomorphic to the Turaev--Virelizier state-sum HQFT associated with $\mathcal C$.
\end{comparison}

The comparison identifies the universal state spaces with the images of the Turaev--Virelizier cylinder projectors \cite{turaev20123}. Its normalization uses $\dim(\mathcal C_e)$, the dimension of the neutral component. The factors attached to boundary colors account for the difference between the two presentations of the cobordism maps. This gives an identification of the functors, including their gluing maps. The comparison concerns the target $K(G,1)$; the crossed-module construction of S\"ozer and Virelizier \cite{sozer2023monoidal} also permits homotopy $2$-types and has a different range of backgrounds.

A further family of inputs is provided by Tambara--Yamagami categories \cite{tambara1998tensor,etingof2010fusion} for a group $A$. 
The cyclic case $A=\mathbb Z_d$ gives the categorical setting for comparison with the original Quon language of Liu, Wozniakowski, and Jaffe \cite{liu2017quon}. 
That construction uses $\mathbb Z_d$-charged strings and loop normalization $\sqrt d$. 
Here the charge group $A$ and the background grading group $G=\mathbb Z_2$ have different roles. 
The case $d=2$ gives the Ising fusion rules.
This article identifies Quon language as a $(2+1)$-dimensional $\mathbb{Z}_2$-equivariant Alterfold TQFT.
Recent applications of Quon diagrams to quantum circuits and classical simulation include \cite{kang20252d,feng2025quon}.

Section~2 develops the PL conventions and relative transversality results. Sections~3 and~4 introduce the local algebras, their morphisms, and the three algebraic conditions. Section~5 defines the spacetime cobordism category and relative gauge equivalence. Section~6 constructs the partition function and the TQFT. Section~7 discusses the HQFT comparison and examples related to Quon language. The appendix contains the geometric proofs deferred from the main text.

\include{Manifold}

\include{G-Sn-algebra}

\include{G-category}

\include{G-Alterfold_cobordism}

\include{G-equivariant_TQFT}

\include{Examples}

\section*{Acknowledgement}
\addcontentsline{toc}{section}{Acknowledgements}
We thank Shuang Ming, Yilong Wang and Jinsong Wu for helpful discussions.
The main ideas and constructions of this work were developed by the authors. We used AI as auxiliary tools to help clarify and check certain technical details in PL topology and to assist with manuscript editing. The authors take full responsibility for all mathematical statements, proofs, and conclusions. 
Both authors acknowledge support from the Beijing Natural Science Foundation Key Program (GrantZ220002) and the Beijing Natural Science Foundation (GrantZ221100002722017).

\bibliographystyle{alpha} 
\bibliography{ref}        

\include{app12}

\end{document}

%% file: Manifold.tex
\section{Stratified PL Manifolds}

\subsection{PL manifolds}

In this article, all topological manifolds are assumed to be oriented. 
A zero dimensional topological manifold is a set of points labeled with signs $\pm$ by convention.

\begin{defi}
    A Piecewise Linear (PL) manifold is a topological manifold $M$ equipped with a set of charts $(U_i,\phi_i)_i$ in which $(U_i)_i$ is an open covering of $M$ and $\phi_i:U_i\to \mathbb{H}^n:=\mathbb{R}^{n-1}\times[0,\infty)$ are oriented homeomorphisms onto relative open images such that $\phi_j\phi_i^{-1}:\phi_i(U_i\cap U_j)\to \phi_j(U_i\cap U_j)$ are oriented PL homeomorphisms.
    The set of charts is assumed to be maximal, namely it includes all the charts compatible in the above sense. 
\end{defi}

\begin{defi}
    A continuous map $f:M\to N$ between two PL manifolds $M,N$ is called a PL map if for any chart $(U_i,\phi_i)$ of $M$ and $(V_j,\psi_j)$ of $N$, $\psi_j f\phi_i^{-1}$ on $\phi_i(U_i\cap f^{-1}(V_j))$ is a PL map. 
    It is called a PL homeomorphism if it is bijective and its inverse is a PL map. 
    We write $M\sim_{PL}N$ if there is a PL homeomorphism in between.
\end{defi}

All PL homeomorphisms are assumed to be oriented in the rest of the article, i.e. all $\psi_j f\phi_i^{-1}$ are oriented.

\begin{defi}
    A triangulation of a topological manifold $M$ is a homeomorphism $|K|\to M$ in which $K$ is a (locally finite) simplicial complex. 
\end{defi}

Recall that the convex sum is defined as
\begin{equation}
    A\star B:=\{\lambda a + (1-\lambda)b \mid a\in A,b\in B,0\le \lambda\le 1 \}
\end{equation}

\begin{defi}
    In a simplicial complex $K$, the link $lk_K(\sigma)$ of a $k$-simplex $\sigma\in K$ is 
    \begin{equation}
        lk_K(\sigma)=\{\tau\in K| \sigma\star \tau\in K,\sigma\cap\tau=\emptyset \}
    \end{equation}
    and the star and the open star is respectively
    \begin{equation}
        st_K(\sigma):=\{\mu\in K| \exists \tau\in K,\sigma\subset\tau,\mu\subset\tau\},\quad
        ost_K(\sigma)=\bigcup_{\sigma\subset \tau\in K}Int(\tau)
    \end{equation}
    A $n$-dimensional simplicial complex is called a $n$-dimensional combinatorial manifold, if the link of any $k$-simplex is either $|lk_K(\sigma)|\sim_{PL}S^{n-k-1}$ or $|lk_K(\sigma)|\sim_{PL}D^{n-k-1}$. 
\end{defi}

We take the convention that $S^{-1}=D^{-1}=\emptyset$.

The followings are fundamental results of PL manifolds.

\begin{thm}
    The support $|K|$ of a combinatorial manifold admits a natural structure of a not necessarily orientable PL manifold. 
\end{thm}

In this article we only consider combinatorial manifold for which the base manifold is orientable and is given an orientation.

\begin{defi}
    A triangulation $h:|K|\to M$ of a PL manifold $M$ is called a combinatorial triangulation if $K$ is a combinatorial manifold and $h$ is a PL homeomorphism w.r.t the natural PL structure on $|K|$.
\end{defi}

\begin{thm}\label{common subdivision}
    A compact PL $n$-manifold admits a combinatorial triangulation. 
    Any two combinatorial triangulations admits a common subdivision, which is also combinatorial. 
\end{thm}

\begin{defi}
    For a PL manifold $M$, a PL submanifold is a pair $(N,f)$ where $N$ is a PL manifold and $f:N\to M$ is a PL map and a topological homeomorphism onto the image. 
\end{defi}

\begin{defi}
    Two PL homeomorphisms $\phi_i:M\to M',i=0,1$ are called isotopic if there is a PL map $\Phi:M\times[0,1]\to M'$ such that $\Phi_{M\times\{i\}}=\phi_i,i=0,1$ and $\Phi_t:=\Phi|_{M\times\{t\}}:M\to M'$ are all PL homeomorphisms. 
    Two homeomorphisms onto images $\phi_i:N\to M,i=0,1$ are called ambient isotopic if there is a PL homeomorphism $F:M\to M$ isotopic to the identity $id_M$ such that $F\circ \phi_0=\phi_1$.
\end{defi}

\subsection{Stratification}

\begin{defi}
    A pre stratified PL manifold $\mathcal{M}$ is
    \begin{itemize}
        \item a PL $n$ manifold $M$, $B=\partial M$ being its boundary
        \item a sequence of subsets $\emptyset = M^{-1}\subset M^0\subset M^1\subset \dots \subset M^n=M$
    \end{itemize}
    such that 
    \begin{itemize}
        \item $M^i\subset M$ are closed subsets
        \item define $M_i:=M^i\setminus M^{i-1}$. We demand that $M_i\cap \mathring{M}$ is a PL $i$-dimensional submanifold without boundary
        \item any point $p$ in $M$ has an open neighborhood whose intersection with $M_i$ has finite connected components
        \item the boundary $B=\partial M$ admits a stratified PL homeomorphism $c:B\times[0,1)\to c(B\times[0,1))\subset M$ whose image is open in $M$ and $c(b,0)=b,\forall b\in B$. (Stratified means that $c^{-1}(M^i)=(B\cap M^i)\times[0,1)$.)
    \end{itemize}
\end{defi}

\begin{defi}
    A PL map $f:M\to N$ between two pre stratified PL manifolds $M,N$ is called a stratified PL map if $f(M^i)\subset N^i$ and $f(\partial M)\subset \partial N$.
    It is called a stratified PL homeomorphism if it is bijective and $f$ and $f^{-1}$ are all stratified PL maps.
\end{defi}

\begin{defi}
    When $\mathcal{A}$ is a pre stratified PL manifold, we denote by $c\mathcal{A}$ its open cone.
    It has support $cA=A\times[0,1)/A\times\{0\}=A\times(0,1)\cup \{v\}$, and stratification $(c\mathcal{A})^{k+1}:=c(\mathcal{A}^k),0\le k\le n$ and $(c\mathcal{A})^0:=\{v\}$.
\end{defi}

By convention $c\emptyset=\{v\}$.

\begin{defi}
    When $\mathcal{A}$ is a pre stratified PL manifold, we denote by $\Lambda\mathcal{A}$ its closed cone.
    It has support $\Lambda A=A\times[0,1]/A\times\{0\}=A\times(0,1]\cup \{v\}$, and stratification $(\Lambda\mathcal{A})^{k+1}:=\Lambda(\mathcal{A}^k),0\le k\le n$ and $(\Lambda\mathcal{A})^0:=\{v\}$.
\end{defi}

\begin{defi}\label{stratified PL manifold}(Stratified PL manifold)

    Suppose $\mathcal{M}$ is a pre stratified PL manifold with boundary $B=\partial M$.
    It is called a stratified PL manifold if for any point $p$ in $\mathcal{M}$, suppose that $p\in M_k$, set 
    \begin{equation}
        \mathbb{P}_k:=\begin{cases}
            \mathbb{R}^k,&p\notin B\\
            H^k=\mathbb{R}^{k-1}\times[0,\infty),&p\in B
        \end{cases}
    \end{equation}
    (we demand that $k\ge 1$ if $p\in B$) there is an open neighborhood $U\ni p$, a pre stratified PL manifold $|\mathcal{S}|\sim_{PL} S^{n-k-1}$ and a PL homeomorphism $\phi:\mathcal{M}|_U\to \mathbb{P}_k\times c\mathcal{S}$, such that
    \begin{itemize}
        \item $\phi(p)=(O_{k},v)$
        \item it is a stratified PL homeomorphism, i.e. 
        \begin{equation}
            \phi(M^i\cap U)=
            \begin{cases}
                \emptyset & i<k
                \\
                \mathbb{P}_k\times\{v\} & i=k
                \\
                \mathbb{P}_k\times c(S^{i-k-1})& i>k
            \end{cases}
            \quad\quad
            \phi(B\cap U)=\partial \mathbb{P}_k\times cS
        \end{equation}
        \item $\mathcal{S}$ is recursively a stratified PL manifold
    \end{itemize}
\end{defi}

This definition of stratified PL manifolds is equivalent to that of \cite{friedman2020singular}, when all the base restricted to PL manifolds.

\begin{defi}
    For later use, we denote by $\mathcal{M}^{(k)}:=\pi_0(M^{k}\setminus M^{k-1})$ the connected components of $M^{k}\setminus M^{k-1}$.
    We remember their orientations as PL manifolds and coorientations inside $M$ inherited from the orientation of $M$.
    Its elements are called $k$-dimensional defects. 
    Codimension $0$ defects are called chambers. 
    Codimension $1$ defects are called walls. 
\end{defi}

\begin{defi}
    For a wall $e\in \mathcal{M}^{(n-1)}$, we define its adjacent two chambers $c_-(e),c_+(e)$ as follows.
    By definition for any point $x\in e\cap Int M$, we have a chart $h:\mathcal{M}|_U\to \mathbb{R}^{n-1}\times (-1,1)$.
    Suppose that $h|_{e\cap U}$ is also oriented.
    The stratification of $(-1,1)$ has two chambers $(-1,0),(0,1)$ and a wall $\{0\}$.
    Their pre-images are contained in two chambers of $\mathcal{M}$, denoted by $c_-(e),c_+(e)$. 
    It is well defined, i.e. depends not on the point $x\in e$.
\end{defi}

\begin{defi}
    For a stratified PL manifold $\mathcal{M}=(M;M^i)$, a stratified submanifold is a subset $W\subset M$ such that $(W,\iota:W\hookrightarrow M)$ is a top-dimensional PL submanifold and $W^i:=W\cap M^i$ makes $W$ a stratified PL manifold.
\end{defi}

\subsection{Transversal intersection}

In this subsection we define the notion for a triangulation to be transversal to a stratification. 
Only transversal triangulations are useful in the definition of state sum in TQFT.
We then state several results, for which the detailed proofs are postponed to the appendix.

\begin{defi}
    Suppose that $\mathcal{F}=(F;F^i)$ is a $m$-dimensional stratified PL manifold, $N\subset F$ is a $j$-dimensional PL submanifold. 
    $N$ is said to be transversal to $\mathcal{F}$ if for any $k$ and any $x\in N\cap (F^k\setminus F^{k-1})$, there is an open neighborhood $V\subset \mathbb{P}_k$ of $O_k$, and a $(j+k-m)$-dimensional PL submanifold $Q$ of $V$ containing $O_k$, such that:
    \begin{itemize}
        \item $j+k-m\ge 0$
        \item there is a stratified PL homeomorphism $h: U \to V\times cS$, where $U$ is an open neighborhood of $x$ in $F$ and $\mathcal{S}$ is a stratified PL manifold supported on $S^{m-k-1}$, satisfying
        \begin{equation}
            h(\partial F\cap U) = (V\cap \partial\mathbb{P}_k) \times cS, \quad h(x) = (O_k,v)
        \end{equation}
        and
        \begin{equation}
            h(N\cap U) = Q\times cS, \quad h(\partial N\cap U) = \partial Q \times cS
        \end{equation}
    \end{itemize}
\end{defi}

\begin{defi}
    Suppose that $\mathcal{F}$ is a stratified manifold and $\Delta$ is a combinatorial triangulation.
    $\Delta$ is said to be transversal to $\mathcal{F}$ if every (closed) $j$ simplex is transversal to $\mathcal{F}$ as a $j$-dimensional PL submanifold. 
\end{defi}

\begin{defi}
    A triangulation of $B$ is called full, if for any simplex $\sigma$ all of whose vertices lie in the boundary $\partial B$ is contained in the boundary.
\end{defi}

\begin{thm}\label{transversal triangulation}\label{App: transversal triangulation}
    Suppose that $\mathcal{F}$ is a compact stratified PL manifold, and $\Delta$ is a combinatorial triangulation on $\mathcal{F}$.
    Then there is an ambient isotopy $F_t$ on $F$ such that $F_{1*}(\Delta)$ is transversal to $\mathcal{F}$. 
    Further, if $\partial F$ is a subcomplex of $\Delta$ and $\Delta|_{\partial F}$ is transversal to the stratification restricted on $\partial F$, then the isotopy can be chosen so as to fix the boundary $\partial F$.
\end{thm}

\begin{thm}\label{transversal triangulation extend to bulk}\label{App: transversal triangulation extend to bulk}
    Suppose that $\mathcal{B}=(B;P_B)$ is a $m$-dimensional stratified PL manifold, $F\subset \partial B$ is a $(m-1)$-dimensional PL submanifold such that:
    \begin{itemize}
        \item $B$ is compact and $F$ is compact;
        \item set $M=\overline{\partial B\setminus F}$, then $F\cap M=\partial F = \partial M$;
        \item $F$ admits a stratified collar in $B$, that is there is a stratified PL homeomorphism $c:F\times(-1,0]\to U$ onto an open neighborhood of $F$ in $B$ identifying $F\times\{0\}$ with $F$ and $c(\partial F\times(-1,0])\subset \partial B$.
        We denoted by $P_B|_F$ the stratification $F^a=F\cap B^{a+1}$ and the same for $P_B|_M$.
    \end{itemize}
    Then for any transversal triangulation of $(F,P_B|_F)$, it can be extended to a transversal triangulation of $(B,P_B)$. 
    Further if it is full, then the extension can be full.
\end{thm}

\begin{thm}\label{transversal isotopy step by step}\label{App: transversal isotopy step by step}
    Suppose that $\mathcal{B}=(B;P)$ is a compact stratified PL manifold, 
    $\Delta$ is a transversal triangulation, 
    $F_t$ is an ambient isotopy such that $\Delta$ is also transversal to $F_{1*}(\mathcal{B})$.
    Then there is a sequence of ambient isotopies $g_{j,t}, j=1,...,N$ and vertices $v_j\in \Delta^{(0)}$ such that
    \begin{itemize}
        \item $g_{j,0}=id_B$ and $g_{j,t}$ is the identity restricted on $B\setminus ost_\Delta(v_j)$
        \item set $G_0=id_B$ and $G_j=g_{j,1}G_{j-1}$, then $\Delta$ is transversal to $G_{j*}\mathcal{B}$ for any $j=1,...,N$
        \item $G_N=F_1$
    \end{itemize}
\end{thm}

Next we discuss the concept of refinement of stratifications, which will be used to formulate the gauge relations. 

\begin{defi}\label{refinement}\label{App: refinement}
    Suppose $B$ is a PL manifold and $P=(B;B^i)$ is a stratified PL manifold with support $B$. 
    A refinement $P'\succcurlyeq P$ of stratification $P$ is a new stratified PL manifold structure $P'=(B;B'^i)$ on $B$ such that:
    \begin{itemize}
        \item for any $i$ and any new defect $S'\in B'^{(i)}:=\pi_0(B'^i\setminus B'^{i-1})$, there is an old defect $S\in B^{(j)}:=\pi_0(B^j\setminus B^{j-1})$ such that $S'\subset S$
        \item there is an open neighborhood $C$ of $\partial B$ in $B$ such that $C\cap B'^i=C\cap B^i,\forall i$
    \end{itemize}
\end{defi}

\begin{prop}\label{transversal of refinement}\label{App: transversal of refinement}
    Suppose $P'\succcurlyeq P$ is a refinement of stratification on a PL manifold $B$.
    Then if a triangulation $\Delta$ is transversal to $P'$, it is transversal to $P$.
\end{prop}

\begin{thm}\label{transversal subdivision preparation}\label{App: transversal subdivision preparation}
    Suppose that $\mathcal{F}$ is a stratified PL manifold, $\Delta$ is a transversal triangulation, and $\Gamma$ is a subdivision of $\Delta$ which is combinatorial (but may not be transversal).
    Then there is an isotopy $h_t$ such that $h_{1*}(\mathcal{F})$ is transversal to $\Gamma$ (and therefore also to $\Delta$) and $h_t(\sigma)=\sigma,\forall \sigma\in \Delta$.
\end{thm}

%% file: G-Sn-algebra.tex
\section{G-equivariant hyper disk algebras and hyper sphere algebras}

Throughout the article, $G$ denotes a discrete group. 

\subsection{G-equivariant-$D^n$-algebra}

Firstly we give the definition of G-equivariant $D^n$ algebra.

\begin{defi}
    A G-equivariant empty disk is a stratified PL manifold $\mathcal{S}$ with support $S^{n-1}$, with a labeling $L_{\mathcal{S}}$ when $n\ge 2$ such that 
    \begin{itemize}
        \item each $codim=1$-defect $e\in \mathcal{S}^{(n-2)}$ is labeled by an element $g_e \in G$;
        \item the flatness condition holds at each $codim=2$-defect $f\in \mathcal{S}^{(n-3)}$: 
        \begin{equation}
            g_{e_1}^{\epsilon_1}\dots g_{e_n}^{\epsilon_n}=1
        \end{equation}
        where $e_1\dots e_n$ are the $codim=1$-defect adjacent to $f$ in a clockwise direction, with $\epsilon_i=\pm$ denotes whether the coorientation of $e_i$ is toward the clockwise direction;
        \item upon performing one-point suspension, $\Lambda \mathcal{S}$ extends $L_{\mathcal{S}}$ such that the flatness condition holds at each $codim=2$-defect $f\in (\Lambda\mathcal{S})^{n-2}$;
        \item above defined are the cases when $n\ge 2$. 
        When $n=1$, there is a label $g\in G$ for the whole $S^0$, called the transport.
    \end{itemize}
\end{defi}

\begin{rmk}
    When $n=2$, G-equivariant empty disk has seemingly no flatness constraints on $S^1$, but one flatness constraint on the central point of $S^1$. 
    As a result it still demand that the product of G-labelings on the boundary equals 1. 
    When $n\ge 3$, there is no further flatness constraints in the interior of $\Lambda\mathcal{S}$, because $(c\mathcal{S})^{n-2} = c(\mathcal{S}^{n-3})$ if $n-2\ge 1$, and the flatness conditions inherit from the boundary.
\end{rmk}

\begin{defi}
    We define $T^n(m)$ as the topological space obtained by remove the interior of $m$ disjoint little $D^n$'s inside a big $D^n$. 
    It is well defined topologically.
    Define an $n$-tangle as a stratified PL manifold $\mathcal{T}$ with support $|\mathcal{T}|\sim_{PL} T^n(m)$, such that $T^0\cap \mathring{\mathcal{T}}=\emptyset$.
\end{defi}

\begin{rmk}
    We demand that there is no bare stratification points in the interior of the tangle. 
    This is crucial. 
    Otherwise later in Example~\ref{G-graded associative algebra} we must have a system of natural maps $A_g\to A_{gh}$
    by adding a label-$h$ stratification point inside $D^1_g\setminus \epsilon D^1_g$, which is not presented in G-graded associative algebras. 
\end{rmk}

\begin{defi}
    A G-equivariant $n$-tangle is an $n$-tangle $\mathcal{T}$ together with a labeling $L_{\mathcal{T}}$ such that 
    \begin{itemize}
        \item 
        each $codim=1$-defect $e\in \mathcal{T}^{(n-1)}$ is labeled by an element $g_e \in G$
        \item 
        when $n\ge 2$, the flatness condition holds for each $codim=2$-defect $f\in \mathcal{T}^{(n-2)}$: 
        \begin{equation}
            g_{e_1}^{\epsilon_1}\dots g_{e_n}^{\epsilon_n}=1
        \end{equation}
        where $e_1\dots e_n$ are the $codim=1$-defect adjacent to $f$ in a clockwise direction, with $\epsilon_i=\pm$ denotes whether the coorientation of $e_i$ is toward the clockwise direction.
    \end{itemize}
\end{defi}

\begin{defi}(G-$D^n$-algebra)

    The operad of G-equivariant $D^n$-algebra has objects isotopy equivalence sets of G-equivariant empty disks, in which we fix a representative for each equivalence relation, and operations isotopy equivalence sets of G-equivariant $n$-tangles with boundary identifications, that is,
    \begin{itemize}
        \item 
        a G-equivariant $n$-tangle $\mathcal{T}$ together with stratified homeomorphism $\iota: \mathcal{S}_{out} \bigsqcup_j \mathcal{S}_{in,j} \to\partial \mathcal{T}$ is regarded as an operation from $\otimes_j [\mathcal{S}_{in,j}]$ to $[\mathcal{S}_{out}]$;
        \item 
        the equivalence relation is defined as $(\mathcal{T},\iota: \mathcal{S}_{out} \bigsqcup_j \mathcal{S}_{in,j} \to\partial \mathcal{T})\sim (f(\mathcal{T}),f|_{\partial T}\circ\iota: \mathcal{S}_{out} \bigsqcup_j \mathcal{S}_{in,j} \to f(\partial \mathcal{T}))$
    \end{itemize}
    such that the stratification and the G-labeling of the tangle coincide on the boundary with G-labelings of its input empty disks:
    \begin{itemize}
        \item 
        when $n\ge 2$ we demand that
        \begin{equation}
            L_{\mathcal{T}}|_{\mathcal{S}_i}=L_{\mathcal{S}_i}
        \end{equation}
        or more explicitly:
        \begin{equation}
            g^{(\mathcal{T})}_e = g^{(\mathcal{S}_i)}_{\partial e};
        \end{equation}
        \item 
        when $n=1$ we demand that the transports satisfy
        \begin{equation}
            g^{(\mathcal{S}_{out})} = \prod_j g^{(\mathcal{S}_{in,j})}.
        \end{equation}
    \end{itemize}
    
    And a G-equivariant $D^n$-algebra, or simply a G-$D^n$-algebra, is a representation in $Vect_{\mathbb{C}}$ of this operad.
\end{defi}

Note that these data represent a trivial G-bundle and its chosen trivialization inside each chamber.

\begin{eg}(G-graded associative algebra)\label{G-graded associative algebra}

    In the case $n=1$, a G-equivariant $D^1$-algebra is the same as a G-graded associative algebra $A=\oplus_g A_g$. 
\end{eg}

\begin{rmk}
    The specialness of the dimension $0$ comes from the following fact. 
    When $k\ge 1$, $k$-morphisms come from idempotents of an algebra, which can be summed to identity.
    So one has a resolution of the identity $I=\sum X_i$, which is the bare dimension-$k$ stratification. 
    But when $k=0$, $0$-morphisms do not fit into an invariant summation. 
    So bare dimension-$0$ stratification does not naturally exist.
    This is further elucidated in the reconstruction of G-graded fusion categories in the next section.
\end{rmk}

\subsection{Labeled manifolds}

Before defining the G-$S^n$-algebra, one need to define labeled manifolds to define a hyper sphere functional.

Let $V$ be a G-equivariant $D^n$-algebra. 

In this subsection, we define the $V$-labeled manifold.

A G-labeled manifold is a presentation of a manifold with a G-structure, i.e. a manifold together with a flat G-bundle over itself. 

\begin{defi}
    A G-labeled manifold is 
    \begin{itemize}
        \item 
        a stratified PL manifold $\mathcal{M}$;
        \item 
        an element $g_e\in G$ labeled at each $codim=1$ defect $e\in M^{(\dim M-1)}$;
        \item 
        such that one has the flatness condition
        \begin{equation}
            g_{e_1}^{\epsilon_1}\dots g_{e_n}^{\epsilon_n}=1
        \end{equation}
        at each $codim=2$ defect $f\in M^{(\dim M-2)}$, where $e_1\dots e_n$ are the $codim=1$ defects adjacent to $f$ in a clockwise order, and $\epsilon_i=\pm$ depends on the orientation of $e_i$.
    \end{itemize}
\end{defi}

Note that by flatness condition, the spherical boundary for any point in a G-labeled stratified manifold becomes a G-equivariant empty disk, i.e. an object in the operad of G-equivariant $D^n$-algebra. 

$V$-labeled manifold should have a G-labeled stratification and a labeling of a vector in $V$ at a finite number of positions.

\begin{defi}
    A $V$-labeled manifold $(M,P_M,\{v_p\}_p)$ is the following data:
    \begin{itemize}
        \item A G-labeled $n$ manifold $(M, P_M)$;
        \item A chosen vector $v_p\in V_{\mathcal{S}_p}$, where $\mathcal{S}_p$ is the spherical boundary of $p$, for each $p\in M^0$.
    \end{itemize}
\end{defi}

\begin{rmk}
    We demand that any 0-dimensional stratification point is a position point, i.e. there are no bare stratification points. 
    This is consistent with the definition of the G-$D^n$ operad.
\end{rmk}

A V-labeled manifold $\mathcal{M}$ with support $|\mathcal{M}|\sim_{PL} D^n$ can be viewed as a vector in $V_{\partial\mathcal{M},-}$, in the natural way of the operations in the $D^n$ operad.
In this way one can view a local disk region inside a V-labeled manifold as a single label. 

\begin{defi}(Null vectors)

    Suppose $\mathcal{M}=\sum_i a_i\mathcal{M}_i$ is a formal linear combination of V-labeled manifolds with same support $|\mathcal{M}_i|\sim_{PL}M$. 
    Suppose that there exists a local n disk region $D^n\sim_{PL}U\subset M$ with $\mathcal{M}_i$ and $\partial U$ intersects transversely and $\mathcal{M}_i|_{M\setminus U}=\mathcal{M}_j|_{M\setminus U}$.
    It is called a null vector if $\sum_i a_i \mathcal{M}_i|_{U}\in V_{\mathcal{M}|_{\partial U},-}$ is equal to $0$.
\end{defi}

\begin{defi}(Condensation space)

    For a PL $n$ manifold $M$ with a G-labeled $(n-1)$-manifold $\mathcal{S}$ supported on $\partial M$, write 
    \begin{equation}
        \widetilde{M}_{\mathcal{S}}(V) := Free_\mathbb{C}\{ V\text{-labeled manifold $\mathcal{M}$ with support $M$ and boundary $\mathcal{S}$} \} / \mathbb{C}\{ \text{null vectors} \}
    \end{equation}
    When $\partial M=\emptyset$, write $\widetilde{M}(V)=\widetilde{M}_{\emptyset}(V)$ for short.
\end{defi}

The non-degeneracy condition (Eq.~\ref{non-degenracy}) guaranties that $\widetilde{D^{n}}_{\mathcal{S}}(V)=V_{\mathcal{S},-}$ for any $|\mathcal{S}|\sim S^{n-1}$.

\begin{defi}(Notation/Terminology)\label{Notation: labeled at}

    Suppose $x\in V_{\mathcal{S},-}$ with $|\mathcal{S}|\sim_{PL}S^{n-1}$, and $\mathcal{M}$ is a G-labeled $n$-manifold with a local disk $\mathcal{D}^n\overset{f}{\cong} \mathcal{U}\subset \mathcal{M}$ with $\partial \mathcal{D}^n=\mathcal{S}$.
    We define $\mathcal{M}$ with $x$ labeled at $\mathcal{U}$ to be the V-labeled $n$-manifold obtained by $\mathcal{M}$ changing $\mathcal{U}$ into $\Lambda(\partial(\mathcal{U}))$ and add a position point the center of it and its label $x$.
\end{defi}

\begin{rmk}
    Using this ``labeled at" concept, one may rephrase the definition of  condensation space in the way that it is mod out by the equivalence relation that for the same situation as above,
    \begin{equation}
    \sum_i a_i \mathcal{M}_i\sim
    \text{``$\mathcal{M}$ with $\sum_i a_i \mathcal{M}_i|_U\in V_{\mathcal{M}|_{\partial U},-}$ labeled at $U$"}
    \end{equation}
\end{rmk}

\subsection{G-equivariant $S^n$-algebra}

Now we can give the definition of G-equivariant $S^n$-algebras. 
That is, a G-equivariant $D^n$-algebra together with a hyper sphere functional.

\begin{defi}(Notation)\label{Notation: VS+-}

    For a G-$D^n$-algebra $V$, denote by $V_{\mathcal{S},-}:=V(\mathcal{S})$ and $V_{\mathcal{S},+}:=V(\mathcal{S}^*)$ the vector space of the G-equivariant empty disk $\mathcal{S}$, where $\mathcal{S}^*$ stands for the orientation reversal of $\mathcal{S}$.
\end{defi}

\begin{defi}
    Define $\widehat{S}^n(V)\subset \widetilde{S}^n(V)$ to contain those which can be extended into a G-labeled manifold with support $D^{n+1}$.
\end{defi}

For $n\ge 2$, $\widehat{S}^n(V)= \widetilde{S}^n(V)$.
For $n=1$, $\widehat{S}^1(V)$ contains those with trivial holonomy.

\begin{defi}(Notation)\label{Notation: cup_S}

    Suppose that $|\mathcal{S}|\sim_{PL}S^{n-1}$ is a G-equivariant empty disk and $(a,b)\in V_{\mathcal{S},-}\times V_{\mathcal{S},+}$.
    Define $a\cup_{\mathcal{S}}b\in \widetilde{S}^n(V)$ as the V-labeled manifold with support $S^n$ obtained by gluing the boundaries of two G-equivariant empty disks in the canonical way, where $\widetilde{S}^n(V)$ is the condensation space defined before.
    It is obvious that $a\cup_{\mathcal{S}}b\in \widehat{S}^n(V)$.
\end{defi}

\begin{defi}(G-$S^n$-algebra)

    A G-equivariant $S^n$-algebra, or simply a G-$S^n$-algebra, is
    \begin{itemize}
        \item 
        a G-equivariant $D^n$-algebra $V$;
        \item 
        a linear functional 
        \begin{equation}
            Z: \widehat{S}^n(V)\to \mathbb{C};
        \end{equation}
        \item 
        $Z$ induces non-degenerate bilinear form for any $|\mathcal{S}|\sim S^{n-1}$:
        \begin{equation}\label{non-degenracy}
        \begin{aligned}
            &V_{\mathcal{S},-}\times V_{\mathcal{S},+}\to\mathbb{C}
            \\
            &a,b\mapsto \langle a,b\rangle:=Z(a\cup_{\mathcal{S}}b).
        \end{aligned}
        \end{equation}
    \end{itemize}
\end{defi}

\begin{eg}

    In the case $n=1$, a G-equivariant $S^1$-algebra is the same as a G-graded associative algebra $A=\oplus_g A_g$ with a trace $A_0\to \mathbb{C}$, satisfying non-degeneracy properties.
    That is, it satisfies trace property:
    \begin{equation}
        Tr(xy)=Tr(yx), \quad\forall x\in A_g, y\in A_{g^{-1}}
    \end{equation}
    and non-degeneracy property:
    \begin{equation}
        A_g\times A_{g^{-1}}\to \mathbb{C},\quad (x,y)\mapsto Tr(xy)
    \end{equation}
    is a non-degenerate bilinear form. 
\end{eg}

\begin{eg}(G-graded biangular algebra)

    G-equivariant $S^1$-algebra cover the biangular G-algebras in \cite{turaev1999homotopy}.
    Choose 
    \begin{equation}
        Z(l)=\operatorname{Tr}(\mu_l|_{L_1}:L_1\to L_1)
    \end{equation}
    This choice of $Tr$ makes sure of the trace property.
    Condition (ii) says that $\langle a,b\rangle: L_g\times L_{g^{-1}}\to \mathbb{C}$ is nondegenerate.
    Condition (i), the most mysterious axiom, implies that this satisfies (MI).
\end{eg}

\subsection{Reflection positivity (RP)}

\begin{defi}
    A reflection on a G-$S^n$-algebra $\mathcal{V}=(V,Z)$ is a set of order-two conjugate-linear maps $\theta_\mathcal{S}: V_{\mathcal{S},\pm}\to V_{\mathcal{S},\mp}$ for all G-equivariant empty disks $|\mathcal{S}|=S^{n-1}$, such that they are compatible with tangle operations
    \begin{equation}
        \theta\mathcal{T}(\otimes v_i)
        =\overline{\mathcal{T}}(\otimes \theta(v_i))
    \end{equation}
    for all tangles $\mathcal{T}$, where $\overline{\mathcal{T}}$ denotes its orientation reversal.
\end{defi}

\begin{defi}(RP)

    A G-$S^n$-algebra $\mathcal{V}=(V,Z)$ is called reflection positive (RP) with respect to a reflection $(\theta_{\mathcal{S},\pm}:V_{\mathcal{S},\pm}\to V_{\mathcal{S},\mp})_{\mathcal{S}}$ if
    \begin{itemize}
        \item 
        It is Hermitian in the sense that
        \begin{equation}
            Z(v_-\cup_{\mathcal{S}} v_+)^* = Z(\theta(v_+)\cup \theta(v_-))
            , \quad\forall |\mathcal{S}|=S^{n-1},\forall v_\pm\in V_{\mathcal{S},\pm}
        \end{equation}
        \item
        and is positive in the sense that
        \begin{equation}
        Z(v\cup_{\mathcal{S}} \theta(v)) \ge 0, \quad\forall |\mathcal{S}|=S^{n-1},\forall v\in V_{\mathcal{S},\pm}
    \end{equation}
    \end{itemize}
    where we have used notation Def.~\ref{Notation: cup_S}.
\end{defi}

\begin{rmk}
    Later on we will abuse the language and say that an $S^n$-algebra $\mathcal{V}=(V,Z)$ is reflection positive and make the reflection structure $(\theta_{\mathcal{S},\pm}:V_{\mathcal{S},\pm}\to V_{\mathcal{S},\mp})_\mathcal{S}$ implicit.
\end{rmk}

%% file: G-category.tex
\section{Emergent G-graded spherical n-category}

We study the higher representation theory of the G-$S^n$-algebra and construct an $n$-category. 
We consider such an $n$-category as a G-graded spherical $n$-category.
When $\mathcal{V}$ is (RP), then it is considered as a G-graded unitary $n$-category.

The basic procedure is parallel to \cite{liu2024functional}, replacing ordinary $S^n$-algebra by its G-version, and a difference in manner that we start from a (non-degenerate) G-$S^n$-algebra instead of a functional.

\subsection{Idempotent completion}

Let $\mathcal{V}=(V,Z)$ be a G-$S^n$-algebra. 

For a G-labeled manifold $\mathcal{D}^{n-1}$ with support $D^{n-1}$, we consider the vector space
\begin{equation}
    V_{\mathcal{S},-}, \quad \mathcal{S}:=\partial(\mathcal{D}^{n-1}\times D^1)
\end{equation}
This vector space forms an algebra in the following sense.

\begin{defi}
    For two vectors $x,y\in V_{\mathcal{S},-}$, define the multiplication $xy\in V_{\mathcal{S},-}$ as follows.
    We label $x$ at $\mathcal{D}^{n-1}\times[-1,0]$ and $y$ at $\mathcal{D}^{n-1}\times[0,1]$ after rescaling their coordinates in the obvious way, where we have used terminology in Def.~\ref{Notation: labeled at}.
    The result is defined as their multiplication $xy\in V_{\mathcal{S},-}$.
\end{defi}

\begin{defi}
    For a vector $x\in V_{\mathcal{S},-}$, define its trace as
    \begin{equation}
        Tr(x):=Z(x\cup_{\mathcal{S}} (\mathcal{D}^{n-1}\times D^1))
    \end{equation}
    where we have used the notation in Def.~\ref{Notation: cup_S}.
\end{defi}

\begin{prop}
    For any G-labeled manifold $|\mathcal{D}^{n-1}|=D^{n-1}$, the above multiplication and trace are well defined on $V_{\mathcal{S},-}$. 
    In this way $V_{\mathcal{S},-}$ forms an associative algebra with trace, denoted by $A(\mathcal{D}^{n-1}\times D^1)$.
\end{prop}

\begin{proof}
    Isotopy equivalences are mod out in the definition of the G-$D^n$-operad.
    So, the multiplication is associative.
    $Tr(xy)=Tr(yx)$ is also caused by isotopy invariance.
\end{proof}

\begin{prop}
    For $2\le k\le n$ and any G-labeled manifold $|\mathcal{D}^{n-k}|=D^{n-k}$, consider $V_{\mathcal{S},-}$, where
    \begin{equation}
        \mathcal{S}:=\partial(\mathcal{D}^{n-k}\times D^k)
    \end{equation}
    We consider $\mathcal{D}^{n-k}\times D^{k-1}=:\mathcal{D}^{n-1}$, then by the above $V_{\mathcal{S},-}$ forms a commutative associative algebra with trace, denoted by $A(\mathcal{D}^{n-k}\times D^k)$.
\end{prop}

\begin{proof}
    The commutativity for $k\ge 2$ follows from switching two smaller $k$-disks in $D^k$ by isotopy.
\end{proof}

Having these associative algebras, let's define the $n$-category emerged from idempotent completion. 

\begin{defi}
    For $0\le k\le n-1$, and a G-labeled manifold $|\mathcal{D}^k|=D^k$, we call a minimal idempotent $\alpha$ of $A(\mathcal{D}^k\times D^{n-k})$ an irreducible $k$-morphism of type $tp(\alpha)=\mathcal{D}^k$ and link boundary $lb(\alpha)=\partial\mathcal{D}^k$.
    A $k$-morphism is a formal direct sum of irreducible $k$-morphisms. 
\end{defi}

\begin{defi}
    For an irreducible $k$-morphism $\alpha$ with $tp(\alpha)=\mathcal{D}^k$, we define its quantum dimension as $Tr(\alpha)$, where the trace is the trace in the algebra $A(\mathcal{D}^k\times D^{n-k})$.
\end{defi}


\begin{defi}
    If $0\le k\le n-2$, suppose $\alpha_\pm$ are irreducible $k$-morphisms of type $\mathcal{D}_\pm$ respectively, and they have the same boundary $\mathcal{S}$. 
    Suppose $|\mathcal{D}|\sim D^{k+1}$ with boundary $(\mathcal{D}_-\times\{-1\}\cup \mathcal{D}_+\times\{+1\}\cup \mathcal{S}\times D^1)$.
    Suppose $\beta$ is an idempotent of $A(\mathcal{D}\times D^{n-k-1})$.
    $\beta$ is said to be an $\alpha_--\alpha_+$ bimodule, if $P(\alpha_-)\beta Q(\alpha_+)=\beta$, where $P(\alpha_-),Q(\alpha_+)\in A(\mathcal{D}\times D^{n-k-1})$ are $\mathcal{D}\times D^{n-k-1}$ labeled by $\alpha_-$ at $\mathcal{D}_-\times[-1,-1+\epsilon]\times D^{n-k-1}$ and $\alpha_+$ at $\mathcal{D}_+\times[1-\epsilon,1]\times D^{n-k-1}$ respectively.

    If $k=n-1$, we adjust the above definition by requiring $\beta$ to be an element of $V_{\partial\mathcal{D},-}$ and the multiplication $P(\alpha_-)\beta Q(\alpha_+)$ is regard as natural actions of $P(\alpha_-)\in V_{\mathcal{D}_-\cup_\mathcal{S} \mathcal{D}_-,-}$ and $Q(\alpha_+)\in V_{\mathcal{D}_+\cup_\mathcal{S} \mathcal{D}_+,-}$ on $V_{\partial\mathcal{D},-}$ by concatenating from upper and lower boundary.
\end{defi}

\begin{defi}
    Suppose $\alpha_0,\alpha_1,\alpha_2$ are irreducible $k$-morphisms and $\beta_0$ is an $\alpha_0-\alpha_1$ bimodule and $\beta_1$ is an $\alpha_1-\alpha_2$ bimodule. 
    Bimodule fusion $\beta_0\otimes_{\alpha_1}\beta_1$ is defined as their concatenating in the $(k+1)$-th coordinate.
    The resulting $\beta_0\otimes_{\alpha_1}\beta_1$ is an $\alpha_0-\alpha_2$ bimodule.
\end{defi}

\subsection{Annular equivalence, (SS) and (CF)}

Now we define the notion of annular equivalence. 
This plays the role of Morita equivalence.

Suppose $|\mathcal{S}|=S^{k-1}$, $|\mathcal{D}_i|=D^k$ and $\partial(\mathcal{D}_i)=\mathcal{S}$ for $i=0,1$. 
We consider $V_{\partial(\mathcal{D}_i\times D^{n-k})}$.
Abbreviate by
\begin{equation}
    V(\mathcal{D}_1,\mathcal{D}_0)=\widetilde{D^k\times(D^{n-k}\setminus\frac{1}{2}D^{n-k})}
    _{
    \mathcal{S}\times (D^{n-k}\setminus\frac{1}{2}D^{n-k}) \cup
    \mathcal{D}_0\times S^{n-k-1}
    \cup \mathcal{D}_1\times \frac{1}{2}S^{n-k-1}
    }
    (V)
\end{equation}

\begin{defi}
    Define an action
    \begin{equation}
        V(\mathcal{D}_1,\mathcal{D}_0)\times V_{\partial(\mathcal{D}_1\times D^{n-k})}
        \to
        V_{\partial(\mathcal{D}_0\times D^{n-k})}
    \end{equation}
    For a $T\in V(\mathcal{D}_1,\mathcal{D}_0)$ and $v\in V_{\partial(\mathcal{D}_1\times D^{n-k})}$, we label $v$ at $\mathcal{D}_1\times \frac{1}{2}D^{n-k}$ and join it inside $T$.
    The result is defined as $Tx\in V_{\partial(\mathcal{D}_0\times D^{n-k})}$.
\end{defi}

\begin{defi}
    For irreducible $k$-morphisms $\alpha_0,\alpha_1$ of type $tp(\alpha_i)=\mathcal{D}_i$ and common link boundary $\mathcal{S}^{k-1}$, we call them annular equivalent, if there are $T_1\in V(\mathcal{D}_1,\mathcal{D}_0)$ and $T_2\in V(\mathcal{D}_0,\mathcal{D}_1)$, such that
    \begin{equation}
        T_2\alpha_0=\alpha_1, \quad T_1\alpha_1=\alpha_0
    \end{equation}
\end{defi}

It is indeed an equivalence relation, because one has composition of annular vectors by annular concatenating, symmetricity by definition, and reflexivity by trivial annular vector.

\begin{rmk}
    Annular equivalence relation is not additive. 
\end{rmk}

Now fix a representative set of irreducible $k$-morphisms for each $k=0,1,\dots,n-1$.

\begin{defi}
    The global dimension of an irreducible $(n-1)$-morphism is defined to be $1$. 
    Suppose $n\ge 2$.
    For an irreducible $k$-morphism $\alpha$ of type $tp(\alpha)=\mathcal{D}^k$, the global dimension is defined inductively for $k=n-2\dots,1$ as
    \begin{equation}
        \mu(\alpha) = \sum_{\beta}\frac{Tr(\beta)^2}{\mu(\beta)}
    \end{equation}
    summing over our chosen representatives of irreducible $\alpha-\alpha$ bimodules $\beta$. 
    For $k=0$ the global dimension is defined as
    \begin{equation}
        \mu(\alpha) = \sum_{gr(\beta)=e}\frac{Tr(\beta)^2}{\mu(\beta)}
    \end{equation}
    summing over our chosen representatives of irreducible $\alpha-\alpha$ bimodules $\beta$ with trivial grading as a $1$-morphism. 
\end{defi}

For this to be well-defined, we need to ensure $\mu(\beta)\ne 0$ inductively, and that the summations are finite sums.
We impose these constraints as two axioms. 

\begin{defi}(SS)

    For a G-$S^n$-algebra $\mathcal{V}=(V,Z)$, it is called strong semisimple if $A(\mathcal{D}^{n-1}\times D^1)$ is a semisimple algebra for any $|\mathcal{D}^{n-1}|=D^{n-1}$ and the global dimension of any irreducible $k$-morphism is nonzero, for any $0\le k\le n-1$.
\end{defi}

\begin{eg}
    When G is trivial, a spherical fusion category can be viewed as a $2$-category with a unique $0$-morphism. 
    The global dimension of the $0$-morphism is just the global dimension of the category.
    It is nonzero proven in \cite{etingof2005fusion}.
    So any spherical fusion category is strong semisimple (SS). 
\end{eg}

\begin{defi}(CF)

    For a G-$S^n$-algebra $\mathcal{V}=(V,Z)$,
    \begin{itemize}
    \item
    it is called $n$-finite if for any G-equivariant empty disks $|\mathcal{S}|=S^{n-1}$, the vector space $V_{\mathcal{S},-}$ is finite dimensional. 
    \item
    for $2\le k\le n-1$, it is called $k$-finite, if for any G-labeled $|\mathcal{S}^{k-1}|=S^{k-1}$, there are finitely many annular equivalence classes of irreducible $k$-morphisms $\beta$ with link boundary $lb(\beta)=\mathcal{S}^{k-1}$. 
    \item
    it is called $1$-finite, if for any fixed $g\in G$, there are finitely many annular equivalence classes of irreducible $1$-morphisms with link boundary $S^0$ whose transport is $g$;
    \item 
    it is called $0$-finite, if there are finitely many annular equivalence classes of irreducible $0$-morphisms;
    \end{itemize}
    it is called compete finite (CF), if it is $k$-finite for $0\le k\le n$.
\end{defi}

Complete finiteness (CF) is crucial for the summations to be finite sums. 
It has been used in the definition of global dimensions, and will be used in the $\mathcal{S}^{k-1}$-relations.


\subsection{G-graded fusion categories}

When $n=1$, G-$S^1$ algebras are G-graded associative algebras with trace, as explained in Example~\ref{G-graded associative algebra}.

When $n=2$, suppose $|G|$ is finite and we are given a G-$S^2$-algebra $\mathcal{V}$ which is (SS), (CF), and suppose that the $0$-morphism is unique.
Then the emergent $2$-category is a G-graded spherical fusion category, as follows.

First let's see the 1-morphisms. 
Denote by $D^1_{g_1,\dots,g_m}$ the G-labeled $D^1$ with labels $g_1,\dots,g_n$ from left to right, and $A_{g_1,\dots,g_n}:=A(D^1_{g_1,\dots,g_n}\times D^1):=V_{\partial(D^1_{g_1,\dots,g_n}\times D^1),-}$ the associative algebra.
The minimal idempotents in 
\begin{equation}
    \hat{A}_g:=\bigoplus_{g_1,\dots,g_n=g}A_{g_1,\dots,g_n}
\end{equation}
cannot be annular equivalent to those in $\hat{A}_h$ for $h\ne g$, because the G-empty disks used in the annular equivalence relation are $V(D^1_{g_1,\dots,g_n}\times\{\frac{1}{2}\}\cup D^1_{h_1,\dots,h_m}\times\{1\} \cup S^0\times[\frac{1}{2},1])$ and its mirror, and it should have trivial total holonomy, which implies $g=h$.
We define those to have gradings $gr=g$.
Tensor product of 1-morphisms is the horizontal concatenation of $D^1_{g_1,\dots,g_n}\times D^1$ and $D^1_{h_1,\dots,h_m}\times D^1$, which is multiplicative in degrees.

Let's see the 2-morphisms. 
For 1-mirphisms $X\in A_{g_1,\dots,g_n}\subset \hat{A}_g$ and $Y\in A_{h_1,\dots,h_m}\subset \hat{A}_h$, the 2-morphisms $X\to Y$ should live in 
$V(D^1_{g_1,\dots,g_n}\times\{0\}\cup D^1_{h_1,\dots,h_m}\times\{1\} \cup S^0\times[0,1])$, so by the same reason it is nonzero only when $g=h$.

So the emergent spherical 2-category is a G-graded spherical fusion category when the 0-morphism is unique.

On the contrary, G-$S^2$ algebras can recover all the G-graded strict spherical fusion categories. 
Suppose $\mathcal{C}=\mathcal{C}_g$ is a G-graded strict spherical fusion category. 
One can reconstruct $\mathcal{C}$ by the following G-$S^2$-algebra. 
Denote by
\begin{equation}
    \Omega^0_g := \bigoplus_{gr(X)=g} X
\end{equation}
where we sum over irreducible objects $X$ in $\mathcal{C}_g$.
The vector space assigned to the G-empty disk $S^1_{g_1,\dots,g_n}$ labeled by $g_1,\dots,g_n$ on its $n$ stratification-points in a counterclockwise order with $\Pi_{i}g_i=e$ is 
\begin{equation}
    V(S^1_{g_1,\dots,g_n}):=\operatorname{Hom}(\mathbbm{1}, \bigotimes_{i}\Omega^0_{g_i})
\end{equation}
The tangle operation is just defined as the graphical calculus in $\mathcal{C}$. 
The only new thing is the bare $g$-string in the tangles, which is defined as $\Omega^0_g$. 
Tangle compositions are compositions in the category $\mathcal{C}$ together with the identity of composition of $1$-morphisms $\Omega^0_g \underset{1}{\times} X = X = X \underset{1}{\times} \Omega^0_g$ for any $X\in \mathcal{C}_g$. 
Partition function is defined by the canonical isomorphism $\operatorname{End}(\mathbbm{1})\cong \mathbb{C}$. 
Sphericality of $\mathcal{C}$ ensures isotopy invariance.
So this is a well defined G-$S^2$-algebra. 
It reconstructs $\mathcal{C}$ via the idempotent completion procedure described in this section as follows.

Let's find the 1-morphisms. 
Then minimal idempotents in
\begin{equation}
    \bigoplus_{g_1,\dots,g_n=g}A_{g_1,\dots,g_n}
\end{equation}
all have annular representatives in $A_g$, which is $\iota_X\circ p_{X}\in V(\partial(D^1_{g_1,\dots,g_n}\times D^1))=\operatorname{Hom}(\Omega^0_g, \Omega^0_g)$ for an irr. object $X\in \mathcal{C}_g$ in the original category.
And the 2-morphisms from $X\in\mathcal{C}_g$ to $Y\in \mathcal{C}_h$ are just the bimodules between $\iota_X\circ p_{X}$ and $\iota_Y\circ p_{Y}$ in $V(D^1_g\times\{-1/2\}\cup D^1_h\times\{+1/2\}\cup S^0\times[-1/2,+1/2])$, which is one dimensional when $g=h$ and $X\cong Y$ and zero otherwise.
As the result, one recovers the original G-graded spherical fusion category $\mathcal{C}$.

Any G-graded spherical fusion category is (SS) by \cite{etingof2005fusion}.

\subsection{Morita Invariance (MI)}

This is the new phenomenon in the G-equivariant setting.

\begin{defi}(MI)

    Given a G-$S^n$-algebra $\mathcal{V}=(V,Z)$ and its emergent $n$-category $\mathcal{C}$, $\mathcal{V}$ is said to be Morita invariant (MI) if for any irreducible $0$-morphism $\alpha$ in $\mathcal{C}$, 
    \begin{equation}
        \sum_{gr(\beta)=g}\frac{Tr(\beta)^2}{\mu(\beta)} = \mu(\alpha), \quad\forall g\in G
    \end{equation}
    summing over irreducible $\alpha-\alpha$ bimodules with grading $g$.
    That is, one can change the trivial grading sector to arbitrary sector, without changing the global dimension. 
\end{defi}

\begin{prop}
    Suppose $|G|$ is finite.
    In a G-$S^2$-algebra $\mathcal{V}$ is reconstructed by a strict faithful G-graded spherical fusion category, then $\mathcal{V}$ is (MI).
\end{prop}
\begin{proof}
    This is a standard fact in the theory of G-graded fusion category.
\end{proof}

\subsection{A basis}

Let $\mathcal{V}=(V,Z)$ be a G-$S^n$ algebra.
We will derive a property that will be used in the proof of the Adjacent Move.

\begin{defi}(Notation)\label{Notation: labeled at (2)}

    Suppose $0\le k\le n-1$.
    Write $\mathcal{D}^k(\beta)\times S^{n-k}$ to denote the $\mathcal{V}$-labeled manifold $\mathcal{D}^k\times S^{n-k}$ constructed by setting an irreducible $k$-morphism $\beta$ of type $\mathcal{D}^k$ labeled at $ \mathcal{D}^k \times \{+1\}\times\epsilon D^{n-k}$ (where we think of $\{+1\}\times\epsilon D^{n-k} \subset S^{n-k}$). 
    Similarly for $\mathcal{D}^k(\beta)\times \epsilon S^{n-k}$.
    We used the terminology ``labeled at" in Def.~\ref{Notation: labeled at}. 
\end{defi}

\begin{lem}\label{isotopy WLOG}\label{App: isotopy WLOG}
    Suppose that $l$ is a V-labeled manifold with support $D^k\times S^{n-k}$ and boundary $\mathcal{S}^{k-1}\times S^{n-k}$.
    Then there is an ambient isotopy $F_t$ fixing $S^{k-1}\times S^{n-k}$ such that there is a local open disk $C\subset S^{n-k}$, $F_{1*}(l)|_{D^k\times C}=\mathcal{D}^k\times C$ for some $\partial \mathcal{D}^k=\mathcal{S}^{k-1}$.
\end{lem}
\begin{proof}
    See the appendix.
\end{proof}

\begin{prop}(A basis)\label{A basis}

    When $0\le k\le n-1$ and suppose $A(\mathcal{D}^k\times D^{n-k})$ is finite dimensional semisimple for any $\mathcal{D}^k$ with link boundary $\mathcal{S}^{k-1}$, then:
    \begin{itemize}
        \item (1). the vector space $\widetilde{D^k\times S^{n-k}}_{\mathcal{S}^{k-1}\times S^{n-k}}(\mathcal{V})$ is spanned by $\mathcal{D}^k(\beta)\times S^{n-k}$ for irr. $k$-morphisms $\beta$ with $lb(\beta)=\mathcal{S}^{k-1}$;
        \item (2). any V-labeled manifold with support $D^k\times S^{n-k}$ with boundary V-labeling ${\mathcal{S}^{k-1}\times S^{n-k}}$ (i.e. a ``pure vector" in $D^k\times S^{n-k}_{\mathcal{S}^{k-1}\times S^{n-k}}(\mathcal{V})$) and a local part $\mathcal{D}^k\times\{+1\} \times 2\epsilon D^{n-k}$ can be written as a linear combination of $\mathcal{D}^k(\beta)\times S^{n-k}$'s for some $\beta$ and a null vector in which the local null relation is disjoint from $D^k\times\{+1\}\times \epsilon D^{n-k}$;
        \item (3). if $\beta'$ and $\beta$ are annular equivalent, then $\mathcal{D}^k_2(\beta)\times S^{n-k}$ and $\mathcal{D}_3^k(\beta')\times S^{n-k}$ are scalar multiples of each other;
        \item (4). any V-labeled manifold with support $D^k\times S^{n-k}$ with boundary V-labeling ${\mathcal{S}^{k-1}\times S^{n-k}}$ (i.e. a ``pure vector" in $D^k\times S^{n-k}_{\mathcal{S}^{k-1}\times S^{n-k}}(\mathcal{V})$) and a local part $\mathcal{D}^k\times \{+1\}\times 2\epsilon D^{n-k}$ can be written as a linear combination of $\mathcal{D}^k_2(\beta)\times S^{n-k}$'s for $\beta$'s in the representative set and null vectors in which the local null relation terms are disjoint from either $D^k\times\{+1\}\times \epsilon D^{n-k}$ or $D^k\times\{-1\}\times \epsilon D^{n-k}$.
    \end{itemize}
    
\end{prop}

\begin{proof}
    (1). 
    Take $l=\sum_j l_j\in \widetilde{D^k\times S^{n-k}}_{\mathcal{S}^{k-1}\times S^{n-k}}(\mathcal{V})$ where each $l_j$ is a single V-labeled manifold.
    For each $l_j$ let's prove that it is equivalent to a linear combination of standard vectors $\mathcal{D}^k(\beta)\times S^{n-k}$.
    By Lem.~\ref{isotopy WLOG} one can find an isotopy $F_j$ for each $j$, a point $p\in S^{n-k}$ and $C$ its local disk neighborhood in $S^{n-k}$ such that $F_{j*}(l_j)|_{D^k \times C}=\mathcal{D}_j^k\times C$ for a G-labeled manifold $\partial \mathcal{D}^k_j=\mathcal{S}^{k-1}$ for each $j$. 
    For later convenience let's denote it also as $l_j$.
    
     $l_j|_{D^k\times(S^{n-k}\setminus C)}$ is a vector in $A(\mathcal{D}^k_j\times D^{n-k})$ if we identify $S^{n-k}\setminus C$ with $D^{n-k}$.
     Use the decomposition $A=\sum\mathbb{C}\beta$ summing over minimal idempotents.
    The region of the summation is disjoint from $D^k\times C$.
    So $l_j$ can be written as a linear combination of $\mathcal{D}^k_j(\beta)\times S^{n-k}$ and a null relation inside $D^k\times (S^{n-k}\setminus C)$.
    So these span the vector space $\widetilde{D^k\times S^{n-k}}_{\mathcal{S}^{k-1}\times S^{n-k}}(\mathcal{V})$.

    (2). By the proof of (1) for a single V-labeled manifold. 
    $C=\{+1\}\times \frac{3}{2}\epsilon \mathring{D^{n-k}}\subset S^{n-k}$ in this case.

    (3). Since $\beta'=T(\beta)$ for an annular action $T$, one can rewrite $\mathcal{D}^k_3(\beta')\times S^{n-k}$ as a V-labeled manifold with support $D^k\times S^{n-k}$ and a local part $T(\beta )$.
    The local null relation term in this step is supported on $D^k\times\{+1\}\times \epsilon D^{n-k}$, avoiding the negative side.
    Now one can make a copy of $\beta$: $T(\beta)=T(\beta^2)$ and isotope $\beta$ together with lower half of the annular crossing $D^k\times \{+1\}\times \epsilon D^{n-k}$.
    The local null relation in this step is supported on $D^k\times \{+1\}\times 2\epsilon D^{n-k}$, avoiding the negative side.
    Now the region $D^k\times (S^{n-k}\setminus C)$ is $\beta a \beta=c\beta$ by semisimplicity of $A(\mathcal{D}^k\times D^{n-k})$ and the local null relation is supported on $D^k\times(S^{n-k}\setminus C)$, avoiding the positive side.

    (4). By the proof of (1) and (3).
\end{proof}

%% file: G-Alterfold_cobordism.tex
\section{G-Alterfold cobordism category}

Let $\mathcal{V}=(V,Z)$ be a G-equivariant $S^n$-algebra. 

\subsection{Bulk vectors}

In this subsection we are going to define the ``bulk vectors", to which we are going to define the partition function in the next section. 

Now we are ready to define the bulk vectors, to which we are going to define the partition function. 

\begin{defi}
    We call the triple $(B,\mathcal{M}, P_B)$ a bulk vector:
    \begin{itemize}
        \item $(B,P_B)$ is a compact G-labeled $(n+1)$ manifold
        \item $\mathcal{M}=(M=\partial B,P_B|_{M},\{v_p\}_p)$ is a $\mathcal{V}$-labeled $n$ manifold with support $\partial B$, called the Space Boundary
    \end{itemize}
\end{defi}

In order for the G-labeling to really represent a G-structure, we need to introduce equivalence relations.
The crucial point here is that gauge transformations are required to be trivial on the boundary, for the following two different reasons. 
Fix the space boundary is to let local relations(moves) to be well defined. 
Fix the time boundary is to let the gluing to be well defined. 

Now we introduce the gauge transformations of bulk vectors. 
Gauge equivalence classes will correspond to manifolds with G-bundles, under suitable identifications of the bundle on the boundary.

\begin{defi}(Gauge)\label{Gauge}

    Two $(n+1)$-dimensional G-labeled manifolds $(B,P_1)$ and $(B,P_2)$ with $P_1|_{\partial B}=P_2|_{\partial B}$ are called gauge equivalent, if $P_1$ and $P_2$ can be identified by a sequence of the followings and their inverses:
    \begin{itemize}
        \item 
        PL ambient isotopy (relative to the boundary): 
        given a PL ambient isotopy $F_t$ relative to the boundary (i.e. $F_t|_{\partial B}=id_{\partial B}$), $P_2=F_{1*}(P_1)$
        \item 
        refinement (relative to the boundary):
        $P_2\succcurlyeq P_1$ is a refinement of stratification (Def.~\ref{refinement}), and the G-label of any new wall $e'\in P_2^{(n)}$ is:
        \begin{equation}
            \tilde{g}(e')=\begin{cases}
                g(e), \quad \text{if $e'\subset e$ for a wall $e\in P_1^{(n)}$}
                \\
                1,\quad \text{if $e'\subset c$ for a chamber $c\in P_1^{(n+1)}$}
            \end{cases}
        \end{equation}
        where we regard the coorientation to inherit from the old ones;
        \item 
        chamber gauge transformation (relative to the boundary): assign each $(n+1)$-chamber $c\in P^{(n+1)}$ an element $a_c\in G$, and define:
        \begin{equation}
            \tilde{g}(e)=a(c_-(e)) g(e) a(c_+(e))^{-1}, \quad\forall e\in P^{(n)}
        \end{equation}
        where $c_-(e),c_+(e)\in P^{(n+1)}$ denotes the adjacent $(n+1)$-chamber of $e$. 
        We demand that
        \begin{equation}
            a(c) = 1,\quad \forall \overline{c}\cap\partial B\ne \emptyset
        \end{equation}
    \end{itemize}
\end{defi}

\begin{defi}

    Two bulk vectors $(B,\mathcal{M}, P_1)$ and $(B,\mathcal{M}, P_2)$ are called gauge equivalent, if $(B,P_1)$ and $(B,P_2)$ are gauge equivalent as G-labeled manifolds.
\end{defi}

Since gauge transformations are trivial on the boundary, it is consistent with the $V$-labeling.

\begin{eg}
    Small ball $\frac{1}{2}D^{n+1}$ inside big ball $D^{n+1}$. 
    Consider the bulk vector $D^{n+1}\setminus \frac{1}{2}D^{n+1}$ with a wall $\frac{3}{4}S^n$ with a G-labeling $g$.
    Different open holonomies $g$ (relative potential) represent non-equivalent bulk vectors.
\end{eg}

\subsection{Spacetime cobordism category}

In this subsection we are going to define the G-equivariant Spacetime cobordism category, to which we are going to define the G-equivariant TQFT functor in the next section. The Spacetime cobordism category of \cite{liu2024functional} is the case where the group $G=\{e\}$ is trivial. 

The objects in this category are called ``G-equivariant Time Boundaries".

\begin{defi}
    A G-equivariant Time Boundary $(F, P_F)$ is just a compact G-labeled $n$ manifold.
\end{defi}

\begin{rmk}
Note that there exists G-label structure on Time Boundary itself, as well as on Space Boundary. But Time Boundaries do not permit $V$-labeling. 
\end{rmk}

The morphisms in this category are the so called ``G-equivariant Spacetime cobordisms".

\begin{defi}
    A G-equivariant Spacetime cobordism with G-equivariant Time Boundary $\mathcal{F}=(F, P_F)$ is a $\mathcal{B}=(B,P_B,M, M_V,\iota)$ such that:
    \begin{itemize}
        \item $(B,P_B)$ is a compact G-labeled $(n+1)$-manifold, $P_B$ is called the virtual phase
        \item $M\subset \partial B$ is a compact $n$ PL submanifold, $M_\mathcal{V}$ a $\mathcal{V}$-labeled structure on $M$, with $\mathcal{M}:=(M,P|_M,M_\mathcal{V})$ called the Space Boundary
        \item $\iota:F\to \partial B$ is an embedding of $G$-labeled manfolds
        \item $M \cup \iota(F) = \partial B$
        \item $M \cap \iota(F) = \partial M = -\partial \iota(F)$
        \item $\iota^*(P_B|_{\iota(F)})=P_F$
        \item $\iota(F)$ admits a stratified collar in $B$: there is a stratified PL homeomorphism $c:F\times(-\epsilon,0]\to U$ onto an open neighborhood $U$ of $\iota(F)$ in $B$ preserving the G-labeling, and $c(x,0)=\iota(x),\forall x\in F$, and $c(\partial F\times (-\epsilon,0])\subset \partial B$.
    \end{itemize}
\end{defi}

\begin{defi}
    Two Spacetime cobordisms $\mathcal{B}_i=(B_i,P_i,M_i, M_{V,i},\iota_i), i=1,2$ are called gauge equivalent if there is a PL homeomorphism $f:B_1\to B_2$ $f(M_1)=M_2,f_*(M_{V,1})=M_{V,2}, f\circ \iota_1=\iota_2$ and $(B_2,P_2)$ and $(B_2,f_*(P_1))$ are gauge equivalent as G-labeled manifolds.
\end{defi}

\begin{defi}
    G-equivariant Spacetime cobordism category is defined as the category whose:
    \begin{itemize}
        \item objects are G-equivariant time boundaries
        \item morphisms from $\mathcal{F}_1$ to $\mathcal{F}_2$ are gauge equivalence classes of G-equivariant Spacetime cobordisms with G-equivariant Time Boundary $\bar{\mathcal{F}_1}\sqcup \mathcal{F}_2$
        \item Compositions are defined by gluing in the obvious way
    \end{itemize}
\end{defi}

This cobordism category is a symmetric monoidal category. 
Tensor product is defined as the horizontal concatenation. 

\begin{defi}
    A G-equivariant Spacetime TQFT or a G-equivariant TQFT with spacetime boundary is defined as a symmetric monoidal functor from G-equivariant Spacetime cobordism category to the category of finite dimensional vector spaces.
\end{defi}

Sometimes when there is no confusion, we may abbreviate and not to say ``G-equivariant".

When the G-$S^n$-algebra is (RP), the Spacetime cobordism category also has a dagger structure, which is defined as the mirror image. 
Namely, the manifold is orientation reversed, the G-labelings remain the same (though its simplicies are orientation reversed), and the $\mathcal{V}$-labeling is changed to its reflection. 

\begin{defi}
    A unitary G-equivariant Spacetime TQFT or a unitary G-equivariant TQFT with spacetime boundary is defined as a symmetric monoidal functor from G-equivariant Spacetime cobordism catefory to the category of finite dimensional Hilbert spaces which is a dagger functor.
\end{defi}

%% file: G-equivariant_TQFT.tex
\section{G-equivariant Alterfold TQFT}

Let $\mathcal{V}=(V,Z)$ be a G-equivariant $S^n$-algebra. 
Fix a representative set of irreducible $k$-morphisms for each $k=0,1,\dots,n-1$.

All the summations involving $k$-morphisms for $0\le k\le n-1$ are understood to involve only the ones in the representative set.

\subsection{$\mathcal{S}^{k-1}$-relations}

Before defining TQFT partition function, let's introduce the so called $S^{k-1}$-relations, which is the most important gredient of this TQFT construction. 
They provide both the triangulation-based definition of the partition function and global relations, which do not necessarily perform on a triangulation. 

\begin{defi}\label{Sk-1 relations}
    Suppose that 
    \begin{itemize}
        \item 
        $\mathcal{B}=(B,\mathcal{M},P)$ is a bulk vector;
        \item 
        $\mathcal{U}\subset \mathcal{B}$ is an $(n+1)$ compact stratified submanifold and we have a stratified homeomorphidm $\phi:\mathcal{U}\to \mathcal{D}^{k}\times D^{n+1-k}$ preserving G-labelings.
        Denote by $\mathcal{S}^{k-1}=\partial \mathcal{D}^k$;
        \item 
        $\phi(\mathcal{U}\cap \partial B)=\mathcal{S}^{k-1}\times D^{n+1-k}$;
    \end{itemize}
    and set $V=\phi^{-1}(D^k\times \frac{1}{2}D^{n+1-k})$, $W=\phi^{-1}(D^k\times \frac{1}{4}D^{n+1-k})$, $\epsilon<\frac{1}{8}$.
    
    We define an $\mathcal{S}^{k-1}$-relation (also called Move-$k$) for this bulk vector, where in the followings we identify $\mathcal{U}$ with $\mathcal{D}^{n+1}$ by $\phi$.
    (Or abuse of language, Spacetime cobordism $(\mathcal{D}^{n+1}, \mathcal{S}^{k-1}\times D^{n-k+1})$ with Time Boundary supported on $D^k \times S^{n-k}$.)
    (We use the notation in Def.~\ref{Notation: labeled at (2)})
    \begin{itemize}
        \item 
        When $k=0$,
        \begin{equation}
            \Phi(\mathcal{D}^{n+1}, \emptyset)
            =
            \sum_{\alpha_0} \frac{Tr(\alpha_0)}{\mu(\alpha_0)}
            (D^{n+1}\setminus \epsilon D^{n+1}, 
            \mathcal{D}^0(\alpha_0)\times\epsilon S^{n},
            P')
        \end{equation}
        summing over representatives of $0$-morphisms $\alpha_0$.
        The new G-labeled structure $P'$ in the new bulk is arbitrarily chosen so that 
        \begin{equation}
        \begin{cases}
            P'=P \text{ on }B\setminus W\\
            \text{the radial transport is the identity}
        \end{cases}
        \end{equation}
        and coincide with the new Space Boundary. 
        
        \item 
        When $k=1$,
        \begin{equation}
            \Phi(\mathcal{D}^{n+1}, \mathcal{S}^{0}\times D^{n})
            =
            \sum_{lb(\alpha_1)=\mathcal{S}^0} \frac{Tr(\alpha_1)}{\mu(\alpha_1)}
            (D^{n+1}\setminus D^1\times \epsilon D^{n}, 
            \mathcal{S}^{0}\times (D^{n}\setminus \epsilon D^{n})\cup \mathcal{D}^1_2(\alpha_1) \times \epsilon S^{n-1}
            ,P')
        \end{equation}
        summing over summing over representatives of irr. $1$-morphisms $\alpha_1$ with link boundary $\mathcal{S}^{0}$.
        (Here $\mathcal{D}^1_2$ is the type of $\alpha_1$.)
        The new G-labeled structure $P'$ in the new bulk is arbitrarily chosen so that 
        \begin{equation}
            P'=P \text{ on }B\setminus W
        \end{equation}
        and coincide with the new Space Boundary. 
        
        \item 
        When $2\le k\le n-1$,
        \begin{equation}
        \begin{aligned}
            \Phi(\mathcal{D}^{n+1}, \mathcal{S}^{k-1}\times D^{n-k+1})
            =
            \sum_{lb(\alpha_k)=\mathcal{S}^{k-1}} 
            &\frac{Tr(\alpha_k)}{\mu(\alpha_k)}
            (D^{n+1}\setminus D^k\times \epsilon D^{n+1-k}, 
            \\
            &\mathcal{S}^{k-1}\times (D^{n+1-k}\setminus \epsilon D^{n+1-k})\cup \mathcal{D}^k_2(\alpha_k) \times \epsilon S^{n-k},
            P')
        \end{aligned}
        \end{equation}
        summing over representatives of irr. $k$-morphisms $\alpha_k$ with link boundary $\mathcal{S}^{k-1}$.
        (Here $\mathcal{D}^k_2$ is the type of $\alpha_k$.)
        The new G-labeled structure $P'$ in the new bulk is arbitrarily chosen so that 
        \begin{equation}
            P'=P \text{ on }B\setminus W
        \end{equation}
        and coincide with the new Space Boundary. 
        
        \item 
        When $k=n\ge 2$,
        \begin{equation}
        \begin{aligned}
            \Phi(\mathcal{D}^{n+1}, \mathcal{S}^{n-1}\times D^{1})
            &=
            \sum_{\alpha_n \in bs(V_{\mathcal{S}^{n-1}})}
            (D^{n+1}\setminus D^n\times \epsilon D^{1}, 
            \mathcal{S}^{n-1}\times (D^{1}\setminus \epsilon D^{1})
            \\
            &\cup \mathcal{D}^n(\alpha_n) \times \{-\epsilon\}
            \cup \mathcal{D}^n(\alpha_n^\vee) \times \{+\epsilon\},
            P')
        \end{aligned}
        \end{equation}
        summing over $\alpha_n$ on a basis of $V_{\mathcal{S}^{n-1}}$ with $\alpha_n^\vee$ in the dual basis, 
        (Here $\mathcal{D}^n=\Lambda\mathcal{S}^{n-1}$.)
        The new G-labeled structure $P'$ in the new bulk is arbitrarily chosen so that 
        \begin{equation}
            P'=P \text{ on }B\setminus W
        \end{equation}
        and coincide with the new Space Boundary. 
        \item 
        When $k=n+1$,
        \begin{equation}
            \Phi(\mathcal{D}^{n+1}, \mathcal{S}^{n})
            =
            Z(\mathcal{S}^n)
        \end{equation}
    \end{itemize}
\end{defi}

\begin{rmk}
Note that here the $S^{k-1}$-relations are the G-equivariant version of that of \cite{liu2024functional}, but the $k=1$-relation (or called Move-1) is different from $2\le k\le n-1$. 
\end{rmk}

\begin{rmk}
For the "substituting area" $\mathcal{D}^k \times \epsilon^2 D^{n-k}$, its boundary
\begin{equation}
    \partial(\mathcal{D}^k \times \epsilon^2 D^{n-k})
    =
    \mathcal{S}^{k-1} \times \epsilon^2 D^{n-k}
    \cup
    \mathcal{D}^k \times \epsilon^2 S^{n-k-1}
\end{equation}
has two parts. 
The first part $\mathcal{S}^{k-1} \times \epsilon^2 D^{n-k}$ is the link boundary of $\alpha_k$ and is completely determined by the Space Boundary in which the relation is performed. 
The second part $\mathcal{D}^k \times \epsilon^2 S^{n-k-1}$ is the type of $\alpha_k$ and can be freely chosen, as long as they coincide at the intersection. 
This is the case when $2\le k\le n-1$. 

For $k=1$, the open holonomy between $\{+1\} \times \epsilon^2 D^{n-1}$ and $\{-1\} \times \epsilon^2 D^{n-1}$ is undetermined by the local Space Boundary on which the relation is performed, yet constrained by the global holonomy structure, i.e. constrained by the other parts of the Space Boundary or the bulk.
The admissible summation (such that $P'$ exists) in Move-1 is over all $\alpha_1$ with degree
\begin{equation}
    \prod_{i} g_i^{\epsilon_i}
\end{equation}
where the product is over all the $codim=1$-stratification $i\in P^{n}$ along the intersection of longitude $D^1$ with $P$, and $\epsilon_i$ is determined by the orientation. 

The above are summarized in Lemma~\ref{Filling: Sector is determined by homotopy group}.
\end{rmk}

\begin{lem}\label{Filling: anything}\label{App: Filling: anything}

    (1). Suppose that $\mathcal{S}$ is a G-labeled manifold with support $S^k$.
    Then there is a G-labeled manifold $\mathcal{F}$ with support $D^{k+1}$ such that $\partial \mathcal{F}=\mathcal{S}$, given $k\ge 2$ or $k=1$ and the holonomy of $\mathcal{S}$ is trivial.

    (2). Suppose that $\mathcal{B}=(B,\mathcal{M},P)$ is a bulk vector. 
    If $S$ is PL submanifold of $\mathcal{M}$ with support $S^1$, transversal to $\mathcal{M}$ and is contractible in $\mathcal{B}$.
    Then the holonomy of $\mathcal{S}$ is trivial.
\end{lem}
\begin{proof}
    One only need to notice that
    \begin{equation}
        \pi_1(BG)=G
    \end{equation}
    \begin{equation}
        \pi_{\ge 2}(BG)=0
    \end{equation}
    For the detailed proof, see the appendix.
\end{proof}

\begin{lem}(Sector)\label{Filling: Sector is determined by homotopy group}

    In order for the existence of the new G-labeled structure $P'$ in the $\mathcal{S}^{k-1}$ relations, for $2\le k\le n-1$ it is equivalent that the link boundary $lb(\alpha)=\mathcal{S}^{k-1}$ is fixed, without any constraint of the type $\mathcal{D}^k_2$.
    
    For $k=1$, the existence of $P'$ is equivalent to that the transport of the type $\mathcal{D}^1_2$ is equal to the transport of the type $\mathcal{D}^1$. 
\end{lem}
\begin{proof}
    Take any summand $\alpha_k$ in the $\mathcal{S}^{k-1}$ relation, and denote by $\mathcal{D}^k_2=tp(\alpha_k)$.
    By Lem.~\ref{Filling: anything}, one can choose a G-labeled manifold $\mathcal{C}_{k+1}$,
    \begin{equation}
        |\mathcal{C}_{k+1}|=D^k\times [+\epsilon,+1],
        \quad
        \mathcal{C}_{k+1}|_{D^k\times \{+1\}}=\mathcal{D}^k
        \quad
        \mathcal{C}_{k+1}|_{D^k\times \{+\epsilon\}}=\mathcal{D}^k_2
    \end{equation}
    and first take $P_2=\mathcal{C}_{k+1}\times S^{n-k}$.
    Thus $P_2$ satisfies the geometric Spacetime cobordism condition, but without labeling $\alpha_k$.
    
    Then one can label $\alpha_k$ at $D^k \times \epsilon^2 D^{n-k}\times\{+\epsilon\}$ and uses Lem.~\ref{Filling: anything} again to change the filling of the small ball neighborhood $D^k \times \epsilon^2 D^{n-k}\times[+\epsilon,+\epsilon+\epsilon^2]$.
    This proves the ``if" part.

    For the ``only if" part for $k=1$, let's choose a point $p\in \epsilon S^{n-1}$ avoiding the $\alpha_1$ label, and thus $\mathcal{D}_2^1\times \{p\}\cup \mathcal{S}^{0}\times [p,p/\epsilon] \cup \mathcal{D}^1\times\{p/\epsilon\}$ is a PL submanifold transversal to the new Space Boundary.
    By Lem.~\ref{Filling: anything}, one must have its holonomy to be trivial, or equivalently, transports of $\mathcal{D}^1$ and $\mathcal{D}^1_2$ to be equal.
\end{proof}

\begin{eg}
    Torus with nontrivial longitude closed holonomy. Local data does not determine degree of Move-1 completely.
\end{eg}

In the rest of this subsection, we introduce a property satisfied by the hyper sphere functional.

\begin{defi}(Adjacent Move)

    Suppose $\mathcal{S}^n$ is $\mathcal{V}$-labeled with a local part $\mathcal{S}^{k-1}\times \epsilon D^{n-k+1}$ and a local part $\mathcal{S}^k \times \epsilon D^{n-k}$. 
    We first apply $\mathcal{S}^{k-1}$ relation to the local part $\mathcal{S}^{k-1}\times \epsilon D^{n-k+1}$, and then the $\mathcal{S}^k_\pm$ relation, where
    \begin{equation}
    \begin{aligned}
        \mathcal{S}^k_+ := (\mathcal{D}^k \times\{\epsilon\} \cup \mathcal{S}^k|_{x_{k+1}\ge \epsilon})\times O_{n-k}
        \\
        \mathcal{S}^k_- := (\rho_{k,k+1}(\mathcal{D}^k \times\{\epsilon\}) \cup \mathcal{S}^k|_{x_{k+1}\le -\epsilon})\times O_{n-k}
    \end{aligned}
    \end{equation}
    where $\rho_{k,k+1}$ denote the 180 rotation of the $(k,k+1)$-th coordinate.
    The resulted $\mathcal{V}$-labeled manifold has support homeomorphic to $S^n$.
    The composition of the $\mathcal{S}^{k-1}$ relation and the $\mathcal{S}^{k}_\pm$ relation is denoted as $F^k_\pm$.
\end{defi}

\begin{rmk}
    We need two $F^k$'s to do induction in order to deal with the problem that the local null relation may intersect the second cut.
\end{rmk}

\begin{thm}(Adjacent identity)\label{Adjacent Move}

    If the G-equivariant $S^n$-algebra $\mathcal{V}$ is (CF) and (SS), the $S^n$ functional $Z$ is invariant under the Adjacent Move $(k,k+1)$ for $1\le k$: $Z\circ F^k_\pm=Z$. 
    If $\mathcal{V}$ is further (MI), $Z$ is also invariant under the Adjacent Move $(0,1)$.
\end{thm}

\begin{proof}
    (1). For $(k,k+1)=(n,n+1)$, this follows from
    \begin{equation}
        \langle \alpha_+,\alpha_- \rangle = \sum_{\beta_n \in bs(\tilde{V}_{\mathcal{S}^{n-1},-})}
        \langle \alpha_+,\beta_n \rangle \langle \beta_n^\vee,\alpha_- \rangle,\quad
        \forall \alpha_\pm \in V_{\mathcal{S}^{n-1},\pm}
    \end{equation}
    
    (2). For $(k,k+1)=(n-1,n)$, first apply $\mathcal{S}^{n-2}$-relation the Space Boundary becomes
    \begin{equation}
        \mathcal{S}^n\setminus(\mathcal{S}^{n-2}\times \epsilon D^{2}) \cup \sum_{lb(b)=\mathcal{S}^{n-2}}
        \frac{Tr(b)}{\mu(b)}\mathcal{D}^{n-1}(b)\times \epsilon S^{1}
    \end{equation}
    The second part is the newly created S.B. by the Move. 
    The support of the first part, along its G-labeled boundary, is homeomorphic (denoted by $\phi$) to $D^{n-1}\times S^{1}$ with boundary $\mathcal{S}^{n-2}\times S^{1}$.
    By Prop.~\ref{A basis} we may assume that it is $\phi^{-1}(\mathcal{D}^{n-1}_1(\beta)\times S^{1})+\text{local null relation term } N$, for some $\partial(\mathcal{D}^{n-1}_1)=\mathcal{S}^{n-2}$ and $\beta$ not necessarily lying in the representative set.

    The result follows from a direct computation as follows.
    The local null relation term contributes to $0$ since it is disjoint from the surgery regions and thus is intact during the move and its final evaluation vanishes (later on we will omit this argument).
    The first term contributes to
    \begin{equation}
    \begin{aligned}
         &\sum_{lb(\widehat{\beta})=\mathcal{S}^{n-2}}
         Tr(\widehat{\beta})
         \sum_{\gamma\in bs(\tilde{V}_{tp(\widehat{\beta})\cup_{\mathcal{S}^{n-2}} tp(\beta)})}
         Z(\beta\underset{n}{\times}\gamma\underset{n}{\times}\widehat{\beta}\underset{n}{\times}\gamma^\vee)
         \\
         &=
         Tr(b)
         \sum_{\gamma\text{ is a basis of $V_{tp(b)\cup_{\mathcal{S}^{n-2}}tp(\beta)}$}}
         Z(\beta\underset{n}{\times}\gamma\underset{n}{\times}b\underset{n}{\times}\gamma^\vee)
         \\
         &=
         Tr(b)Tr(x\mapsto \beta x b)=Tr(b)dim(\beta V_{tp(\beta)\cup_{\mathcal{S}^{n-2}} tp(b)} b)
         \\
         &=Tr(b)=Tr(\beta)
         =Z(id_{\beta})
    \end{aligned}
    \end{equation}
    we used in the first line that the only term with $\widehat{\beta}\sim \beta$ contributes, which is denoted $b$.
    
    (3). $(k,k+1)=(1,2)\dots(n-2,n-1)$. 

    We first prove that $Z\circ F^k_+=Z\circ F^k_-$.
    One can directly verify that $F^{k+1}_-\circ F^k_+=F^{k+1}_+\circ F^k_-$.
    By induction hypothesis, $Z\circ F^{k+1}_-=Z\circ F^{k+1}_+=Z$.
    So $Z\circ F^k_+=Z\circ F^{k+1}_-\circ F^k_+=Z\circ F^{k+1}_+\circ F^k_-=Z\circ F^k_-$.

    At last let's prove that $Z\circ F^k_\pm=Z$.
    First apply $\mathcal{S}^{k-1}$-relation and the Space Boundary becomes
    \begin{equation}
        \mathcal{S}^n\setminus(\mathcal{S}^{k-1}\times \epsilon D^{n-k+1}) \cup \sum_{lb(b)=\mathcal{S}^{k-1}}
        \frac{Tr(b)}{\mu(b)}\mathcal{D}^k(b)\times \epsilon S^{n-k}
    \end{equation}
    The second part is the newly created S.B. by the Move. 
    The support of the first part, along its G-labeled boundary, is homeomorphic (denoted by $\phi$) to $D^k\times S^{n-k}$ with boundary $\mathcal{S}^{k-1}\times S^{n-k}$.
    By Prop.\ref{A basis}, we may assume that $\mathcal{S}^n|_{D^k\times S^{n-k}}=\mathcal{D}^k_1(\beta)\times S^{n-k}+\text{``local null relation terms" }N_++N_-$, for some $\partial(\mathcal{D}^k_1)=\mathcal{S}^{k-1}$ and $\beta$ in the representative set, and the terms of the null relation $N_\mp$ disjoint from $S^k|_{x_{k+1}\le -\epsilon}\times \epsilon D^{n-k}$ and $S^k|_{x_{k+1}\ge +\epsilon}\times \epsilon D^{n-k}$respectively.
    (We map $D^k\times\{-1\}\times \epsilon D^{n-k}$ in Prop.~\ref{A basis} to $S^k|_{x_{k+1}\le -\epsilon}\times \epsilon D^{n-k}$ in this setting).
    
    One can compute the second term to be $Z\circ F^k_\pm(N_\pm)=0$.
    The first term is computed to be
    \begin{equation}
    \begin{aligned}
        &Z\circ F^k_-(\mathcal{D}^k_1(\beta)\times S^{n-k})
        \\
        &\overset{Lem.~\ref{Filling: Sector is determined by homotopy group}}{=}\sum_{lb(\widehat{\beta})=\mathcal{S}^{k-1}}\frac{Tr(\widehat{\beta})}{\mu(\widehat{\beta})}
        \sum_{lb(\gamma)=tp(\widehat{\beta})\times\{-1\}\cup tp(\beta)\times\{+1\}\cup \mathcal{S}^{k-1}\times D^1} \frac{Tr(\gamma)}{\mu(\gamma)}
        Z(\widehat{\beta}\underset{k+1}{\times}\gamma\underset{k+1}{\times}\beta)
        \\
        &=
        \frac{Tr(\beta)}{\mu(\beta)}
        \sum_{\gamma\text{ is a $\beta-\beta$ bimodule}} 
        \frac{Tr(\gamma)}{\mu(\gamma)}Tr(\gamma)
        =Tr(\beta)
    \end{aligned}
    \end{equation}
    In the first equality we used Lemma~\ref{Filling: Sector is determined by homotopy group}.
    All of the bimodules $\gamma$ appear in the summation since $k+1\ge 2$.

    (4). For $(k,k+1)=(0,1)$, by the same argument the null relation term can be computed separately and contributes to zero, and one need only to compute the contribution of the first term with boundary $\phi^{-1}(\mathcal{D}^0_1(\beta)\times S^{n})$ for a $\beta$ in the representative set.
    
    By Lemma~\ref{Filling: Sector is determined by homotopy group}, the summation of $\gamma$ is changed to the summation over a fixed ``sector" which is graded by G. 
    Suppose the open holonomy of $\mathcal{D}^1$ is $g$, then this is
    \begin{equation}
    \begin{aligned}
        &\overset{Lem.~\ref{Filling: Sector is determined by homotopy group}}{=}
        \frac{Tr(\beta)}{\mu(\beta)}
        \sum_{\gamma\text{ is a $\beta-\beta$ bimodule and $gr(\gamma)=g$}} 
        \frac{Tr(\gamma)}{\mu(\gamma)}Tr(\gamma)
        \\
        &=
        \frac{Tr(\beta)}{\mu(\beta)}\mu_g(\beta)
        =Tr(\beta)
    \end{aligned}
    \end{equation}
    In the last equality we used the MI condition: $\mu_g(\beta)=\mu(\beta)$.
\end{proof}

\subsection{Partition function}

From now on suppose that $\mathcal{V}$ is (CF), (SS) and (MI).

In this subsection we are going to define the partition function for bulk vectors and prove that it is well-defined.

\begin{defi}(Algorithm for the partition function)

    Given a bulk vector $\mathcal{B}=(B,P_B,\mathcal{M})$, we define its partition function as follows. 
    Take a full transversal triangulation $h:\Delta\to B$ of $B$. 

    Ordering the simplices of $\Delta$  of dimensional $0,...,n$ (excluding top dimensional simplices) not contained in the boundary of $B$ in a sequence $\sigma_1,...,\sigma_N$ such that each face of each simplex is at the former of itself.

    Set $\mathcal{B}_0=\mathcal{B}$. 
    Do the following inductively for $s=1,...,N$:
    \begin{itemize}
        \item 
        Take the $k_s$-simplex $\sigma_s\in \Delta$, and take a linear term, i.e. a bulk vector $\mathcal{B}_{s-1}$ from the last iteration.
        Denote by $\mathcal{D}_s:=h(|\sigma_s|)\cap \mathcal{B}_{s-1}$ and $\mathcal{S}_s=h(|\sigma_s|)\cap \partial \mathcal{B}_{s-1}$.
        Denote by $\mathcal{S}^{k_s-1}=\partial \mathcal{D}^{k_s}_1$.
        \item
        According to Lem.~\ref{Filling: algorithm can continue}, one has a stratified homeomorphism $\phi_s:\mathcal{D}^{k_s}_1\times D^{n+1-k_s}\to \mathcal{U}$ onto an stratified submanifold $\mathcal{U}$ preserving the G-labelings and $\phi_s(D^{k_s}_1\times \{O_{n+1-k_s}\})=D_s$ and $\phi_s(S^{k_s-1}_1\times D^{n+1-k_s})=U\cap \partial B_{s-1}$.
        \item 
        Perform $\mathcal{S}^{k_s-1}$-relation to $(\mathcal{U},\mathcal{U}\cap \partial B_{s-1})=\phi_s(\mathcal{D}^{k_s}_1\times D^{n+1-k_s}, \mathcal{S}^{k_s-1}\times D^{n+1-k_s})$.
        Use the fillings in Lem.~\ref{Filling: algorithm can continue};
        \item 
        Apply the above for each summand created by the previous step and extend linearly.
        We get a linear combination of bulk vectors $\sum_{\lambda} a_\lambda \mathcal{B}_{s,\lambda}$ with support
        \begin{equation}
            B_s := \overline{B_{s-1}\setminus \phi_s(D^{k_s}_1\times \epsilon D^{n+1-k_s})}^{B_{s-1}}
        \end{equation}
    \end{itemize}
    At last, use Move $(n+1)$ to each remaining ball.

    Eventually, we obtain a scalar. It is defined as the partition function $Z(\mathcal{B})$.
\end{defi}

The topological details ensuring the continuation of the algorithm is treated in the following lemma.

\begin{lem}(Continuation)\label{Filling: algorithm can continue}\label{App: Filling: algorithm can continue}

    Suppose that $\mathcal{B}$ is a bulk vector, $\Delta$ is a full transversal triangulation.
    Then at each step of the algorithm, one can guaranty that
    \begin{itemize}
        \item 
        $\mathcal{D}_s:=h(|\sigma_s|)\cap \mathcal{B}_{s-1}$ is supported on $D^{k_s}$, and $\mathcal{D}_s\cap \partial B_{s-1}=\partial\mathcal{D}_s$;
        \item 
        one has a stratified homeomorphism $\phi_s:\mathcal{D}^{k_s}_1\times D^{n+1-k_s}\to \mathcal{U}$ onto a stratified submanifold $\mathcal{U}$ preserving the G-labelings and $\phi_s(D^{k_s}_1\times \{O_{n+1-k_s}\})=D_s$ and $\phi_s(S^{k_s-1}_1\times D^{n+1-k_s})=U\cap \partial B_{s-1}$;
        \item 
        There exists a filling inside $\phi_s(D^{k_s}\times (D^{n+1-k_s}\setminus \epsilon D^{n+1-k_s}))$ that guaranties that the next step satisfies the above conditions.
    \end{itemize}
    At the end we are left with balls.
\end{lem}
\begin{proof}
    For the detailed proof, see the appendix.
\end{proof}

The algorithm is independent of the fillings $P'$ of each $\mathcal{S}^{k-1}$ relation as long as the algorithm is legal w.r.t. this filling, since the link boundary (and the transport for Move $1$) for the later $\mathcal{S}^{k-1}$ relations are independent of the fillings (by assumptions on Move $0$ filling), and the hyper sphere function is independent of the bulk.

\begin{prop}\label{App: order}
    The algorithm is independent of the order of steps, as long as the condition that ``each face of each simplex is at the former of itself" is satified.
\end{prop}
\begin{proof}
    See the appendix.
\end{proof}

\begin{defi}
    Suppose $\Delta^n=[e_0,...,e_n]$ is a linear $n$-simplex and $p$ is a point in the interior of $\Delta^k=[e_0,...,e_k]$.
    The $p$-subdivision of $\Delta^n$ is the subcomplex
    \begin{equation}
        \{[e_0,...,\hat{e_i},...,e_k]\star p\star [e_{k+1},...,e_n]\mid i=0,...,k\}
    \end{equation}
\end{defi}

This is called stellar and weld in other literature, for example \cite{lickorish1999simplicial}.
The Adjacent identity is just the invariance of the partition function under $p$-subdivisions.

\begin{defi}
    A triangulation of a ball $D^{n+1}$ is called shellable, if its $(n+1)$ dimensional simplices can be order as $\sigma_1,...,\sigma_N$ such that for $i=2,...,N$,
    \begin{equation}
        \sigma_i \cap \bigcup_{j\le i-1} \sigma_j
    \end{equation}
    is a union of some $n$-dimensional sub simplices of $\sigma_i$ and is a combinatorial ball.
\end{defi}

\begin{lem}(Shellable ball)\label{Shellable}\label{App: Shellable}

    Suppose that the algorithm reaches a step in which there is a PL ball $|\mathcal{Q}|\sim_{PL}D^{n+1}$, and the remaining triangulation in $Q$ is shellable.
    Then if we continue the algorithm for $\mathcal{Q}$ (the factor is denoted temporarily as $Z_{\text{tail}}(\mathcal{Q})$, which equals the algorithm output $Z(\mathcal{Q})$ after Coro.~\ref{partition function ball eval}), the factor is the evaluation of the hypersphere functional
    \begin{equation}
        Z_{\text{tail}}(\mathcal{Q})=Z(\mathcal{Q}|_{\partial Q}).
    \end{equation}
\end{lem}
\begin{proof}
    Use the Adjacent identity.
    See the appendix.
\end{proof}

\begin{rmk}
    Shellableness does not imply fullness. 
    So one cannot say that one is performing the algorithm for the bulk vector $\mathcal{Q}$.
\end{rmk}

\begin{prop}\label{App: p-subdivision invariance}
    The partition function is invariant under $p$-subdivisions.
\end{prop}
\begin{proof}
    See the appendix.
\end{proof}

In order to prove the invariance of the partition function of the triangulations, we need to prove its invariance under subdivisions and isotopies.
Invariance under subdivisions are compositions of invariance under $p$-subdivisions, but one should take care of the transversal properties.

\begin{lem}\label{App: p-subdivision sequence}
    Suppose that the triangulation $\Gamma$ is a subdivision of the full triangulation $\Delta$. Then they are connected by a sequence of $p$-subdivisions and their inverses, which are all full.
\end{lem}
\begin{proof}
    This is a slight generalization of a theorem in \cite{lickorish1999simplicial}.
    For the details, see the appendix.
\end{proof}

\begin{prop}\label{partition function subdivision inv}
    The partition function is invariant under subdivisions.
\end{prop}
\begin{proof}
    Suppose that $\Delta$ and $\Gamma$ are full transversal triangulations of the bulk vector $\mathcal{B}$.
    By the above lemma one can connect them by a full sequence $\Delta = K_0 \leftrightarrow K_1 \leftrightarrow ... \leftrightarrow K_m =\Gamma$.
    Note that $K_j$'s may not be transversal to $\mathcal{B}$, but one can find their common subdivision $\Lambda$ by Thm.~\ref{common subdivision} and an isotopy $h_*(\mathcal{B})=\mathcal{B}'$ such that $\Lambda$ is transversal to $\mathcal{B}'$ and $h(\Gamma)$ preserves simplices of $\Gamma$ by Thm.~\ref{transversal subdivision preparation}.
    Then all the $K_j$'s are full and transversal to $\mathcal{B}'$.
    Then one has 
    \begin{equation}
        Z(\mathcal{B},\Delta) = Z(h_*(\mathcal{B}),h(\Delta)) = Z(\mathcal{B}',\Delta) = Z(\mathcal{B}',K_j) = ... = Z(\mathcal{B}',\Gamma) = Z(h_*(\mathcal{B}),h(\Gamma)) = Z(\mathcal{B},\Gamma)
    \end{equation}
    where the first and the last equalities are tautologies, the second and the second last are by assumption $h(\Gamma)$ preserve simplices of $\Gamma$, and the middle equalities are by the Adjacent Move Thm.~\ref{Adjacent Move}.
\end{proof}

\begin{coro}\label{partition function ball eval}
    Suppose that $\Delta$ is a full transversal triangulation of a bulk vector $\mathcal{Q}$ with support $D^{n+1}$. Then
    \begin{equation}
        Z(\mathcal{Q},\Delta) = Z(\mathcal{Q}|_{S^n})
    \end{equation}
    where the LHS is the partition function, and the RHS is the hyper sphere functional.
\end{coro}
\begin{proof}
    Take a shellable transversal triangulation of the boundary sphere $\partial Q$, and take the cone. 
    Using Thm.~\ref{transversal triangulation} one can WLOG suppose that it is transversal to $\mathcal{Q}$. 
    It is denoted $\Delta'$.
    By Thm.~\ref{common subdivision}, there is a common subdivision $\Gamma$ of $\Delta,\Delta'$.
    Using Thm.~\ref{transversal subdivision preparation}, there is an isotopy $F$ fixing the simplices of $\Delta$ and make $\Gamma$ transversal.
    Then 
    \begin{equation}
        Z(\mathcal{Q},\Delta)=Z(F_*(\mathcal{Q}),\Delta)=Z(F_*(\mathcal{Q}),\Gamma)=Z_{\text{tail}}(F_*(\mathcal{Q}),\Delta')=Z(F_*(\mathcal{Q})|_{\partial Q})=Z(\mathcal{Q}|_{\partial Q})
    \end{equation}
    where we used subdivision invariance Prop.~\ref{partition function subdivision inv} and Lem.~\ref{Shellable}.
\end{proof}

\begin{lem}\label{subdivision shellable}
    For a transversal triangulation, it has a subdivision in which the star of each vertex is shellable.
\end{lem}
\begin{proof}
    See \cite{adiprasito2017subdivisions}.
\end{proof}

\begin{prop}\label{partition function isotopy invariance}
    The partition function is invariant under isotopy.
\end{prop}
\begin{proof}
    Take a full transversal triangulation $\Delta$.
    We are to prove that $Z(\mathcal{B},\Delta)=Z(\mathcal{B}',\Delta)$. 
    By subdivision invariance Prop.~\ref{partition function subdivision inv}, one can pass to a subdivision $\Gamma$ in which each vertex star is shellable by Lem.~\ref{subdivision shellable}, and using an isotopy $F$ by Thm.~\ref{transversal subdivision preparation} to make it transversal while fixing the simplices of the parent triangulation $\Delta$.
    Then 
    \begin{equation}
        Z(\mathcal{B},\Delta)=Z(F_*(\mathcal{B}),\Delta)=Z(F_*(\mathcal{B}),\Gamma)
    \end{equation}
    where we only used subdivision invariance.
    This left us to prove that $Z(F_*(\mathcal{B}),\Gamma)=Z(F'_*(\mathcal{B}'),\Gamma)$.
    
    So one can WLOG assume that $\Delta$ is full transversal and each vertex star is shellable (replace the triangulation by $F^{-1}_*(\Gamma)$ and isotopy by its composition with $F^{-1}\circ F'$).
    
    By Thm.~\ref{transversal isotopy step by step}, one can divide the isotopy $\mathcal{B}'=h_1(\mathcal{B})$ by a finite number of steps $\mathcal{B}=\mathcal{B}_0\to \mathcal{B}_1\to\dots \mathcal{B}_T=\mathcal{B}'$ such that $\mathcal{B}_j$'s are all transversal to $\Delta$, $\mathcal{B}_{j+1}=g_{j,1}(\mathcal{B}_j)$ and $g_{j,t}|_{B\setminus ost_K(v_j)}=id$ for vertices $v_1,\dots,v_T\in \Delta^0$.
    Thus one only need to prove the invariance of the state sum $Z$ under an isotopy inside the star $Q$ of a vertex. Write $P=\overline{B\setminus Q}$.
    Since the triangulation is combinatorial, the star is supported on $D^{n+1}$. 

    In the state sum formula of $Z(\mathcal{B},\Delta)$, we regroup the summation according to the decorations created on $\partial Q$. When the decoration $l$ on $\partial Q$ is fixed, the summation inside $Q$ and outside $Q$ is just $Z(P(l),\Delta)$ and $Z(\mathcal{Q}(l),\Delta)$.
    So the summation can be recast as $\sum_{l} w(l)Z_{\text{tail}}(P(l),\Delta)Z_{\text{tail}}(\mathcal{Q}(l),\Delta)$, a summation over a finite number of some certain $l\in \widehat{S^n}(V)$ decorated on $\partial Q\sim_{PL}S^n$ (where we used an identification), and $P(l),Q(l)$ denote the V-labeled manifold after performing the algorithm on the intersection $Q\cap \overline{B\setminus Q}$.
    ($Z_{\text{tail}}(Q(l),\Delta)$ is the continuation of the algorithm rather than the algorithm of the bulk vector $Q$, since the star may not be full).
    
    Because $\mathcal{B}_j$ and $\mathcal{B}_{j+1}$ is the same outside $Q$, $Z_{\text{tail}}(P_j(l),\Delta)=Z_{\text{tail}}(P_{j+1}(l),\Delta)$.
    By assumption $Q$ is shellable under $\Delta$, and by Lem.~\ref{Shellable}, the state sum $Z$ inside the star is equal to the evaluation of the hypersphere functional $Z$ on the link, i.e. $Z_{\text{tail}}(\mathcal{Q}(l),\Delta)=Z(l)$, which remains the same for an isotopy inside $Q$.

    So we are done.
\end{proof}

\begin{thm}(Well-defined)\label{Well-defined}

    If the G-equivariant $S^n$-algebra $\mathcal{V}$ is (CF), (SS) and (MI), then the partition function is independent of the transversal triangulation $\Delta$.
\end{thm}
\begin{proof}
    Suppose $\Delta_0,\Delta_1$ are two full transversal triangulations of the bulk vector $\mathcal{B}=(B,\mathcal{M},P_B)$.
    By Thm.~\ref{common subdivision}, $\Delta_0,\Delta_1$ have a common subdivision $\Gamma$, which is automatically full.
    By Thm.~\ref{transversal triangulation}, one can isotope $\mathcal{B}$ to $\mathcal{B}'$ such that it is transversal to $\Gamma$.
    We are going to prove $Z(\mathcal{B},\Delta_0)=Z(\mathcal{B},\Delta_1)$ by the following sequence
    \begin{equation}
    Z(\mathcal{B},\Delta_0)=Z(\mathcal{B}',\Delta_0)=Z(\mathcal{B}',\Gamma)=Z(\mathcal{B}',\Delta_1)=Z(\mathcal{B},\Delta_1)
    \end{equation}
    
    The middle two equalities are by subdivision invariance Prop.~\ref{partition function subdivision inv}, and the first and the last are by isotopy invariance Prop.~\ref{partition function isotopy invariance}.
    
    So we are done.
\end{proof}

\begin{coro}\label{buyaotouhuangainian}
    Suppose a bulk vector $\mathcal{B}=(B,\mathcal{M},P)$ has a null vector labeled on a local disk region on $\mathcal{M}$, then its partition function is zero.
\end{coro}
\begin{proof}
    By assumption one can find a local ball $H\subset B$, $T:=\overline{\partial H\setminus \partial B}$ has a stratified collar in $B$, and $H\cap \partial B$ contains a neighborhood of the local null relation $0=\sum_i a_i L_i$ in $\mathcal{V}$.
    
    First let's prove that one can find a transversal triangulation $\Delta_i$ for each $i$ such that $T$ is a subcomplex of $\Delta_i$'s, $\Delta_i=\Delta_j$ outside $H$ and $\Delta_i$'s are shellable inside $H$.
    In fact one can find a full transversal triangulation of $T$ once for all $i$, and extend it into a transversal one by Thm.~\ref{transversal triangulation extend to bulk} once for all outer bulk, and the inner triangulation is extended separately for $i$'s.
    By \cite{adiprasito2017subdivisions}(Coro.~I.3.10) one can take a common time of derived subdivisions inside $H$ to make each one shellable, and the triangulations on $T$ are the same for all $i$'s.
    Then use Thm.~\ref{transversal triangulation} to isotope the triangulation on $T$ into a transversal one, and then use isotopy inside $H$ to make each transversal, while fixing $T$ pointwise.
    Then one extend it again to a full transversal triangulation of the outside once for all $i$'s by Thm.~\ref{transversal triangulation extend to bulk} and isotopy fixing $T$.
    
    Secondly, by Prop.~\ref{App: order}, one can eliminate the simplices on $\overline{\partial H\setminus \partial B}$ not contained in $\partial B$ first, so the state sum becomes
    \begin{equation}
        Z(\mathcal{B})=\sum_i a_i \sum_s w_{s} Z_{i,\text{tail}}(H_{i,s})Z_{\text{tail}}((B\setminus H)_{s})
    \end{equation}
    and by Lem.~\ref{Shellable}, $\sum_i a_i Z_{i,\text{tail}}(H_{i,s})=Z(\sum_i a_i H_{i,s}|_{\partial H})=0$, since the local null relation region is intact in $H_{i,s}$.
\end{proof}

Now we have a partition function independent of the triangulation.
Next we prove that the $\mathcal{S}^{k-1}$-relations hold for arbitrary surgery moves and arbitrary fillings, not restricted to skeletons and the special fillings. 

At first we state a topological lemma guarantying the validity for a second cut (Move $(k+1)$) after a cut Move $k$ for special bulk vectors.

\begin{lem}\label{Filling: global relation can cut second for any general filling}\label{App: Filling: global relation can cut second for any general filling}
    Suppose that $0\le k\le n-1$ and $\mathcal{B}$ is a bulk vector with support $D^k\times [+\epsilon,+1]\times S^{n-k}$ and boundary $\mathcal{D}^k_2(\alpha_2)\times\{\epsilon\}\times  S^{n-k}\cup \mathcal{S}^{k-1}\times [\epsilon,+1]\times S^{n-k}\cup \mathcal{D}^k_1(\alpha_1)\times \{+1\}\times S^{n-k}$ for some $\partial \mathcal{D}^k_2=\partial \mathcal{D}^k_1=\mathcal{S}^{k-1}$ and $tp(\alpha_j)=\mathcal{D}^k_j$.
    Then there is an ambient isotopy $F$ fixing the boundary such that there is a G-labeled manifold $\mathcal{D}^{k+1}$ with support $D^k\times[\epsilon,+1]$ and $\mathcal{D}^{k+1}\times \epsilon^2 D^{n-k}$ is a stratified submanifold of $F_*(\mathcal{B})$ preserving G-labelings (where we regarded $\epsilon^2 D^{n-k}$ as a local disk region of $S^{n-k}$ avoiding label regions of $\alpha_j$'s).
\end{lem}
\begin{proof}
    See the appendix.
\end{proof}

\begin{thm}(Global relations)\label{Global relations}

    If the G-equivariant $S^n$-algebra $\mathcal{V}$ is (CF), (SS) and (MI), then the partition function $Z$ is invariant under the $\mathcal{S}^{k-1}$-relations.
\end{thm}

\begin{proof}
    Under the settings in Def.~\ref{Sk-1 relations}, let's prove $Z(\mathcal{B})=Z(\mathcal{B}')$. 
    We prove this by induction on $k=n+1,n,\dots,0$.

    (0). For Move $(n+1)$, it is proven in Coro.~\ref{partition function ball eval}.
    
    Assume that $Z$ is invariant under any Move $(k+1)$, let's prove its Move $k$ invariance. 

    Since the partition function is independent of full transversal triangulations (Thm.~\ref{Well-defined}), one can choose a convenient one as follows.
    Set $V=\phi^{-1}(D^k\times \frac{1}{2}D^{n+1-k})\subset U$.
    We choose a full transversal triangulation on $T=\overline{\partial V\setminus \partial B}$ by Thm.~\ref{transversal triangulation}.
    By the assumption on $U$ that $V$ and $\overline{B\setminus V}$ have collars of $T=\overline{\partial V\setminus \partial B}$, one can extend it to a transversal triangulation $\Delta$ on $B\setminus \mathring{V}$ and a transversal triangulation $\Delta'$ on $V$ by Thm.~\ref{transversal triangulation extend to bulk}. 
    (One can WLOG assume that they are full, otherwise one can use subdivisions.)
    By Prop.~\ref{App: order} we can eliminate the simplices in $\partial V$ not contained in $\partial B$ first, and then those in $B\setminus \overline{V}$, and we are left with $V$ with its collar shortened.
    Identify $V$ with $D^k\times D^{n+1-k}$ in the rest of the proof for convenience (drop the $\frac{1}{2}$).

    (1). So for Move $n$,  we obtain a $\mathcal{V}$-labeled manifold with support $\phi^{-1}(D^n\times S^{0})$.
    Write the decoration on two $D^n$ to be $\alpha_\pm$. 
    It is then sufficient to check that
    \begin{equation}
        \langle \alpha_+,\alpha_- \rangle = \sum_{\beta_n \in bs(\tilde{V}_{\mathcal{S}^{n-1},-})}
        \langle \alpha_+,\beta_n \rangle \langle \beta_n^\vee,\alpha_- \rangle,\quad
        \forall \alpha_\pm \in V_{\mathcal{S}^{n-1},\pm}
    \end{equation}
    which is true. (We used Move $(n+1)$)

    (2). For $k=n-1$, we obtain $\phi^{-1}(D^{n-1}\times S^{1})$, and by Prop.~\ref{A basis} together with Coro.~\ref{buyaotouhuangainian} we need only to verify the case of $\phi^{-1}(\mathcal{D}^{n-1}(\beta)\times S^{1})$ for an $(n-1)$ morphism $\beta$ in the representative set. 
    Then we are to prove
    \begin{equation}
        \begin{aligned}
        &Z(\mathcal{D}^{n+1}, \mathcal{S}^{n-2}\times D^{2}\cup \mathcal{D}^{n-1}(\beta)\times S^{1})
        \\
        &=\sum_{lb(b)=\mathcal{S}^{n-2}} 
        \frac{Tr(b)}{\mu(b)}
        Z(D^{n+1}\setminus D^{n-1}\times \epsilon D^{2}, 
        \mathcal{D}^{n-1}_2(b) \times \epsilon S^{1}
        \cup\mathcal{S}^{n-2}\times (D^{2}\setminus \epsilon D^{2})
        \cup \mathcal{D}^{n-1}(\beta)\times S^{1},
        P')
        \end{aligned}
    \end{equation}
    LHS equals $Tr(\beta)$ by induction hypothesis Move $(n+1)$.
    RHS can be computed by induction hypothesis Move $n$
    \begin{equation}
        \begin{aligned}
         &=
         Tr(\beta)
         \sum_{\gamma\text{ is a basis of $V_{tp(\beta)\cup_{\mathcal{S}^{n-2}}tp(\beta)}$}}
         Z(\beta\underset{n}{\times}\gamma\underset{n}{\times}\beta\underset{n}{\times}\gamma^\vee)
         \\
         &=
         Tr(\beta).
    \end{aligned}
    \end{equation}
    (In the above computation one uses implicitly Lem.~\ref{Filling: global relation can cut second for any general filling} and ambient isotope before Move $n$ such that Move $n$ is legal.)
    
    (3). For $1\le k\le n-2$ we obtain a $\mathcal{V}$-labeled manifold with support $\phi^{-1}(D^k\times S^{n-k})$, and we need only to verify the cases of $\mathcal{D}^k_1(\beta)\times S^{n-k}$ by Prop.~\ref{A basis} together with Coro.~\ref{buyaotouhuangainian}.
    Now it is sufficient to check that
    \begin{equation}
    \begin{aligned}
        &Z(\mathcal{D}^{n+1}, \mathcal{S}^{k-1}\times D^{n+1-k}\cup \mathcal{D}^k_1(\beta)\times S^{n-k})
        \\
        &=\sum_{lb(\alpha_k)=\mathcal{S}^{k-1}} 
        \frac{Tr(\alpha_k)}{\mu(\alpha_k)}
        Z(D^{n+1}\setminus D^k\times \epsilon D^{n+1-k}, 
        \\
        &\mathcal{D}^k_2(\alpha_k) \times \epsilon S^{n-k}
        \cup\mathcal{S}^{k-1}\times (D^{n+1-k}\setminus \epsilon D^{n+1-k})
        \cup \mathcal{D}^k_1(\beta)\times S^{n-k}),
        P')
    \end{aligned}
    \end{equation}
    The LHS is $Tr(\beta)$ by induction hypothesis Move $(n+1)$.
    RHS can be computed by induction hypothesis Move $(k+1)$.
    The value is nonzero only when $\alpha_k=\beta$, and equals (by Move $(n+1)$)
    \begin{equation}
    \begin{aligned}
        &=\frac{Tr(\beta)}{\mu(\beta)}\sum_{\gamma\text{ is $\beta-\beta$ bimodule}}
        \frac{Tr(\gamma)}{\mu(\gamma)}\times Tr(\gamma)
        \\
        &=\frac{Tr(\beta)}{\mu(\beta)}\mu(\beta)
        =Tr(\beta)
    \end{aligned}
    \end{equation}
    (In the above computation one uses implicitly Lem.~\ref{Filling: global relation can cut second for any general filling} and ambient isotope before Move $(k+1)$ such that Move $(k+1)$ is legal.)

    (4). For $k=0$, by the same reason we are to prove
    \begin{equation}
    \begin{aligned}
        &Z(\mathcal{D}^{n+1}, D^0(\beta)\times S^n)
        =\sum_{b} 
        \frac{Tr(b)}{\mu(b)}
        Z(D^{n+1}\setminus \epsilon D^{n+1}, 
        D^0(\beta)\times S^n \cup D^0(b) \times \epsilon S^{n},
        P')
    \end{aligned}
    \end{equation}
    LHS is $Tr(\beta)$. 
    RHS can be computed by induction hypothesis Move $1$:
    \begin{equation}
    \begin{aligned}
        &=\frac{Tr(\beta)}{\mu(\beta)}\sum_{\gamma\text{ is $\beta-\beta$ bimodule},gr(\gamma)=g}
        \frac{Tr(\gamma)}{\mu(\gamma)}\times Tr(\gamma)
        \\
        &=\frac{Tr(\beta)}{\mu(\beta)}\mu_g(\beta)
        =Tr(\beta)
    \end{aligned}
    \end{equation}
    where $g$ is the transport of this Move $1$ and we used (MI) in the last equality.
    (In the above computation one uses implicitly Lem.~\ref{Filling: global relation can cut second for any general filling} and ambient isotope before Move $1$ such that Move $1$ is legal.)
    
    We complete the induction.
\end{proof}

Now we are ready to prove the gauge invariance of the partition function. 
Namely, the partition function depends only on the G-structure, instead of the presentations. 



\begin{lem}\label{chamber inside ball}
    Suppose that $\mathcal{B}_i,i=0,1$ are two bulk vectors with a same support $B$.
    Suppose that $U\subset B$ is a G-labeled manifold with support $D^{n+1}$ and there is a boundary stratified collar (i.e., a stratified homeomorphism onto an open neighborhood such that $c(x,0)=x$):
    \begin{equation}
        c:\partial U\times(-1,+1)\to Im(c)\subset B.
    \end{equation}
    Suppose further that $\mathcal{B}_0|_{\overline{B\setminus U}}=\mathcal{B}_1|_{\overline{B\setminus U}}$.
    Then $Z(\mathcal{B}_0)=Z(\mathcal{B}_1)$.
\end{lem}
\begin{proof}
    By Thm.~\ref{transversal triangulation} and Lem.~\ref{Shellable}, one can take a full transversal triangulation of $\partial U$, and by Thm.~\ref{transversal triangulation extend to bulk} one can extend it (and take subdivisions if necessary) into a full transversal triangulation of $\mathcal{B}_i|_{B\setminus \mathring{U}}$ and of $\mathcal{B}_i|_{U}$ as $\Delta_i$, and $\Delta_1=\Delta_2$ outside $U$.
    
    One can assume that inside $Q$ it is shellable, otherwise one can take simultaneous derived subdivisions by \cite{adiprasito2017subdivisions}(Coro.I.3.10) and restore transversality by using simultaneous ambient isotopy $B\setminus U$ and separate ambient isotopy inside $U$ by Thm.~\ref{transversal triangulation}.
    One can perform the algorithm on $\partial U$ first, cutting $B$ into $P=B\setminus U$ with collar shortened and $Q=U$ with collar shortened, creating decoration $l$ on the cutting boundary and one can recast the partition function as
    \begin{equation}
        Z(\mathcal{B}_i)=\sum_l w(l) Z(P(l),\Delta_i)Z(Q(l),\Delta_i)
    \end{equation}
    $Z(P(l),\Delta_i)$ is independent of $i$ since $\Delta_0=\Delta_1$ outside $U$, and by Lem.~\ref{Shellable} $Z(Q(l),\Delta_i)=Z(l)$, the evaluation of the hypersphere functional, which is independent of $i$.

    We are done.
\end{proof}

\begin{lem}\label{chamber divided by balls}\label{App: chamber divided by balls}
    Suppose that $\mathcal{B}=(B,\mathcal{P})$ is a compact $(n+1)$-dimensional G-labeled manifold, $K\subset \mathring{B}$ is compact.
    Then there is a refinement of stratification $P'\succcurlyeq P$ such that
    \begin{itemize}
        \item
        if a chamber $D\in (P')^{(n+1)}$ is contained in $K$ ($D\subset K$), then there is a  stratified submanifold $U_D$ of $B$ with support $|U_D|\sim_{PL}D^{n+1}$ such that $\overline{D}\subset \mathring{U_D}$ 
        \item
        and there is a stratified collar of $\partial U_D$ in $(B,P')$, i.e. there is a stratified homeomorphism $c_D:\partial U_D\times (-1,+1)\to Im(c_D)\subset \mathring{B}$ onto an open neighborhood of $\partial U_D$, $c_D(x,0)=x$.
    \end{itemize}
\end{lem}
\begin{proof}
    See the appendix.
\end{proof}

\begin{thm}(Gauge invariance)\label{Gauge invariance}

    When $\mathcal{V}$ is (CF), (SS) and (MI), the partition function is invariant under gauge transformations.
\end{thm}

\begin{proof}
    Gauge transformations involve chamber gauges and refinements.           

    We first prove the refinement gauge invariance. 
    Suppose $P'\succcurlyeq P$ is a refinement. 
    Take a full transversal triangulation of $P'$ by Thm.~\ref{transversal triangulation}.
    It is also a full transversal triangulation of $P$ by Prop.~\ref{transversal of refinement}. 
    The fillings of Move $0$ do not change the transports by assumption, Move $1$ only depends on the transport which is invariant under Move $0$ and refinement, Move $k\ge 2$ depends only on the link boundary created by Move $0,...,k-1$.
    So two state sums under this triangulation with respect to bulk vectors $(B,\mathcal{M},P)$ and $(B,\mathcal{M},P')$ are equal.

    Secondly, we prove the chamber gauge invariance.
    Suppose WLOG that there is only one chamber $A$ and its parameter $a(A)$ involved in the gauge $(B,\mathcal{M},P)\to (B,\mathcal{M},a(P))$.
    By Lem.~\ref{chamber divided by balls}, one can take a refinement of stratification $P'\succcurlyeq P$, which naturally becomes a refinement gauge transformation in the sense of Def.~\ref{Gauge} by labeling new walls with trivial open holonomies.
    
    We define a chamber gauge transformation $b$ on $(B,\mathcal{M},P')$ in which 
    \begin{equation}
    b(D)=
    \begin{cases}
        a(A),\text{ if $D\subset A$}
        \\
        1,\quad\quad \text{otherwise. }
    \end{cases}
    \end{equation}
    Then one can directly see that $b(P')$ is a refinement gauge transformation of $a(P)$.
    By refinement invariance, $Z(B,\mathcal{M},P)=Z(B,\mathcal{M},P')$ and $Z(B,\mathcal{M},b(P'))=Z(B,\mathcal{M},a(P))$.

    In order to prove $Z(B,\mathcal{M},P') = Z(B,\mathcal{M},b(P'))$, one divide the chamber gauge transformation $b$ as steps of single chamber gauges.
    For each step, suppose the chamber being revised is $D$.
    Since $\overline{D}\subset \overline{A}$, by Lem.~\ref{chamber divided by balls}  we have that $\overline{D}\subset \mathring{U_D}$ for a stratified submanifold $U_D$ with support a ball.
    By Lem.~\ref{chamber inside ball}, the partition function is invariant under this transformation. 
    So composing steps we get that $Z(B,\mathcal{M},P') = Z(B,\mathcal{M},b(P'))$.
    Then we have
    \begin{equation}
        Z(B,\mathcal{M},P)=Z(B,\mathcal{M},P') = Z(B,\mathcal{M},b(P'))=Z(B,\mathcal{M},a(P))
    \end{equation}
    So the partition function $Z$ is chamber gauge invariant. 
\end{proof}

\begin{rmk}
    It seems that we proved a solid property of $Z$ without using solid properties of $\mathcal{V}$, i.e. (MI) etc. 
    Actually the solid properties of $\mathcal{V}$ is used in the proof of that the partition function is well defined, i.e. independent of the choice of the full trasnversal triangulations (Thm.~\ref{Well-defined}) and its invariance under global surgery relations (Thm.~\ref{Global relations}). 
\end{rmk}

\subsection{G-equivariant TQFT with space-time boundary}

Suppose the G-equivariant $S^n$-algebra $\mathcal{V}$ is (CF), (SS) and (MI).

Now let us extend the partition function to a G-equivariant TQFT with space-time boundary.
We use the universal reconstruction method of TQFT.

\begin{defi}
    For a G-equivariant Time Boundary $\mathcal{F}=(F,P|_F)$, define $V_{\mathcal{F},\mp}$ to be the vector space spanned by the gauge equivalence classes of G-equivariant Spacetime cobordisms with Time Boundary $\pm\mathcal{F}$. 
\end{defi}

\begin{defi}
    The partition function $Z$ defines a well defined bi-linear form on $V_{\mathcal{F},-}\times V_{\mathcal{F},+}$:
    \begin{equation}
        \langle\mathcal{B}_-, \mathcal{B}_+\rangle := Z(\mathcal{B}_-\cup_{\mathcal{F}} \mathcal{B}_+),\quad
        \forall \mathcal{B}_\mp \in V_{\mathcal{F},\mp}
    \end{equation}
    The quotient spaces are denoted by 
    \begin{equation}\label{cobordism kernel}
    \tilde{V}_{\mathcal{F},\mp}:=V_{\mathcal{F},\mp}
    /\{v\mid \langle v,w\rangle=0, \forall w \in V_{\mathcal{F},\pm}\}
    \end{equation}
\end{defi}
\begin{proof}
    One only need to notice that if $\mathcal{B}$ and $\mathcal{B}'$ are gauge equivalent as G-equivariant Spacetime cobordisms with Time Boundary $-\mathcal{F}$, and $\mathcal{A}$ and $\mathcal{A}'$ are gauge equivalent as G-equivariant Spacetime cobordisms with Time Boundary $+\mathcal{F}$, then $\mathcal{B}\cup \mathcal{A}$ and $\mathcal{B}'\cup \mathcal{A}'$ are gauge equivalent.
\end{proof}

These are the state spaces of the G-equivariant TQFT, namely the functor on the object. 
We are going to prove that these are finite dimensional vector spaces, and further Hilbert spaces when $\mathcal{V}$ is (RP).
In order to prove this, one need to take the ``square root" of the $\mathcal{S}^{k-1}$ relations. 
The topological details are treated in the following lemma.

\begin{lem}\label{Filling: Avoiding the middle}
    Suppose that $0\le k\le n-1$ and one performs global $\mathcal{S}^{k-1}$ relation to $\mathcal{U}\subset \mathcal{B}$ and we identify $\mathcal{U}=\mathcal{D}^k\times D^{n+1-k}$, then the new Space Boundary and the filling $P'$ can be chosen so that the upper half $(D^k\times (D^{n+1-k}\setminus \epsilon \mathring{D^{n+1-k}})|_{x_{n+1}\ge 0},P'|_{...})$ and the lower half $(D^k\times (D^{n+1-k}\setminus \epsilon \mathring{D^{n+1-k})}|_{x_{n+1}\le 0},P'|_{...})$ are both geometric Spacetime cobordisms with Time Boundary the middle surface $(D^k\times (D^{n-k}\setminus \epsilon \mathring{D^{n-k}})\times\{0\},P'|...)$.
\end{lem}
\begin{proof}
    Take any summand $\alpha_k$ in the $\mathcal{S}^{k-1}$ relation, and denote by $\mathcal{D}^k_2=tp(\alpha_k)$.
    By Lem.~\ref{Filling: anything}, one can choose a filling 
    \begin{equation}
        |\mathcal{C}_{k+1}|=D^k\times [+\epsilon,+1],
        \quad
        \mathcal{C}_{k+1}|_{D^k\times \{+1\}}=\mathcal{D}^k
        \quad
        \mathcal{C}_{k+1}|_{D^k\times \{+\epsilon\}}=\mathcal{D}^k_2
    \end{equation}
    when $1\le k$, and an interval with trivial open holonomy when $k=0$.
    First take $P_2=\mathcal{C}_{k+1}\times S^{n-k}$. 
    The Time Boundary is $\mathcal{C}_{k+1}\times S^{n-k-1}\times\{0\}$.
    Thus $P_2$ satisfies the geometric Spacetime cobordism condition, but without labeling $\alpha_k$. 
    (Notice that we may abuse some coordinate representations of the same elementary shape, but without confusions.)
    
    Then one label two $\alpha_k$'s at $D^k \times \epsilon^2 D^{n-k}\times\{+\epsilon\}$ and at $D^k \times \epsilon^2 D^{n-k}\times\{-\epsilon\}$ (one uses idempotency), and uses Lem.~\ref{Filling: anything} again to change only an arbitrarily small neighborhood filling of $P_2$, say $D^k \times \epsilon^2 D^{n-k}\times[\pm\epsilon,\pm\epsilon\pm\epsilon^2]$.
    This is the $P'$ satisfying the condition.
\end{proof}

\begin{thm}\label{finite dimensional}
    When the G-$S^n$-algebra is (SS), (CF), (MI), the vector space $\tilde{V}_{\mathcal{F},\mp}$ is finite dimensional. 
    If $\mathcal{V}$ is further (RP), then it induces a positive definite inner product on $\tilde{V}_{\mathcal{F},\mp}$, so it becomes a finite dimensional Hilbert space. 
\end{thm}

\begin{proof}
    Consider the Spacetime cobordism $\mathcal{F}\times[-1,+1]$, gluing along Time Boundary $\mathcal{F}$ to which can be viewed as the identity operator
    \begin{equation}
        \operatorname{id}_{\tilde{V}_{\mathcal{F},-}} = (B=F\times[-1,+1], P_B=P_F\times[-1,+1], M=\partial F\times[-1,+1], M_{\mathcal{V}}=\emptyset)
    \end{equation}
    
    Suppose one performs a global $\mathcal{S}^{k-1}$ relation (Def.~\ref{Sk-1 relations} and Thm.~\ref{Global relations}) to $\phi(D^{n+1}, S^{k-1}\times D^{n+1-k})=(U, U\cap (\partial F\times[-1,+1]))\subset (F\times[-1,+1],\partial F\times[-1,+1])$, for $0\le k\le n$. 
    Suppose $\phi$ preserves the $(n+1)$-th coordinate.
    
    We define the Spacetime cobordism 
    \begin{equation}
    \begin{aligned}
        \iota(\alpha_k, \mp) = [(F\times[-1,+1]\setminus \phi(D^{k}\times \epsilon D^{n-k+1}), \\
        \phi(\mathcal{M}(\alpha_k))\cup \partial F\times[-1,+1] \setminus \phi(S^{k-1}\times\epsilon D^{n-k+1})]
        \cap (F\times[0,\mp 1])
    \end{aligned}
    \end{equation}
    where $\mathcal{M}(\alpha_k):=\mathcal{S}^{k-1}\times (D^{n+1-k}\setminus \epsilon D^{n+1-k})\cup \mathcal{D}^k_2(\alpha_k) \times \epsilon S^{n-k}$ is the Space Boundary after the $\mathcal{S}^{k-1}$ relation. 
    
    Then for $0\le k\le n-1$ by Lem.~\ref{Filling: Avoiding the middle} one can take $\mathcal{S}^{k-1}$ relation properly such that gluing with $\iota(\alpha_k, -)$ is a well defined map
    \begin{equation}
        \iota(\alpha_k,-): \tilde{V}_{\mathcal{F},-}\to \tilde{V}_{\mathcal{F}_{\alpha_k},-}
    \end{equation}
    and similarly
    \begin{equation}
        \iota(\alpha_k,+):\tilde{V}_{\mathcal{F}_{\alpha_k,-}}\to \tilde{V}_{\mathcal{F},-}
    \end{equation}
    where $\mathcal{F}_{\alpha_k}:=(F\times\{0\}\setminus \phi(D^k\times \epsilon D^{n+1-k}), P'|_{\dots})$ is the new middle Time Boundary after performing $\mathcal{S}^{k-1}$-relation. 
    Then we define
    \begin{equation}
        \iota_{-}=\bigoplus_{\alpha_k} \sqrt{\frac{Tr(\alpha_k)}{\mu(\alpha_k)}}\iota(\alpha_k,-): \tilde{V}_{\mathcal{F},-}\to \bigoplus_{\alpha_k}\tilde{V}_{\mathcal{F}_{\alpha_k},-}
    \end{equation}
    and
    \begin{equation}
        \iota_{+}((y_{\alpha_k})_{\alpha_k})=\sum_{\alpha_k} \sqrt{\frac{Tr(\alpha_k)}{\mu(\alpha_k)}}\iota(\alpha_k,+)(y_{\alpha_k}): \bigoplus_{\alpha_k}\tilde{V}_{\mathcal{F}_{\alpha_k},-}
        \to
        \tilde{V}_{\mathcal{F},-}
    \end{equation}
    for $0\le k\le n-1$, and
    \begin{equation}
        \iota_{-}=\bigoplus_{e_i} \iota(e_i,-): \tilde{V}_{\mathcal{F},-}\to \bigoplus_{e_i}\tilde{V}_{\mathcal{F}_{e_i},-}
    \end{equation}
    and
    \begin{equation}
        \iota_{+}((y_i)_i)=\sum_{e_i^\vee} \iota(e_i^\vee,+)(y_i): \bigoplus_{e_i^\vee}\tilde{V}_{\mathcal{F}_{e_i},-}
        \to
        \tilde{V}_{\mathcal{F},-}
    \end{equation}
    for $k=n$.
    Gluing along the common Time Boundary $\mathcal{F}_{\alpha_k}$'s we obtain the cobordism created by the $\mathcal{S}^{k-1}$-relation, since when $0\le k\le n-1$ by $\alpha_k^2=\alpha_k$ we can place $\alpha_k$'s on each half. 
    That is, 
    \begin{equation}
        \iota_+ \circ \iota_- = \operatorname{id}_{\tilde{V}_{\mathcal{F},-}}
    \end{equation}
    So $\iota_-:\tilde{V}_{\mathcal{F},-}\to \bigoplus_{\alpha_k}\tilde{V}_{\mathcal{F}_{\alpha_k},-}$ is a linear embedding.

    Pick a full transversal triangulation of $\mathcal{F}$ by Thm.~\ref{transversal triangulation} and perform the above embedding map from $k=0,\dots,n$, and by the same argument as Lem.~\ref{Filling: algorithm can continue}, this procedure can continue.
    So we get an embedding of $\tilde{V}_{\mathcal{F},-}$ to a direct sum of one dimensional spaces $\tilde{V}_{\emptyset,-}=\mathbb{C}$ with nondegenerate bilinear form.
    The number of simplices of a triangulation and the number of representative irreducible $k$-morphisms $\alpha_k$ for a fixed link boundary (or transport) and basis $e_i$ are finite, so the target vector space is finite dimensional. 
    So $\tilde{V}_{\mathcal{F},-}$ is finite dimensional. 

    If $\mathcal{V}$ is (RP), then the target space is a Hilbert space.
    One can choose at the beginning that $e_i$ is a ONB, $e_i^\vee=\theta(e_i)$, and $\alpha^*=\alpha$ for any chosen $k$-morphism for $0\le k\le n-1$.
    And we required in Lem.~\ref{Filling: Avoiding the middle} the fillings of the bulk of the upper half are always the reflection of that of the lower half.
    Then the dagger of $\iota_-$, is just $\iota_-^\dagger=\iota_+$.
    The embeddings are injective isometries, so the total embedding is an isometry of vector spaces with Hermitian form. 
    Since the final space is a finite dimensional Hilbert space, 
    $\tilde{V}_{\mathcal{F},-}$ is a finite dimensional Hilbert space.
\end{proof}

\begin{defi}
    For a G-equivariant Time Boundary $\mathcal{F}=(F,P_F)$, we call $\mathcal{B}=(B,\mathcal{M},P_B)\in \tilde{V}_{\mathcal{F},-}$ a collar representative, if $B=F\times[-\epsilon,0]$ as PL manifolds, the identification of $\mathcal{F}$ is $\iota(x)=(x,0)$ and $P_B|_{\partial F\times[-\epsilon,0]}=\partial \mathcal{F}\times[-\epsilon,0]$.
\end{defi}

\begin{prop}\label{collar representative}
    For any G-equivariant Time Boundary $\mathcal{F}=(F,P_F)$, the vector space $\tilde{V}_{\mathcal{F},-}$ is spanned by collar representatives.
\end{prop}

\begin{proof}
    Take $\mathcal{B}=(B,\mathcal{M},P_B)\in \tilde{V}_{\mathcal{F},-}$.
    By definition $(B,P_B)=(F\times[-\epsilon,0],P_F\times[-\epsilon,0])\cup (B_0,P_B|_{B_0})$ and $\partial F\times [-\epsilon,0]$ has no $\mathcal{V}$-labels, where $B_0$ is a closed sub $n+1$ manifold of $B$ and $B_0\cap (F\times[-\epsilon,0])=F\times\{-\epsilon\}$. 
    
    Take a full transversal triangulation of $F\times\{-\epsilon/2\}$ by Thm.~\ref{transversal triangulation}, and extend it into a full (if necessary, take subdivisions) transversal one on $B_0\cup (F\times[-\epsilon,-\epsilon/2])$ by Thm.~\ref{transversal triangulation extend to bulk}.
    Apply $\mathcal{S}^{k-1}$ relations to $B_0\cup (F\times[-\epsilon,-\epsilon/2])$.
    Then the Spacetime cobordism $\mathcal{B}=(B,\mathcal{M},P_B)$ becomes a linear combination of collar representatives. 
\end{proof}

\begin{rmk}
    Notice that this itself does not imply Thm.~\ref{finite dimensional}, because we may have an infinite number of collar representatives.
\end{rmk}

\begin{thm}(TQFT)

    When the G-$S^n$-algebra $\mathcal{V}$ is (SS), (CF), (MI), we obtain an $(n+1)$ G-equivariant TQFT with spacetime boundary. 
    If further $\mathcal{V}$ is (RP), the G-equivariant TQFT is unitary. 
\end{thm}

\begin{proof}
    By Thm.~\ref{finite dimensional}, the vector space $\tilde{V}_{\mathcal{F},-}$ is finite dimensional. 

    The functor on morphisms is defined as follows.
    A representative $\mathcal{B}$ of a G-equivariant Spacetime cobordism $[\mathcal{B}]:\mathcal{F}_1\to \mathcal{F}_2$ acts on a vector in $\tilde{V}_{\mathcal{F}_1,-}$, i.e. a G-equivariant Spacetime cobordism with Time Boundary $\mathcal{F}_1$, by gluing along the common boundary $\mathcal{F}_1$.
    We need to prove its well definess.

    If $\mathcal{B}$ and $\mathcal{B}'$ are gauge equivalent (as G-equivariant Spacetime cobordisms), then $\mathcal{B}\cup \mathcal{F}_1$ and $\mathcal{B}'\cup \mathcal{F}_1$ are gauge equivalent.
    So it is well defined for a gauge equivalence class.
    By definition it is well defined for the kernel \eqref{cobordism kernel}.
    So $\tilde{V}_{\mathcal{F}_1,-}\to \tilde{V}_{\mathcal{F}_2,-}$ is a well defined map which is obviously linear.

    The identification $\tilde{V}_{\mathcal{F}_1,-}\otimes \tilde{V}_{\mathcal{F}_2,-}\to \tilde{V}_{\mathcal{F}_1\otimes \mathcal{F}_2,-}$ is defined via horizontal concatnation.
    It apparently respects gauge transformations and the kernel (thanks to Prop.~\ref{collar representative} the testing vector lies in the tensor product of the dual of the original vector space), so it is well defined. 
    To prove the injectivity, choose a basis $x_1^{(a)}\in \tilde{V}_{\mathcal{F}_1,-}$, $x_2^{(b)}\in \tilde{V}_{\mathcal{F}_2,-}$ and a dual basis $y_1^{(a)}\in \tilde{V}_{\mathcal{F}_1,+}$, $y_2^{(b)}\in \tilde{V}_{\mathcal{F}_2,+}$.
    Suppose that $\sum_{a,b} c_{a,b} x_1^{(a)}\otimes x_2^{(b)}=0\in \tilde{V}_{\mathcal{F}_1\otimes \mathcal{F}_2,-}$.
     We get $c_{i,j}=\langle y_1^{(i)}\sqcup y_2^{(j)},\sum_{a,b} c_{a,b} x_1^{(a)}\otimes x_2^{(b)}\rangle= 0$.
    Surjectivity follows from the collar representative Prop.~\ref{collar representative}, and that any collar representative of $\mathcal{F}_1\sqcup \mathcal{F}_2$ is a concatnation of that of $\mathcal{F}_1$ and of $\mathcal{F}_2$. 
    
    Gluing property and cylinder property are obvious.
    Thus we obtain a $(n+1)$-dimensional G-equivariant TQFT with spacetime boundary.
    If $\mathcal{V}$ is (RP), then by Thm.~\ref{finite dimensional} $\tilde{V}_{\mathcal{F},-}$ is a Hilbert space. 
    Gluing the cobordism is apparently a dagger functor.
    So the G-equivariant TQFT is unitary. 
\end{proof}

%% file: Examples.tex
\section{Examples}

In this section we show two classes of examples for G-equivariant Alterfold TQFTs.

\subsection{HQFTs}

When $n=1$, G-equivariant Alterfold TQFT recover the $1+1$ HQFT with target $K(G,1)$ in \cite{turaev1999homotopy}.

When $n=2$, we recover the $2+1$ HQFT with target $K(G,1)$ in \cite{turaev20123,sozer2023monoidal}.

\begin{thm}
For a strict faithful spherical G-fusion category $\mathcal{C}=\oplus_g \mathcal{C}_g$, one can use it to construct a G-$S^2$-algebra $\mathcal{V}_{\mathcal{C}}$.
Then the G-equivariant Spacetime TQFT $Z^{\text{Alt}}_{\mathcal{V}_\mathcal{C}}$ restricted to the sub category without Space Boundary is monoidally naturally isomorphic to the Turaev-Virelizier HQFT $Z^{\text{TV},K(G,1)}_{\mathcal{C}}$ constructed by $\mathcal{C}$.
\end{thm}

The proof is straightforward once we show the dictionary.

Choose simple representatives $I_g\subset Obj(\mathcal{C}_g)$ and write
\begin{equation}
d_X=\dim_{\mathcal C}(X),\qquad
D=\dim(\mathcal C_e)=\sum_{X\in I_e}d_X^2\ne0.
\end{equation}
Let $P$ be the dual $2$-skeleton of a triangulation $\Delta$ of a closed
oriented $3$-manifold $B$, with region labels $\ell$ representing the
fixed background. Write $r_a,e_f,v_t$ for the dual region, edge, and
vertex of an edge $a$, triangle $f$, and tetrahedron $t$.
Use the same cyclic order and duality conventions on both sides.

\begin{table}[H]
\centering\small
\renewcommand{\arraystretch}{1.3}
\setlength{\tabcolsep}{4pt}
\begin{tabular}{@{}p{.08\textwidth}p{.18\textwidth}p{.31\textwidth}p{.31\textwidth}@{}}
\toprule
Step & Primal / dual cell & Alterfold contribution & Turaev--Virelizier contribution \\
\midrule
Move $0$ 
& 
Vertex / complementary $P$-ball 
&
$\sum_{\alpha_0}\frac{Tr(\alpha_0)}{\mu(\alpha_0)}(D^3\setminus \epsilon D^3)=\frac{1}{D}$
&
One factor $D^{-1}$ per complementary ball:
$D^{-|P|}=D^{-|\Delta_0|}$. 
\\
\addlinespace
Move $1$ 
& 
Edge $a$ / dual disk region $r_a$ 
&
$\sum_{gr(\alpha_1)=g}\frac{Tr(\alpha_1)}{\mu(\alpha_1)}(D^3\setminus \mathcal{D}^1(\alpha_1)\times \epsilon D^2)
=\sum_{X\in I_{\ell(r_a)}}d_X(D^3\setminus \mathcal{D}^1(X)\times \epsilon D^2)$
&
Color $r_a$ by $X\in I_{\ell(r_a)}$ and insert
$d_X^{\chi(r_a)}=d_X$. \\
\addlinespace
Move $2$ 
& 
Triangle $f$ / dual edge $e_f$ 
&
$\sum_{lb(\alpha_2)=\mathcal{S}^1}(D^3\setminus D^2\times \epsilon D^1,\mathcal{D}^2(\alpha_2)\times\{-\epsilon\}\cup \mathcal{D}^2(\alpha_2^\vee)\times\{+\epsilon\})$
&
Contract incident tensors by $*_{e_f}$. This is the same dual-basis sum, with no extra dimension factor. \\
\addlinespace
Move $3$ 
& 
Tetrahedral ball $t$/ dual vertex $v_t$ 
&
Evaluate the decorated boundary sphere by the sphere functional of $\mathcal V_{\mathcal C}$. 
&
Evaluate the colored link graph $\Gamma_{v_t}\subset S^2$ by $F_{\mathcal C}(\Gamma_{v_t})$. \\
\bottomrule
\end{tabular}
\caption{The algorithm compared with
\cite{turaev20123}.}
\label{tab:Alt-TV-dictionary}
\end{table}

Here $|P|$ denotes the number of complementary balls.

\subsection{Quon language}

Quon language \cite{liu2017quon,liu2019quon,kang20252d,feng2025quon} for quantum information is an example of $\mathbb{Z}_2$-equivariant $(2+1)$-dimensional Alterfold TQFT.

Take $G=\mathbb{Z}_2=\{0,1\}$.
Inside the tangle $1$-labeled defects are strings, and $0$-labeled strings are double-strings, and on the $\mathbb{Z}_2$-equivariant empty disk $1$-labeled defects are stratification points, and $0$-labeled strings are a pair of two stratification points.
The $\mathbb{Z}_2-D^n$-algebra is defined as follows.
Take
\begin{equation}
\begin{aligned}
    W(\mathcal{S}):=Free_\mathbb{C}\{\text{non-intersecting strings, the set of endpoints of open strings}\\
    \text{are the stratification point of $\mathcal{S}$}\}/\{\text{closed string $=\sqrt{2}$.}\}
\end{aligned}
\end{equation}
and the hypersphere functional
\begin{equation}
    Z(\text{a diagram on $S^2$}):=\sqrt{2}^{\text{number of closed strings}}
\end{equation}
Define a kernel
\begin{equation}
    N(\mathcal{S}):=\{x\in W(\mathcal{S})\mid \forall y\in W(\mathcal{S}^*),Z(x\cup_\mathcal{S} y)=0\}
\end{equation}
which is an ideal w.r.t. the natural tangle operations of $(W,Z)$.
Then the $\mathbb{Z}_2-S^2$-algebra is defined as
\begin{equation}
    V(\mathcal{S}):=W(\mathcal{S})/N(\mathcal{S})
\end{equation}
and $Z$ descends naturally.

The emergent $2$-category is the Ising fusion category
\begin{equation}
    \mathcal{C}=\{1,e,\tau\}
\end{equation}
when considered as a $2$-category with a unique $0$-morphism. 
In fact, minimal idempotents fall into three annular equivalence classes:
    \[
1=
\tikz[baseline=(current bounding box.center),scale=.8]{
  \draw[blue] (-1,-.7) rectangle (1,.7);
}
\quad\quad
\tau=
\tikz[baseline=(current bounding box.center),scale=.8]{
  \draw[blue] (-1,-.7) rectangle (1,.7);
  \draw (0,-.7) -- (0,.7);
}
\quad\quad
e=
\tikz[baseline=(current bounding box.center),scale=.8]{
  \draw[blue] (-1,-.7) rectangle (1,.7);
  \draw (-.5,.7) --(-.5,-.7);
  \draw (.5,-.7) --(.5,.7);
}
-\frac{1}{\sqrt{2}}
\tikz[baseline=(current bounding box.center),scale=.8]{
  \draw[blue] (-1,-.7) rectangle (1,.7);
  \draw (-.5,.7) arc[start angle=180,end angle=360,radius=.5];
  \draw (-.5,-.7) arc[start angle=180,end angle=0,radius=.5];
}
\]
and $gr(1)=gr(e)=0$, $gr(\tau)=1$.
The (MI) condition reads
\begin{equation}
    1^2+1^2=(\sqrt{2})^2
\end{equation}

The $\mathbb{Z}_2$-equivariant Alterfold TQFT is the Quon language.

The same holds true for arbitrary Tambara-Yamagami inputs, which recovers \cite{liu2017quon}.

%% file: app12.tex
\providecommand{\id}{\mathrm{id}}
\providecommand{\Int}{\operatorname{int}}
\providecommand{\lk}{\operatorname{lk}}
\providecommand{\st}{\operatorname{st}}
\providecommand{\ost}{\operatorname{ost}}
\appendix

\section{PL topology}

Each subsection proves one statement deferred to this appendix.
Auxiliary results are placed with the proof that uses them.
We work in the PL category, including jointly PL isotopies.
Stratified maps preserve the filtration, specified boundary and,
when present, wall labels and coorientations.

For an ambient $m$-manifold the boundary filtration is
$(\partial M)^a=\partial M\cap M^{a+1}$.
Collars are required to exist; a choice of collar is not part of a bulk
vector. Removing a surgery neighborhood means taking the closed
complementary manifold with its new boundary.

In the numerical arguments $n\ge2$ and CF, SS and MI are understood.
We write $E(B,K)$ for the algorithmic sum and $Z$ for the given sphere
functional. The numerical input to the shellable-tail and stellar
comparisons is the adjacent identity proved in the main text.
Neither global relations nor general triangulation independence is
used to prove the absolute stellar comparison.

\subsection{Theorem~\ref{App: transversal triangulation}: transverse positioning}
\label{appsec:position}

The following elementary tools are used throughout the appendix.

\begin{lem}[The filtration terms are closed PL subspaces]
\label{app:closed-filtration}
The recursively filtered PL charts in the main text give closed PL
filtration terms. In the compact case they admit a common adapted
triangulation.
\end{lem}
\begin{proof}
Locally the terms are products of Euclidean or half-space factors
with cones on the corresponding terms in a compact link. Induction
on the dimension of the link proves that they are PL subspaces.
Closedness is part of the main text's filtration convention, and
the PL subspace assertion is local in the ambient polyhedron.
For a common adapted triangulation, triangulate the finitely many
closed PL subspaces successively, subdividing the previously
retained subcomplexes at each step; compactness makes this finite.
\end{proof}

\begin{lem}[A product slice preserves the filtered link]
\label{app:slice}
Suppose a finite PL filtration is locally a product $(Y;Y^i)\times
\mathbb R^a$. Suppose that $(y,0)$ has a filtered chart
$\mathbb R^k\times cL$, and that $\{y\}\times J$ lies in its $k$-stratum
for some neighborhood $J$ of $0$ in $\mathbb R^a$.
Then $a\le k$, and the filtration on $Y$ has a chart
$\mathbb R^{k-a}\times cL$ at $y$, with the same filtered link $L$.
The assertion also holds with a specified boundary, provided the
directions removed are tangent to that boundary.
\end{lem}
\begin{proof}
The case $a=0$ is immediate. Otherwise triangulate a small neighborhood
so that the given filtered chart is affine on each simplex and the
straight fiber $\{y\}\times J$ is a subcomplex. Choose a point in the
relative interior of an $a$-simplex of that fiber. Every simplex meeting
a sufficiently small neighborhood of this point contains that face.
Its local defining inequalities are invariant under translation along
the face, as long as the translated points stay in the neighborhood.

The affine pieces of the chart agree on this common face. Their linear
parts therefore agree on its $a$-dimensional tangent space. Its image
is an $a$-dimensional subspace $T\subset\mathbb R^k$, since the fiber lies
in the center stratum. Locally the chart consequently intertwines
translations along the fiber with translations along $T$.
Take a linear transverse slice on each side and project along these
directions. The resulting map and its inverse are PL, preserve every
filtration term, and identify the slice with
$(\mathbb R^k/T)\times cL$. Moving to the chosen point on the product
fiber has not changed the germ of $Y$ at $y$.

For a boundary chart perform the same construction on the boundary
factor; its tangent directions do not contain the inward half-line.
The half-line and the entire filtered link survive the quotient.
The same argument without filtration proves cancellation of a
Euclidean product factor from a PL manifold, also as a pair with its
intrinsic boundary.
\end{proof}

\begin{lem}[Simultaneous collar tools]
\label{app:collars}
Let $A\subset X$ be compact PL spaces and let $X_1,\ldots,X_s$ be closed
PL subspaces of $X$. If $A$ is locally collared in $X$ by charts that
simultaneously preserve all $X_i$, then it has a simultaneous global
PL collar. Two such collars agree, after shortening and an ambient PL
isotopy fixed on $A$, supported in an arbitrarily small neighborhood of
$A$, and preserving all $X_i$.
\end{lem}
\begin{proof}
For existence use the finite local-to-global collar construction of
\cite[Theorem 2.25, printed pp.24--25; PDF pp.32--33]{rourke2012introduction}.
In a local collar every marked set is $(A\cap X_i)\times[0,\epsilon)$.
The successive push-ins change only the interval coordinate. Hence
they preserve all the marked sets at once. Shrinking on overlaps and
composing finitely many such push-ins gives the global collar.

For uniqueness apply the collar-straightening construction underlying
\cite[Lemma 4.23 and Chapter 4, Exercise (4), printed pp.58--59;
PDF pp.66--67]{rourke2012introduction}. Triangulate the two collar neighborhoods and the
finite marked family simultaneously. The derived-cell moves preserve
their old cells; every marked subcomplex is therefore preserved by
the same moves. Keep the old zero end and the exterior of the chosen
neighborhood fixed. After straightening, in the first collar the
second has the form $e(x,r)=(b_r(x),r)$, where $b_0=\id$ and each $b_r$
preserves all $A\cap X_i$.

For completeness, the remaining horizontal motion can be removed
without a non-PL interpolation. Choose a small $\delta>0$ and put
\[
 \kappa(r)=\begin{cases}r,&0\le r\le\delta,\\
 2\delta-r,&\delta\le r\le2\delta,\\0,&r\ge2\delta.
 \end{cases}
\]
The maps $(x,r)\mapsto(b_{\min(t,\kappa(r))}^{-1}(x),r)$,
$0\le t\le\delta$, form a jointly PL isotopy. They are the identity
at $r=0$ and for $r\ge2\delta$, preserve every mark, and straighten
the short collar. Compose this with the preceding straightening,
in that order. The common subdivisions and the cell-preserving
construction are why the argument works for a finite family, rather
than only separately for each mark.
\end{proof}

We use a \textbf{ball complex} in the sense of
\cite[I, \S1, printed pp.4--5; PDF p.6]{buoncristiano1976geometric}: its cells are PL balls,
their boundaries are unions of cells, and intersections are unions of
common faces. Subdividing a cell boundary need not subdivide the
interior of that cell. This convention is useful in the next argument.

\begin{lem}[Collar criterion for transversality]
\label{app:criterion}
Let $P$ be a stratification on a compact PL manifold $M$, and let $K$
be a finite combinatorial triangulation. Then $P\pitchfork K$ if and
only if, for every positive-dimensional simplex $\sigma\in K$,
$\partial\sigma$ has a collar in $\sigma$ simultaneously preserving
all $M^i\cap\sigma$.
\end{lem}
\begin{proof}
In a transverse chart the simplex and its boundary are
$(Q,\partial Q)\times cL$. A local collar of $\partial Q$ in $Q$,
multiplied by $cL$, preserves every filtration term. Globalize on the
compact $\partial\sigma$ by Lemma~\ref{app:collars}.

Conversely the simultaneous collars give, near a point in the
relative interior of an $r$-cell $\alpha$, a common transverse-star
chart
\begin{equation}\label{app:star-chart}
 (U;M^i\cap U)=(V;Y^i)\times T,\qquad
 \sigma\cap U=V\times T_\sigma.
\end{equation}
Here $T$ is a neighborhood of the vertex in the transverse star of
$\alpha$, and $T_\sigma$ is its part in an incident simplex $\sigma$.
To construct the chart, start with a small $V\subset\Int\alpha$ and
induct over incident cells. If the boundary chart of a cell is already
$\Phi$, extend it by $(v,z,t)\mapsto c_\sigma(\Phi(v,z),t)$ using
\emph{one collar of its whole boundary}. At $t=0$ this is the old
chart, so all old faces are unchanged. Membership in every marked
set depends only on $v$. The finitely many collar neighborhoods
contain a common smaller neighborhood of the point. This is the
transverse-star construction in \cite[II, Lemma 1.2, printed p.21;
PDF p.14, right]{buoncristiano1976geometric}, with the marks retained.

If $\dim M=m$ and the point is in its $k$-stratum, at an interior
point $T$ is locally $\mathbb R^{m-r}$. Apply Lemma~\ref{app:slice}
to obtain $(V;Y^i)=\mathbb R^{k+r-m}\times cL$, retaining the original
link. At a boundary point use the ambient filtered collar and apply
the same lemma on the boundary, removing $m-r-1$ directions from a
$(k-1)$-stratum. The exponent is again $k+r-m$.

For a $j$-simplex $\sigma$ containing $\alpha$, its transverse part
$T_\sigma$ is intrinsically a $(j-r)$-dimensional PL half-space
neighborhood, or a point if $j=r$. Its intrinsic boundary corresponds
to $\partial\sigma$. Set
\[
 \mathbb P_k=\mathbb R^{k+r-m}\times T,
 \qquad Q=\mathbb R^{k+r-m}\times T_\sigma.
\]
This is exactly the required chart, with $\dim Q=k+j-m$, and preserves
both specified boundaries. No simultaneous linearization of the
incident sectors in $T$ is needed.
\end{proof}

\begin{coro}[Amalgamation of collared cells]
\label{app:amalgamation}
The construction in Lemma~\ref{app:criterion} applies to ball complexes.
If a PL manifold block and its boundary are unions of cells with
simultaneous marked collars, its entire boundary also has a
simultaneous marked collar.
\end{coro}
\begin{proof}
Indeed,
restrict \eqref{app:star-chart} to those cells. The block is
$V\times T_D$, with its boundary $V\times\partial T_D$;
Euclidean-factor cancellation in Lemma~\ref{app:slice} identifies
$T_D$ as a manifold with boundary. Its local collar multiplied by
$(V;Y^i)$ gives the required local marked collar, and
Lemma~\ref{app:collars} globalizes it.
\end{proof}

\medskip
\noindent\textbf{Positioning with a prescribed boundary.}

\begin{lem}[Relative positioning]
\label{app:relative-position}
Let $P$ be a stratification on a compact PL manifold $M$, and let
$\Delta$ be a finite combinatorial triangulation. If
$P|_{\partial M}\pitchfork\Delta|_{\partial M}$, there is an ambient
PL isotopy $g_t$, fixed pointwise on $\partial M$, such that
$g_{1*}P\pitchfork\Delta$.
\end{lem}
\begin{proof}
Write $B=\partial M$ and $K=\Delta|_B$. We give the simultaneous
marked version of the dual-cell and subdivision construction of
\cite[II, Theorem 2.1 and the proof of Theorem 3.2,
printed pp.23,30--31; PDF pp.15,19]{buoncristiano1976geometric}. The marks are the finitely
many closed subpolyhedra $M^i$. They need not themselves be manifolds.

\emph{An initial collared ball complex.}
The ambient boundary has a simultaneous filtered collar. By
Lemma~\ref{app:criterion}, each $\partial\sigma\subset\sigma$ for
$\sigma\in K$ also has a simultaneous collar. Triangulate the marks
on each boundary cell, using these collars near its boundary, and
extend to the interior of $M$. Different subdivisions at the ends of
a collar can be joined by triangulating the intervening prism,
cell by cell, relative to its already triangulated boundary.

Take the derived dual cells of this adapted decomposition. For an
interior face $\alpha$, its dual is $b\alpha*\dot D(\alpha)$. Its
intersection with a marked subcomplex is empty if $\alpha$ is absent
from that subcomplex; otherwise it is the cone, with vertex $b\alpha$,
on the intersection with $\dot D(\alpha)$. The boundary collar of the
dual therefore preserves every mark. These duals are PL balls; see
\cite[Proposition 4.1, printed p.283; PDF p.16]{mccrory1975cone}.
At $B$ the prearranged product decomposition gives a boundary dual
times an interval. The side collars and the end collars fit at their
corners: the PL change
\[
 (s,t)\longmapsto(s-t,\min(s,t))
\]
straightens the boundary of a quadrant, and does not change any mark
in the other factors.

On $B$, use the duals \emph{inside each original cell of $K$}, as in
the cited proof of BRS II.3.2. This makes the resulting boundary
complex refine $K$. More explicitly, its local duals are order
complexes of $\{\beta\subset\sigma:\alpha\le\beta\}$, with
$\sigma\in K$. A chain lies in a local dual containing its maximum
cell. Intersections are unions of local duals obtained by choosing
the minimum of the chain and a common original face containing its
maximum. Thus these are compatible cells, not independently chosen
decompositions. Only the boundaries of the interior duals need to
be subdivided. Denote the resulting ball complex by $J$.
Every cell of $J$ has a simultaneous boundary collar, and $J|_B$
refines $K$.

\emph{Passing to a common subdivision.}
Choose a common subdivision $L$ of $J$ and $\Delta$.
The subdivision construction of BRS II.2.1 first refines $L$ so that
in each original ball its subdivision is cylindrical outside an
interior smaller ball \cite[I, Lemma 2.3, printed pp.8--9; PDF p.8]{buoncristiano1976geometric}.
It then reverses the cylindrical collapses, followed by one elementary
collapse. In each elementary extension, a retained ball $A$ acquires
the product $A\times I$, attached along $A\times0$; use the same
product for every mark. Starting with the identity projection on
$M$, these extensions have homeomorphic projections throughout:
each new source is identified with precisely the new target ball,
relative to the retained face. Amalgamation preserves the marked
collars by the argument after Lemma~\ref{app:criterion}.

Here is the compatibility check for this construction, which is
important in the marked case. Proceed over original $J$-cells by
dimension. For a source ball $Q$ put $P_Q=\partial Q$. Suppose the
already constructed side cylinder is identified with $P_Q\times I$,
with its old faces and top end fixed. The cylindrical shell at the
bottom is then another product on this same bottom data. Write their
union as $P_Q\times[-2,0]$, where $[-1,0]$ is the side cylinder.
Adjoin a short simultaneous collar $P_Q\times[0,\epsilon]$ in $Q$.
All marks on this enlarged shell have the form
$(P_Q\cap M^i)\times[-2,\epsilon]$.

The last elementary extension has source $A\times I_u$, where
$A=Q\cup P_Q\times[-2,0]$. Its free face is
$(A\times1)\cup(\partial A\times I_u)$, not merely $A\times1$.
Choose a PL homeomorphism of rectangles
\[
 \psi:[-2,\epsilon]\times I_u\longrightarrow[0,\epsilon]\times I_t
\]
with boundary values
\[
 \psi(\epsilon,u)=(\epsilon,1-u),\quad
 \psi(r,0)=(r,1)\ (0\le r\le\epsilon),\quad
 \psi(r,0)=(0,r+1)\ (-1\le r\le0).
\]
Map the remaining boundary arc monotonically to the bottom edge and
extend by coning a triangulated boundary to an interior point.
Use $\id_{P_Q}\times\psi$ on the shell and $(q,u)\mapsto(q,1-u)$
on the remaining core. The maps agree at $r=\epsilon$, preserve all
marks, agree with the prescribed old side-cylinder identification,
and send the \emph{whole} output end to $Q\times0$. Thus these maps
glue over the original cells. For a zero-cell the construction is
just time reversal.

Taking the output restriction of this cylinder identification, and
then the output projection, gives a PL homeomorphism $h:M\to M$
preserving every original $J$-cell. The transported marks have
simultaneous collars on the refined output complex. Amalgamate first
to $L$, then to $\Delta$. Lemma~\ref{app:criterion} gives
$h_*P\pitchfork\Delta$. The cylinder identification need not preserve
its time coordinate; it has not been used as an isotopy.

\emph{Making the motion relative to $B$.}
A cell-preserving homeomorphism of a ball complex is isotopic to the
identity through cell-preserving homeomorphisms, by
\cite[Corollary 3.23, printed p.38; PDF p.46]{rourke2012introduction}.
Choose such an isotopy $h_t$. On $B$ it preserves every $K$-simplex.
Extend $(h_t|_B)^{-1}$ over the cells of $\Delta$, in increasing
dimension, using the ball isotopy-extension theorem
\cite[Proposition 3.22(ii), printed pp.37--38; PDF pp.45--46]{rourke2012introduction}.
Use the already determined maps on common faces. The extension $a_t$
preserves each $\Delta$-simplex. Hence $g_t=a_th_t$ fixes $B$
pointwise, and $a_1$ preserves the transversality of $h_*P$ to
$\Delta$. This proves the claim.
\end{proof}

\begin{lem}[Closed-manifold positioning]
\label{app:closed-position}
On a compact closed PL manifold $M$, any recursive PL stratification
can be moved by an ambient PL isotopy to be transverse to a given
finite combinatorial triangulation $\Delta$.
\end{lem}
\begin{proof}
This is the empty-boundary case of Lemma~\ref{app:relative-position}.
\end{proof}

\begin{thm}[Transverse triangulations; Theorem~\ref{App: transversal triangulation}]
\label{app:position}
Let $P$ be a stratification on a compact PL manifold $M$, and let
$\Delta$ be a combinatorial triangulation. There is an ambient PL
isotopy $f_t$ such that $P\pitchfork f_1(\Delta)$.
If $P|_{\partial M}\pitchfork\Delta|_{\partial M}$, the isotopy can
be fixed pointwise on $\partial M$.
\end{thm}
\begin{proof}
First apply Lemma~\ref{app:closed-position} on the closed manifold
$B=\partial M$, with its dimension-renumbered filtration, to obtain
$b_t$ with $b_{1*}(P|_B)\pitchfork\Delta|_B$.
Take $b_t=\id$ when the boundary is already transverse.

Extend this motion to $M$ using a collar $c:B\times[0,\epsilon]\to M$.
Choose a PL function $\rho$ equal to $1$ near $0$ and to $0$ near
$\epsilon$, and set
\[
 H_t(c(x,s))=c(b_{\min(t,\rho(s))}(x),s).
\]
Outside the collar take the identity. The formula and its inverse
are jointly PL, and the inverse is obtained by replacing $b_u$ by
$b_u^{-1}$. Now $H_{1*}P$ has transverse boundary. Apply
Lemma~\ref{app:relative-position} to it and $\Delta$, obtaining $g_t$
fixed on $B$ and $g_{1*}H_{1*}P\pitchfork\Delta$.
Concatenate $H_t$ and $g_tH_1$, and take the inverse isotopy. Its
endpoint is $f_1=(g_1H_1)^{-1}$, so
$P\pitchfork f_1(\Delta)$. The asserted relative conclusion follows
from the choice $H_t=\id$ when the boundary was already transverse.
\end{proof}

\subsection{Theorem~\ref{App: transversal triangulation extend to bulk}: extension from a boundary face}
\label{appsec:extension}

\begin{lem}[Exact extension of a boundary triangulation]
\label{app:exact-extension}
A prescribed combinatorial triangulation of the boundary of a compact
PL manifold extends to a combinatorial triangulation of the manifold,
without subdividing the prescribed boundary triangulation.
If the boundary triangulation is transverse to a stratification,
the extension can be chosen transverse too.
\end{lem}
\begin{proof}
Start with any triangulation $T$ of the manifold. Choose a common
subdivision $L$ of $T|_{\partial M}$ and the prescribed boundary
triangulation $K$. Extend this subdivision over $T$ by
\cite[Lemma 2.17, printed p.19; PDF p.27]{rourke2012introduction}.
Attach $\partial M\times I$ with $L$ at its inner end and $K$ at
its outer end. Triangulate each $\sigma\times I$, $\sigma\in K$,
in increasing dimension: its boundary is already triangulated, so
cone that boundary to an interior point. The prisms agree on all
common faces. Absorb the added cylinder into a boundary collar by
a PL reparameterization of its interval, keeping the boundary base
point fixed. The resulting triangulation restricts to precisely $K$.

For the transverse assertion apply the relative part of
Theorem~\ref{app:position}. It fixes every point of the boundary and
therefore does not change $K$.
\end{proof}

\begin{thm}[Extension from a boundary face; Theorem~\ref{App: transversal triangulation extend to bulk}]
\label{app:face-extension}
Let $(B,P)$, $F$, and $M=\overline{\partial B\setminus F}$ satisfy
the assumptions of Theorem~\ref{App: transversal triangulation extend to bulk}. In particular $F$ and $M$ are compact
PL boundary faces, $F\cap M=\partial F=\partial M$, and $F$ has the
specified stratified corner collar in $B$. A transverse triangulation
of $(F,P|_F)$ extends to a transverse triangulation of $(B,P)$.
\end{thm}
\begin{proof}
Put $E=F\cap M$. The relevant filtrations are
\[
 F^a=F\cap B^{a+1},\qquad M^a=M\cap B^{a+1},\qquad
 E^a=E\cap B^{a+2}.
\]
We first verify that these are the required stratified structures.
At a point of $E$, use the collar
$c:F\times(-1,0]\to B$ at a negative parameter $t_0$.
The local pair is
\[
 (B,\partial B;B^i)=
 (F,\partial F;F\cap B^i)\times(t_0-\delta,t_0+\delta).
\]
The boundary identification follows because $c$ is a homeomorphism
onto an open neighborhood and hence preserves intrinsic boundary.
The interval lies in the stratum direction. Lemma~\ref{app:slice}
gives the half-space filtered charts on $F$, with dimensions lowered
by one and the same links. Lemma~\ref{app:collars} supplies their
common collar along compact $E$.

The side $c(E\times(-1,0])$ is a collar on the $M$ side of $E$:
it is exactly the other boundary face of the source corner, and it
covers a neighborhood in $M$. It preserves the filtration. Away
from $E$, the structure on $M$ is the open restriction of the
boundary stratification of $B$. Thus $M$ too is a stratified
manifold with boundary, with the same induced filtration on $E$.

Let $K_F$ be the prescribed triangulation, and put $K_E=K_F|_E$.
Restricting a transverse chart to the genuine boundary decreases
the stratum and ambient dimensions by one and leaves the link
unchanged. Hence $K_E$ is transverse to $E$. By
Lemma~\ref{app:exact-extension} it extends to a transverse
triangulation $K_M$ of $M$. We may arrange that the abstract complexes
glue along exactly $K_E$: any needed preliminary subdivision on the
$M$ side is relative to $E$, and is followed by relative positioning.

The union $K_F\cup_{K_E}K_M$ triangulates $\partial B$.
To check transversality at the seam, join the two filtered collars
of $E$ to a two-sided collar. The collar obtained from the transverse
chart of any incident simplex can be matched to this collar by
Lemma~\ref{app:collars}. Extend its chart across $E$. Its stratum
half-space becomes a full space; the submanifold $Q$ remains on the
original side, and the entire cone factor $cL$ is unchanged.
Its intrinsic boundary is still $\partial Q\times cL$. This is
allowed by the main text's definition, which does not require a
proper embedding of $Q$ in the stratum direction.

The boundary triangulation is therefore transverse everywhere.
Apply Lemma~\ref{app:exact-extension} once more, this time to $B$.
Its boundary restriction, and in particular the original $K_F$,
is unchanged. Empty faces or an empty seam simply omit the
corresponding steps.
\end{proof}

\subsection{Proposition~\ref{App: transversal of refinement}: forgetting a refinement}
\label{appsec:refinement}

\begin{prop}[Transversality descends under refinement; Proposition~\ref{App: transversal of refinement}]
\label{app:refinement}
Let $P'\succcurlyeq P$ be stratifications on the same PL manifold.
If a PL submanifold $N$ is transverse to $P'$, it is transverse to $P$.
In particular every $P'$-transverse triangulation is $P$-transverse.
\end{prop}
\begin{proof}
At $x\in N$ let the fine and coarse strata have dimensions $r,s$,
and put $\dim N=j$, $\dim M=m$. A fine transverse chart gives
\[
 (P',N)=(\mathbb P_r\times cL',\ Q'\times cL'),\qquad
 \dim Q'=j+r-m.
\]
Shrink the stratum factor to a connected neighborhood. Each coarse
filtration term is a union of connected fine strata. Each connected
local fine-stratum piece lies in one such global fine stratum;
membership in a coarse term consequently does not depend on the
$\mathbb P_r$ coordinate. Thus the coarse filtration in this same
chart is $\mathbb P_r\times Y^i$.

At an interior point Lemma~\ref{app:slice}, applied to a coarse
conical chart, gives
$(cL';Y^i)=\mathbb R^{s-r}\times cL$ with the coarse filtered link $L$.
At a boundary point apply it on the true boundary, removing $r-1$
directions from the coarse boundary stratum of dimension $s-1$;
the difference is still $s-r$. Multiply back the fine stratum
factor. The coarse transverse chart is
\[
 (P,N)=(\mathbb P_s\times cL,
 (Q'\times\mathbb R^{s-r})\times cL).
\]
The new stratum-direction submanifold has dimension $j+s-m$, and
its boundary is $\partial Q'\times\mathbb R^{s-r}$. The true ambient
boundary is also preserved. Apply the assertion separately to
every closed simplex of the triangulation.
\end{proof}

\subsection{Theorem~\ref{App: transversal subdivision preparation}: preparation on a subdivision}
\label{appsec:preparation}

\begin{lem}[Preparation on a subdivision]
\label{app:preparation}
Suppose $P\pitchfork K$, and let $L$ be a locally finite combinatorial
subdivision of the locally finite triangulation $K$. A $K$-cell-preserving ambient isotopy moves $P$
to a stratification transverse to $L$. After prescribing ordinary
PL collars $d_\sigma$ of the boundaries of the positive-dimensional
$L$-simplices, a further $L$-cell-preserving isotopy makes every
filtration term a product along sufficiently short $d_\sigma$.
If two starting stratifications agree on $\partial M$, the two
preparation paths can be chosen to agree there at every time.
\end{lem}
\begin{proof}
On a $d$-simplex $\sigma\in K$ use the dimension-renumbered filtration
\[
 (P|_\sigma)^a=M^{a+m-d}\cap\sigma,\qquad \dim M=m.
\]
Transversality implies that $M^{m-d-1}\cap\sigma=\varnothing$ and
gives the recursive charts on this manifold with boundary.
Process the coarse simplices by dimension. Once their boundaries
are transverse to $L$, apply Lemma~\ref{app:relative-position} inside
each simplex, fixed on its boundary. Extend the resulting isotopy
over higher-dimensional simplices by RS Proposition 3.22(ii).
It leaves the already processed faces unchanged. This gives the
first assertion.

For the collar assertion again process by dimension. Extend the
chosen motion of $\partial\sigma$ over $\sigma$, calling the extension
$e$. The image filtration has a simultaneous collar $c$ whose
bottom parameterization is the identity. Ordinary collar uniqueness
gives a homeomorphism $u$, fixed on $\partial\sigma$, with $uc=d_\sigma$
on a short collar. Thus $ue_*P$ is a product there. This use of
ordinary collar uniqueness does not claim that $u$ preserves the
old filtration. On old faces the maps agree by construction.
The final homeomorphism preserves $L$-cells and is isotopic to the
identity through such homeomorphisms by RS Corollary 3.23.
All preparation stages therefore preserve transversality to $K$.

For two inputs with the same boundary, do the preparation on that
boundary once, extend it to both manifolds, and then do the interior
preparations relative to the boundary. Synchronizing the parameters
gives the last assertion.

The construction also proves the stated locally finite version.
The restriction of $L$ to a closed $K$-simplex is finite, by compactness
of that simplex and local finiteness of $L$. At each dimension the
relative constructions on distinct cell interiors agree on their
already fixed boundaries. Glue them at the same time parameter,
then extend over the higher cells. There are only $m+1$ dimension
stages. Every point has a neighborhood meeting finitely many cells,
so the glued level-preserving map and its inverse are locally PL.
No globally uniform collar width is required. Hence no compactness
hypothesis is being added to main-text Theorem~\ref{App: transversal subdivision preparation}.
\end{proof}

\subsection{Theorem~\ref{App: transversal isotopy step by step}: localization in the original vertex stars}
\label{appsec:stars}

In this subsection $K$ is the original triangulation, whereas $L$
is an auxiliary sufficiently fine triangulation. The isotopies in
the conclusion will be supported in open stars of \emph{vertices
of $K$}. Subdivision of $K$ is a proof device, not a change in the
statement. We need control of the exact endpoint homeomorphism,
not merely of the resulting stratification.

The collars can be chosen to make the preceding preparation compatible
with the \emph{fixed star projections} of $L$, rather than merely with
unspecified normal coordinates. We record an explicit choice, to avoid
an additional compatibility assumption in the localization theorem.

\begin{lem}[Projection-compatible collars]
\label{app:projection-collars}
A finite ordered triangulation $L$ of $M$ admits collars $d_\sigma$ with
the following property. A family of closed PL subspaces which is
a product along all shortened $d_\sigma$ has, near each open face,
simultaneous product coordinates compatible with the fixed star
projection. Two such prepared families equal on $\partial M$ are
equal in an ambient neighborhood of $\partial M$.
\end{lem}
\begin{proof}
For an $r$-simplex, write a point of its boundary in barycentric
coordinates $\lambda_0,\ldots,\lambda_r$ in the chosen vertex order.
Choose $0<\eta<1$, put $c=\eta/(r+1)$ and
$S_j=\sum_{i\le j}\lambda_i$, $S_{-1}=0$. Set
\begin{equation}\label{app:collar-formula}
 d_\sigma(\lambda,t)_j
 =\lambda_j-\eta\bigl(\min(t,S_j)-\min(t,S_{j-1})\bigr)+ct.
\end{equation}
This is jointly PL, its coordinates add to $1$, and its minimum
coordinate is $ct$. Given its output $\mu$, recover
$t=\min_j\mu_j/c$. If $Q_j=\sum_{i\le j}(\mu_i-ct)$, then
$Q_j=f_t(S_j)$, where
\[
 f_t(s)=s-\eta\min(t,s),\qquad
 f_t^{-1}(q)=
 \begin{cases}q/(1-\eta),&q\le(1-\eta)t,\\
 q+\eta t,&q\ge(1-\eta)t.
 \end{cases}
\]
Recover $\lambda_j=S_j-S_{j-1}$. These are nonnegative, add to $1$,
and include a zero coordinate. Thus \eqref{app:collar-formula} is a
PL collar onto the shell $\min_j\mu_j\le c$.

Fix a point in the interior of a face $\alpha$, and let $k$ be its
least vertex. In a small neighborhood and a short collar $S_k>t$.
For $j<k$, the sums $S_j$ involve only coordinates outside $\alpha$;
for $j>k$ the two minimum terms cancel. It follows that all output
coordinates outside $\alpha$ depend only on the input coordinates
outside $\alpha$ and on $t$. The same holds for the inverse, with
$t=\min_{j\notin\alpha}\mu_j/c$.

Starting with the identity on $\alpha$, extend by these collars over
incident simplices in increasing dimension. On their common faces
the parameter $t$ is zero and the old map is unchanged. The normal
coordinates remain the coordinates of the fixed star projection;
the remaining coordinates have the form $x_j+V_j(w)$, where $w$ is
normal. Reading $w$ and subtracting $V_j(w)$ gives the PL inverse.
Membership in a prepared marked set depends only on the $\alpha$
coordinate. This proves the product assertion.

The construction is determined by $L$ and its collars, not by the
marks. At a boundary face only the incident simplices actually in
$M$ are used. If two prepared families agree on the face, they agree
in these same coordinates. Intersect their valid neighborhoods and
use compactness of $\partial M$ to obtain a common ambient neighborhood.
\end{proof}

\begin{lem}[Finite simultaneous positioning by small shifts]
\label{app:shifts}
Let $L$ be a finite Brouwer triangulation of a compact closed PL
manifold. Given finitely many finite families of closed PL subspaces,
they can be moved simultaneously into projection-compatible product
position for $L$ by a finite composition $q=q_s\cdots q_1$ of
arbitrarily small special shifts. Any families initially in this
position remain so after every partial composition. The shifts can
be supported in prescribed open stars containing their small support
balls. There is a relative version, fixed on a smaller neighborhood
of a boundary on whose neighborhood all the families already agree
and have these coordinates.
\end{lem}
\begin{proof}
Here a Brouwer triangulation is one whose closed vertex stars admit
linear embeddings in Euclidean space, simplex by simplex; it supplies
the local linear structure used for the shifts. Arbitrarily fine
such triangulations are supplied by
\cite{armstrong1967transversality}.

We use the local shifts and descending bad-skeleton procedure in
Armstrong's polyhedral transversality argument
\cite{armstrong1967transversality}. To specify its input,
choose a single triangulation $J$ adapted to \emph{all} marks, put
$R=|J^{m-1}|$, and retract a compact regular neighborhood $W$ of $R$
onto $R$. Apply that procedure to the map $W\to M$ given by the
retraction, with fixed target triangulation $L$. Its source is a
manifold, as required by Armstrong; the marks need not be manifolds.
The local transimplicial diagram splits the image $R$ as a product
in the normal directions to an $L$-face. The image includes the
whole local skeleton, because the retraction restricts to its
identity. Recover the lower $J$-skeleta successively as nonmanifold
loci: for $1\le r<m$, each $(r-1)$-face of the $r$-skeleton has
at least three incident $r$-simplices, and the nonmanifold locus
is closed. The connected components between successive skeleta
are the open simplices. The top open simplices are the components
of $M\setminus R$. Shortening the connected product factor shows
that all these components, and then their closures, split in the
same coordinates. Every mark is a union of $J$-simplices, so all
marks become products simultaneously.

One must also preserve the already good families. In an old product
chart $\phi:H\times D^b\to U$ write $p\phi(z,u)=z$, where $p$ is
the fixed normal projection, and write its marks as $H\times Y_i$.
For a permitted sufficiently small local shift $\mu$, put
\[
 \delta=\operatorname{pr}_{D^b}\phi^{-1}\mu^{-1},\qquad A=(p,\delta).
\]
Here is the relevant projection estimate. In the local linear
structure take $p$ to be the parallel normal projection; its
comparison with the star projection is part of the local
projection construction in \cite{armstrong1967transversality}.
A local shift is affine on
a fixed finite subdivision of its support. By making its vertex
displacement small we can make the Lipschitz constant $\varepsilon$
of $\mu-\id$ arbitrarily small. For a fixed $u$,
\[
 \left|p\mu\phi(z,u)-p\mu\phi(z',u)-(z-z')\right|
 \le \|p\|\varepsilon\operatorname{Lip}(\phi)|z-z'|.
\]
Choose the coefficient on the right smaller than $1$.
Then $p\mu\phi(-,u)$ is injective. Equality of both coordinates
of $A$ therefore forces equality of its arguments. Restricting
to a smaller neighborhood and using invariance of domain gives
a PL inverse over a product neighborhood. This gives the normal
projection preservation estimate needed here.
Moreover,
$x\in\mu(X_i)$ if and only if $\delta(x)\in Y_i$, for every $i$.
Thus the estimate is an ambient estimate and applies to all marks
in that chart, not a separate transversality assertion for each
subspace. For comparison, the preservation estimate relative to a
coarser linear structure is developed in
\cite{armstrong1967transversality2}.
We use the ambient estimate above for the marked families; we do
not treat a marked subpolyhedron as a proper embedded manifold.
There are finitely many source simplices and families, so the finitely
many size bounds can be imposed simultaneously. This is preservation
under these special shifts, not under arbitrary small $C^0$ perturbations.

Run the descending bad-skeleton procedure with this one adapted
source triangulation. At each stage the product charts just described
hold off a lower-dimensional source skeleton, and the preceding
preservation estimate retains the parts already treated. Dimension decreases finitely many times,
and each skeleton has finitely many simplices. Every resulting
partial composition preserves the initially good families.
The local support neighborhoods can be chosen inside the prescribed
open stars before making the auxiliary subdivision.

For the relative assertion choose two nested boundary neighborhoods
inside the region where the families agree. Only the compact
complement of the larger neighborhood needs new product coordinates.
Use shifts supported outside the smaller one. On the intervening
closed band impose the preservation bounds for the already good
marked families. Thus the shift procedure need regularize the
retraction only on that compact complement, not on the protected
boundary region. The finite bad-skeleton induction and its support
choices are unchanged there. This is the collar cutoff used in
the relative local-shift construction of
\cite{armstrong1967transversality2}, with the foregoing
checks for the marks. Retain the original $L$-projection conditions
on the protected region; do not assert transversality there to an
arbitrary auxiliary subdivision.
\end{proof}

\begin{lem}[Controlled fragmentation of an isotopy]
\label{app:fragmentation}
An ambient PL isotopy of a compact PL manifold admits a finite
factorization by isotopies supported in small balls subordinate to
a prescribed finite open cover. At boundary points use half-balls.
The factorization can retain the homotopy class of the path relative
to its endpoints. If the path preserves finitely many subpolyhedra,
the factors can preserve them too. If the path fixes an open boundary
neighborhood, the factors can fix a smaller such neighborhood.
\end{lem}
\begin{proof}
Use Hudson's graph factorization
\cite[Theorem 6.2 and its refinement 6.2.1, printed pp.130--132;
PDF pp.136--138]{hudson1969piecewise}. We spell out its additional controls.
Write the level-preserving isotopy as $H(x,t)=(H_t(x),t)$.
In finite PL charts its inverse tracks $x_y(t)=H_t^{-1}(y)$ have a
uniform speed bound $C$. If a height function $\phi$ has spatial
Lipschitz constant $L$ with $CL<1$, then
$t-\phi(x_y(t))$ is strictly increasing. It has one zero, since its
endpoint values have opposite weak signs. Thus
$x\mapsto H_{\phi(x)}(x)$ is a PL homeomorphism. Small time divisions
and the successive vertex-height changes in Hudson's construction
ensure the bound and localize the differences to the chosen balls.

There is a jointly PL way of joining adjacent admissible heights:
\[
 \phi_u=\max\{\phi-u,\min(\psi,\phi+u)\},\qquad0\le u\le1.
\]
Its spatial Lipschitz bound is no larger than the common bound for
$\phi,\psi$, and it is stationary where they agree. If the concatenated
height path is $\Phi(x,t)$, connect it to the original height $t$ by
\[
 W(x,s,t)=\max\{\Phi(x,t)-s,\min(t,\Phi(x,t)+s)\}.
\]
The same bound holds, and both endpoint heights are fixed. This
gives the required homotopy of isotopy paths.

If every $H_t$ preserves a marked set, so does each graph map:
the inverse track of a point in that set stays in it. This proves
equality, not just forward inclusion. An open region fixed by the
original path is fixed by the graph maps. Boundary charts are
half-ball charts, so these arguments preserve the true boundary
setwise without forcing its points to be stationary.
\end{proof}

\begin{lem}[Connecting prepared endpoints]
\label{app:prepared-chain}
Suppose two stratifications have the coordinates of
Lemma~\ref{app:projection-collars} for a sufficiently fine Brouwer
triangulation $L$. If an ambient isotopy $h_t$ connects them, on a
closed manifold, it admits a finite star-supported replacement whose
intermediate endpoints are transverse to $L$ and whose total endpoint
is exactly $h_1$. The replacement has the same endpoint-relative
path homotopy class. The same construction works relative to a
boundary neighborhood fixed throughout the original path.
\end{lem}
\begin{proof}
Fragment $h_1=g_r\cdots g_1$ by Lemma~\ref{app:fragmentation}, with
the entire support of the $j$th path in a small ball $B_j$ compactly
contained in a designated open star. Write
$P_j=g_j\cdots g_1P_0$. Apply Lemma~\ref{app:shifts} to this finite
list, preserving the two already prepared endpoints. Let the small
composition be $q=q_s\cdots q_1$. Use the chain
\[
 P_0\longrightarrow qP_0\longrightarrow qP_1
 \longrightarrow\cdots\longrightarrow qP_r\longrightarrow P_r.
\]
The first and last parts insert and withdraw the individual shifts.
Their endpoints are transverse by the preservation assertion. The
middle paths are $qg_{j,t}q^{-1}$ and have transverse endpoints too.
Choose $q$ small enough that every $q(B_j)$ remains in its designated
star. The exact total endpoint is
\[
 q^{-1}(qh_1q^{-1})q=h_1.
\]
The inserted path and its inverse, together with conjugation of the
middle path, do not change its endpoint-relative homotopy class:
the square $Q_sH_t$, where $Q$ is the path to $q$, supplies the
homotopy. All maps on this parameter square are PL after a common
subdivision. In the relative case use the relative shifts and
fragmentation, which fix the chosen smaller boundary neighborhood.
\end{proof}

\begin{thm}[Finite star-supported steps; Theorem~\ref{App: transversal isotopy step by step}]
\label{app:star-steps}
Let $P$ be a stratification on a compact PL manifold $B$, let $\Delta$
be a transverse combinatorial triangulation, and let $F_t$ be an
ambient PL isotopy such that $F_{1*}P\pitchfork\Delta$.
There are finitely many ambient PL isotopies $g_{j,t}$ and vertices
$v_j\in\Delta^{(0)}$ with the following properties:
\begin{enumerate}
\item $g_{j,0}=\id$ and $g_{j,t}$ is the identity on
      $B\setminus\ost_\Delta(v_j)$;
\item if $G_0=\id$ and $G_j=g_{j,1}G_{j-1}$, then
      $G_{j*}P\pitchfork\Delta$ for every $j$;
\item $G_N=F_1$ as ambient homeomorphisms.
\end{enumerate}
No assertion of transversality during an individual step is made.
The boundary is allowed to move.
\end{thm}
\begin{proof}
Choose a sufficiently fine Brouwer triangulation $L$ subdividing
$\Delta$, so that every closed $L$-vertex star lies in some open
$\Delta$-vertex star. Its existence follows from the arbitrarily
fine version used in Lemma~\ref{app:shifts} and a Lebesgue number
for the finite open-star cover. Near a boundary one may choose it
in the double and restrict. A transverse position for $L$ is a
transverse position for $\Delta$: the simultaneous collars amalgamate,
and Lemma~\ref{app:criterion} applies.

Preparation by Lemma~\ref{app:preparation} preserves $\Delta$-cells.
Fragment any preparation path subordinate to the open
$\Delta$-stars. Lemma~\ref{app:fragmentation}, with the cells as marks,
keeps each factor cell-preserving, so all its endpoints are transverse
to $\Delta$. These preparation factors and their inverses may therefore
be inserted in the desired chain. Combining preparation with
Lemma~\ref{app:prepared-chain} proves the assertion for closed
manifolds. It also preserves the endpoint-relative path class there.
This includes the closed manifold $\partial B$ used next.

\emph{First realize the boundary motion.}
Apply the closed construction to $F_t|_{\partial B}$. Its boundary
preparation steps preserve coarse boundary simplices and extend over
$\Delta$ by ball isotopy extension, preserving coarse simplices.
For a step in the star of $v$, start with the identity on simplices
not containing $v$, and extend successively over the others. Its
support stays in that open star.

For the other boundary steps their supports lie compactly inside
an open boundary star. Put $U=\st_\Delta(v)$ for the corresponding
boundary vertex. This is a ball with
\[
 \partial U=D\cup E,\quad D=\st_{\Delta|_{\partial B}}(v),\quad
 E=\lk_\Delta(v),\quad D\cap E=\partial D.
\]
Use the boundary step on $D$ and the identity on $E$. They agree
near the seam. Extend over $U$ by RS Proposition 3.22(ii).
The original transversality gives a common collar on $\partial U$
by amalgamation. Transport it by the extension. At its endpoint
the filtration is transverse to the boundary triangulation of $U$:
on $D$ this is the chosen boundary step; on $E$ and near the seam
it is unchanged. Apply Lemma~\ref{app:relative-position} inside $U$,
fixed on all of $\partial U$, to restore interior transversality.
Concatenate the extension and this correction as one step.
It fixes $E$, so is the identity outside $\ost_\Delta(v)$ in $B$.

Let $a_t$ be the composite path, with endpoint $A$. Its boundary
path $\alpha_t$ has the same endpoint-relative homotopy class as
$\beta_t=F_t|_{\partial B}$.

\emph{Remove the residual boundary path.}
Choose the PL homotopy of boundary isotopies $C_{s,t}$ from $\alpha_t$
to $\beta_t$, fixed at $t=0,1$. Set
\[
 r_t=F_ta_t^{-1},\qquad b_{s,t}=\beta_t C_{s,t}^{-1}.
\]
Then $b_{0,t}=r_t|_{\partial B}$, $b_{1,t}=\id$, and
$b_{s,0}=b_{s,1}=\id$. On an ordinary boundary collar, with its
interval rescaled to $[0,1]$, set
\[
 E_t(c(x,s))=c(b_{s,t}(x),s),
\]
and use the identity outside. This is a jointly PL isotopy with
$E_0=E_1=\id$. Consequently $k_t=E_t^{-1}r_t$ fixes the boundary
pointwise and has endpoint $F_1A^{-1}$. We have not inferred a
relative isotopy merely from the fact that its endpoint fixes the
boundary; the boundary-path homotopy is essential here.

\emph{Deal with the relative residual motion.}
Write its starting and ending stratifications as $P_0,P_1$.
Choose preparation paths $a_{0,t},a_{1,t}$ whose boundary restrictions
agree at every time, and straighten both endpoints using the same
collars of Lemma~\ref{app:projection-collars}. Set
\[
 Q_i=a_{i,1*}P_i,\qquad
 H_t=a_{1,t}k_ta_{0,t}^{-1},\qquad H=H_1.
\]
The two $Q_i$ agree on a boundary neighborhood $V$;
$H_t$ fixes the boundary and $H_*Q_0=Q_1$.
Choose a short filtered collar $c$ for $Q_0$ such that both $c$ and
$Hc$ lie in $V$. They are collars for the same marked family there.
By Lemma~\ref{app:collars} there is $u_t$, fixed on the boundary and
supported in $V$, preserving $Q_0$ and $Q_1$, such that $u_1c=Hc$
on a shorter collar. Thus
\[
 w=u_1^{-1}H,\qquad w_*Q_0=Q_1,
\]
and $w$ is the identity on a boundary collar.

The path $w_t=u_t^{-1}H_t$ fixes the boundary but might not fix that
collar. Choose a PL embedding $i:B\hookrightarrow B$ pushing the
boundary slightly inward and equal to the identity outside a collar
on which $w=\id$. Define $W_t=iw_ti^{-1}$ on $i(B)$ and the identity
on its complement. The maps agree on the seam since $w_t$ fixes the
boundary. The same formula for the inverse shows that this is a
jointly PL isotopy, fixed on a thin outer collar. The embedding $i$
commutes with $w$, so $W_1=w$ \emph{exactly}.

Apply the relative case of Lemma~\ref{app:prepared-chain} to $W_t$.
Use support balls inside $\Int B\cap\ost_L(v)$; the vertex $v$ may
be a boundary vertex. These sets cover the interior. Follow the
resulting chain by a fragmentation of $u_t$. Since $u_t$ preserves
$Q_1$, its factors do too, so their endpoints remain transverse.
Insert the initial preparation and withdraw the final preparation.
The exact total endpoint of this residual chain is
\[
 a_{1,1}^{-1}u_1wa_{0,1}=k_1.
\]
All steps lie in original $\Delta$-stars, by the choice of $L$ and
the control of the preparation factors. Finally concatenate the
boundary chain with this residual chain. Its endpoint is
$(F_1A^{-1})A=F_1$, as required. In dimension zero take the empty chain.
\end{proof}

\subsection{Lemma~\ref{App: isotopy WLOG}: a product region in a sphere product}
\label{appsec:product}

This result uses the given global support $D^k\times S^q$.
It is not a statement that a locally transverse disk automatically
has a prescribed global trivial normal bundle. Its proof uses a
simplicial projection and does not use Theorem~\ref{app:star-steps}.

\begin{lem}[Simultaneous PL trivialization of a simplicial projection]
\label{app:williamson}
Let $K$ be a finite complex and $f:K\to\tau$ a simplicial map to a simplex. Let
$p\in\operatorname{relint}\tau$. Over a sufficiently small relative neighborhood $U\subset\operatorname{relint}\tau$ of $p$
there is a PL homeomorphism
\[
 h:f^{-1}(p)\times U\longrightarrow f^{-1}(U),\qquad
 fh=\operatorname{pr}_U,
\]
such that, for every subcomplex $J\subset K$,
\[
 h^{-1}(|J|)=(|J|\cap f^{-1}(p))\times U.
\]
One can require $h(x,p)=x$.
\end{lem}
\begin{proof}
This is Williamson's Theorem 1.3.1
\cite[printed pp.4--5; PDF pp.5--6]{williamson1966cobordism}. Its statement explicitly
preserves every source subcomplex. Precompose with the inverse of
the central-fiber restriction to obtain the normalization.
Williamson's proof first constructs a barycentric homeomorphism
and then replaces it by a PL one preserving the relevant simplices
and projection. The preliminary barycentric formula itself is not
PL. We use the resulting jointly PL homeomorphism.
\end{proof}

\begin{lem}[The product region; Lemma~\ref{App: isotopy WLOG}]
\label{app:product-region}
Let $q=n-k\ge1$, and let $\mathcal L$ be a $V$-labeled manifold
supported on $X=D^k\times S^q$, with the prescribed boundary
$\mathcal S\times S^q$, $|\mathcal S|=\partial D^k$.
There are a standard closed PL ball $C\subset S^q$, a $G$-labeled
stratified manifold $\mathcal D^k$ supported on $D^k$ with
$\partial\mathcal D^k=\mathcal S$, and an ambient PL isotopy $F_t$
fixed pointwise on $\partial X$ such that
\[
 (F_{1*}\mathcal L)|_{D^k\times C}=\mathcal D^k\times C.
\]
The product region contains no position-points. Its closed complementary region has
support $D^k\times D^q$.
\end{lem}
\begin{proof}
Let $Q$ be the finite set of position-points and let
$\pi:D^k\times S^q\to S^q$ be projection. Triangulate all filtration
terms, the external boundary and $Q$ as subcomplexes, and make $\pi$
simplicial by subdivision, using
\cite[Lemma 1.1.2, printed p.3; PDF p.4]{williamson1966cobordism}.
The points of $Q$ are source vertices, so their images are target
vertices. Choose $p$ in a top-dimensional target simplex and a
sufficiently small closed convex ball $C$ in its interior.
The set $\pi^{-1}(C)$ misses $Q$.

Lemma~\ref{app:williamson} gives
\[
 h:D\times C\longrightarrow\pi^{-1}(C),\qquad
 h(x,y)=(h_y(x),y),\qquad h_p=\id,
\]
preserving all the filtration terms and the external boundary,
where $D=\pi^{-1}(p)$ is identified with $D^k$ by the original
product coordinates. Give $D$ the filtration
$D^r=D\cap X^{r+q}$.
The product-slice Lemma~\ref{app:slice} gives at a point in an
original $j$-stratum the model
$\mathbb R^{j-q}\times cL$, or $H^{j-q}\times cL$ at a boundary
point. In particular $j\ge q$ on the fiber. The original recursive
filtered link is unchanged, and its sphere dimension is
$k+q-j-1=k-(j-q)-1$. Its boundary filtration is precisely $\mathcal S$.
The local half-space charts have a common filtered collar by
Lemma~\ref{app:collars}. Thus the fiber really is a stratified
PL manifold $\mathcal D^k$, not just a filtered polyhedron.

The map $h$ need not fix the external boundary pointwise. For $k>0$
put $b_y=h_y|_{\partial D}$; then $b_p=\id$ and $b_y$ preserves
$\mathcal S$. Choose a closed filtered collar with coordinates
$(x,s)\in\partial D\times[0,1]$ in $D$.
There is a jointly PL map
\[
 r:[0,1]\times C\to C,\qquad
 r(0,y)=y,\quad r(1,y)=p,\quad r(s,p)=p.
\]
For example triangulate the prism relative to its ends and the
segment $[0,1]\times\{p\}$, prescribe these values at its vertices,
and extend affinely on simplices. Convexity of $C$ keeps the image
in $C$. On the collar times $C$ define
\[
 a(x,s,y)=(b_{r(s,y)}^{-1}(x),s,y),
\]
and take the identity off the collar. The formulas agree at its
inner end. The inverse replaces $b^{-1}$ by $b$, so $a$ is jointly
PL and preserves all product filtration terms. Consequently
$g=h\circ a$, written $g(x,y)=(g_y(x),y)$, has
\[
 g_p=\id_D,\qquad g_y|_{\partial D}=\id_{\partial D}.
\]
For $k=0$ there is no boundary correction; take $g=h$.

The ball $C$ lies inside a sphere simplex and can be chosen standard,
so its closed complement is a PL ball. Extend its identity map to
a PL map $\rho:S^q\to C$ by coning the complementary ball into $C$.
Using a triangulated prism as above choose a jointly PL homotopy
$\rho_t:S^q\to C$ from the constant map $p$ to $\rho$.
Now set
\[
 H_t(x,y)=(g_{\rho_t(y)}(x),y).
\]
This is jointly PL, every slice is a homeomorphism, $H_0=\id$, and
the external boundary is fixed. Its level-preserving inverse is
PL as well. On $D^k\times C$, $H_1=g$, so the pullback
$H_1^*\mathcal L$ has the desired product form and has no
position-points there. Take $F_t=H_t^{-1}$ to obtain the pushforward
convention in the statement.

Finally the same product identification carries every wall and its
two sides. A wall label is constant on its connected wall; since
$C$ is connected and the central fiber is fixed, the labels are
exactly the labels inherited from $\mathcal D^k$.
No gauge transformation or invariance of a partition function has
been used. The complement of the block is supported on
$D^k\times\overline{S^q\setminus C}\cong D^k\times D^q$.
\end{proof}

\subsection{Lemma~\ref{App: Filling: anything}: ball fillings and holonomy}
\label{appsec:fillings}

Throughout this subsection $G$ is discrete. A wall is read with the
main text's coorientation convention.
Reversing a crossing replaces its label by its inverse.

\begin{lem}[Ball filling and contractible-loop holonomy; Lemma~\ref{App: Filling: anything}]
\label{app:ball-filling}
A $G$-labeled $k$-sphere bounds a $G$-labeled $(k+1)$-ball if $k\ge2$,
or if $k=1$ and its total transport is $1$.
The filling can retain any prescribed sufficiently short boundary
product collar. A transverse closed curve which is null-homotopic
in a $G$-labeled manifold has transport $1$.
\end{lem}
\begin{proof}
\emph{The filling.} Take the underlying PL cone on the given sphere and give it the
filtration
\[
 X^0=\{v\},\qquad
 X^{i+1}=\{v\}\cup(S^i\times(0,1]),\quad 0\le i\le k.
\]
This includes the old boundary stratum in its radial product stratum;
it does not add the boundary as a lower-dimensional ambient stratum. The recursive charts are those
of the original sphere, times this interval; the vertex chart has the
original sphere as its link. Thus this is a stratified PL ball in
Definition~\ref{stratified PL manifold}, including its boundary collar.

Give a radial wall the label of its boundary wall. Away from the vertex,
every flatness equation is inherited from the original sphere.
The cone vertex has codimension $k+1$. If $k\ge2$, it imposes no new
codimension-two equation. If $k=1$, its equation is exactly that the
ordered product around the given circle is $1$. This proves existence.
A prescribed short collar can be inserted first; cone the inner end.
The resulting underlying ball can be identified with any specified PL
ball: extend the given boundary identification over its cone by
\cite[Lemma 1.10, printed p.8; PDF p.16]{rourke2012introduction}.

\emph{Holonomy.} We give the argument independently of gauge
classification and of every state-sum identity. In a PL manifold of
dimension $d$, the strata of codimension at least three form a closed
subpolyhedron of dimension at most $d-3$. A PL disk homotopy can be
put in general position relative to it, fixed on its already transverse
boundary, and hence made disjoint from it. One way to perform this
familiar relative general-position step is to triangulate the finite
family of subpolyhedra, subdivide the disk so that its simplices lie in
PL coordinate balls, and perturb its interior vertices successively.
The forbidden positions are finite unions of polyhedra of positive
codimension; the intersection bound is $2+(d-3)-d<0$.
The relative construction uses a fixed thin boundary annulus. This is
the PL general-position construction of \cite[Chapter IV]{hudson1969piecewise}.

In the complement of the codimension-three skeleton, a wall is
locally a flat hypersurface, and a codimension-two stratum has a
normal disk whose walls are radial arcs. A further relative generic
choice makes the inverse image of the walls a finite graph in the
homotopy disk. Its interior vertices either represent an inverse pair
of crossings, or the radial wall pattern at a codimension-two stratum.
There are no other singularities: use the preceding product charts
and triangulate the disk map there. At a tangency two successive labels
are $g,g^{-1}$; at a radial vertex their cyclic product is $1$ by
flatness. Subdivide the disk into small regions of these two kinds and
regions avoiding all walls. Reading their boundaries, internal edges
cancel with opposite orientations. Equivalently, remove one region at
a time along a spanning tree of the dual graph; the boundary word
changes by inserting an inverse pair or a conjugate of a flatness
word. Its final word is empty. The original boundary word is therefore
$1$. A boundary curve is first pushed a short distance into a boundary
collar, which does not change any crossing label.
\end{proof}

\begin{coro}[The sector cylinder used by the algorithm]
\label{app:sector-cylinder}
Let $\mathcal D_1^k,\mathcal D_2^k$ have the same identified boundary
$\mathcal S^{k-1}$. For $k\ge2$ there is a $G$-labeled cylinder
$\mathcal C^{k+1}$, supported on $D^k\times I$, with ends
$\mathcal D_1^k,\mathcal D_2^k$ and side
$\mathcal S^{k-1}\times I$, with compatible collars and the usual incoming/outgoing orientations
on these faces. For $k=1$ it exists if and only if the two oriented transports agree.
For $k=0$ choose the undecorated interval with transport $1$.
\end{coro}
\begin{proof}
The two ends and side form a $G$-labeled $S^k$, using their boundary
collars to glue the seams. Apply Lemma~\ref{app:ball-filling}.
When $k=1$, the boundary word is $h_1h_2^{-1}$, since the two fixed
side paths have transport $1$. Necessity follows from the
contractible-loop assertion of that lemma. Extend the prescribed
boundary identification to the cylinder and retain short collars
before coning the inner sphere. At corners, straighten a quadrant in
the two interval coordinates, leaving all marked factors unchanged.
\end{proof}

\subsection{Lemma~\ref{App: Filling: algorithm can continue}: continuation in the prescribed face order}
\label{appsec:algorithm}

Put $d=n+1$. The boundary is full in the vertex sense of the main text's definition of fullness. We keep the prescribed face-before-coface order and do not first remove the true-boundary collar.

\begin{prop}[Full triangulations and the positioning results]
\label{app:full-boundary}
A finite combinatorial triangulation can be made boundary-full by
subdivisions that leave its entire boundary triangulation unchanged.
Subdivisions and ambient PL homeomorphisms preserve boundary fullness.
Consequently a compact stratified PL manifold has a transverse,
boundary-full triangulation. The extension in Theorem~\ref{App: transversal triangulation extend to bulk} can be
chosen boundary-full without changing the prescribed triangulation.
Theorems~\ref{App: transversal isotopy step by step} and~\ref{App: transversal subdivision preparation} retain fullness when their input triangulation
is full.
\end{prop}
\begin{proof}
Subdivide the nonboundary simplices in decreasing dimension, placing
a new vertex in each of their interiors. Every remaining nonboundary
simplex then has an interior vertex. No boundary simplex is changed.
This is the relative derived construction in
\cite[Exercise 3.2 and Lemma 3.3, printed p.31; PDF p.39]{rourke2012introduction}.

For preservation under a subdivision, let a new simplex lie in an
old simplex $\tau$. Its boundary vertices lie in the face
$\tau\cap\partial B$, so the new simplex lies in that face too.
An ambient homeomorphism preserves the true boundary and all the
relevant incidences. To obtain a transverse full triangulation,
first fullify and then apply Theorem~\ref{App: transversal triangulation}. For the extension, first
extend the prescribed boundary face, then fullify relative to the
whole boundary and use boundary-fixed positioning. The other two
assertions follow because their operations are subdivisions and
ambient homeomorphisms.
\end{proof}

\begin{lem}[The ordinary cores]
\label{app:ordinary-cores}
Let $K$ triangulate a compact PL $d$-manifold with full boundary.
Fix any order $\sigma_1,\ldots,\sigma_N$ of its nonboundary simplices
of dimensions less than $d$, extending the face order. There are
successively small ordinary cuts for which
\[
 D_s=|\sigma_s|\cap B_{s-1}
\]
is a properly locally flat ball of dimension $\dim\sigma_s$.
The residual part of every unprocessed simplex is a ball. Its
intersections with the other residual cells are faces, and the final
top-dimensional pieces are disjoint balls.
\end{lem}
\begin{proof}
Let $f_\sigma$ be affine on each simplex, with value zero at vertices
of $\sigma$ and value one at all other vertices. Its zero set is
exactly $|\sigma|$. Choose small positive numbers $t_s$ and put
\[
 B_s=\bigcap_{i\le s}\{f_{\sigma_i}\ge t_i\}.
\]
Choose each level small enough to retain the barycenter of every
unprocessed simplex. This is possible: no previously processed
simplex contains an unprocessed simplex. Each residual cell is
therefore a full-dimensional convex polytope in its original simplex.

The levels can also be chosen generically. On the relative interior
of any old face $\rho$, the affine differentials of the active
constraints are then independent. A dependence with a common level
point would impose an affine equation on the chosen levels; there
are only finitely many such equations. In a local star chart write
the active functions as $a_i(u)+b_i(z)$, with $u$ tangent to $\rho$
and $z$ in its normal cone. Use the independent $a_i$ as some of the
$u$ coordinates and subtract the PL functions $b_i(z)$. This PL
coordinate change makes the retained set an orthant times the normal
cone. It proves that $B_s$ is a PL manifold, including at corners.

All nonboundary proper faces of the current $\sigma$ have been
removed. To examine a surviving true-boundary point, write
$\sigma=\beta*\gamma$, where $\beta=\sigma\cap\partial K$.
If $\gamma$ has at least two vertices, a proper face $\beta*w$ with
$w\in\gamma$ has already been removed, so no point of $\beta$
survives. If $\gamma=\{w\}$, every proper face of $\beta$ has been
removed with a face $\eta*w$. Thus surviving true-boundary contact
is in $\Int\beta$, where $\beta$ is a facet and $w$ is interior.
The normal pair there is
\[
 \bigl(c\lk_K\beta,\ c\{w\}\bigr).
\]
The link is a PL ball with an interior point $w$. This is the
standard half-space--inward-ray pair. Together with the preceding
orthant coordinates it proves proper local flatness of $D_s$ and
$D_s\cap\partial B_{s-1}=\partial D_s$.

The sufficiently short sublevel neighborhood of $D_s$ is a regular
neighborhood: triangulate the finite cut arrangement and place the
derived vertices on its positive boundary level, as in
\cite[printed pp.31--33; PDF pp.39--41]{rourke2012introduction}.
Normal block neighborhoods exist for the proper locally flat disk,
and a block bundle over a PL ball is trivial
\cite[Theorems 4.3 and 1.1, printed pp.16 and 4; PDF pp.17 and 5]{rourke1968block}.
Thus these cuts have the required ordinary disk-product form.

Finally any common face of two distinct top simplices is contained
in an internal facet separating incident top simplices: use a path
in the facet-adjacency graph of its ball or sphere link. That facet
was processed, removing the entire face that still remained.
Consequently the terminal convex top-cell balls are disjoint.
\end{proof}

\begin{lem}[Marked neighborhoods of the current core]
\label{app:marked-core}
Suppose a finite ball complex has the truncated-cell type in
Lemma~\ref{app:ordinary-cores}, and the whole boundary of each of
its cells has a simultaneous marked collar. Let $D$ be its next
proper locally flat disk cell. There are nested normal neighborhoods
of $D$ with the following properties:
\begin{itemize}
\item they have a single normal block structure retaining every mark;
\item after an inner neighborhood is removed, the residual cells
have the same ordinary truncated-cell type and simultaneous marked
boundary collars;
\item on a compact inner disk of $D$, the normal product can retain
the transverse-star product and all its angular cofaces;
\item their total product identifies the ambient-boundary part with
$\partial D\times D^{d-\dim D}$.
\end{itemize}
All choices can be made relative to already fixed short collars,
and simultaneously for a finite collection of marked families having
the same underlying cell collars.
\end{lem}
\begin{proof}
Here are the construction and its relative controls. Work over the
cells in increasing dimension. On a cell $\tau$, retain the cut
already made on $\partial\tau$ and extend it into a short marked
collar of $\partial\tau$. If that boundary cut is $A$, its extension
is $A\times[0,t]$. Its new frontier consists of
$A\times\{t\}$ and $\partial A\times[0,t]$. The collar of
$\partial A$ constructed in the preceding dimension rounds their
corner by a PL change of a rectangle. It preserves every mark,
because the marks are products in these collar coordinates. A cell
contained in $D$ is retained in full, and a cell disjoint from $D$
receives no cut. Use successive short levels for nested neighborhoods.

The intersections $D\cap\tau$ are faces of the ordinary truncated
cells. A short neighborhood of such a face is an unknotted cap;
the cap and its closed complement are balls. Compare this
construction with the convex cuts of Lemma~\ref{app:ordinary-cores}
cell by cell: retain the map already defined on $\partial\tau$,
match the cap frontier, and extend across the cap and its complement.
The extension is by the PL ball extension theorem
\cite[Lemma 1.10, printed p.8; PDF p.16]{rourke2012introduction}.
It preserves the old cell incidences. The preceding rectangle
construction also gives a marked collar on every new frontier
face and on every residual cell boundary. In particular this is a
simultaneous corner assertion, not two independent collar assertions.

We spell out why these neighborhoods can be given a \emph{marked}
normal structure. The dual construction in
\cite[Theorem 4.6 and its proof, printed p.286; PDF p.19]{mccrory1975cone}
gives a structured cellular subdivision $J$ in which the old cells
are subcomplexes and every mark is a union of dual cells. To retain
several marks, triangulate all of them at once in the interior of
each cell; on its boundary keep the structure already constructed,
and extend that structure through the common collar. Dualize this
extension. This is precisely the induction in that proof, with a
finite family in place of one subpolyhedron. It also retains
prescribed product regions. The simultaneous primal and dual
refinement is supplied by
\cite[Lemma 2.6, printed pp.276--278; PDF pp.9--11]{mccrory1975cone}.
Use the collar prisms just described in this construction.

In a flag $\alpha_0<\cdots<\alpha_r$ of $J$, the cells belonging
to $D$ form an initial segment. Give their derived vertices height
zero and the others height one. Short sublevel neighborhoods are
the derived collar neighborhoods above. A dual cell indexed by a
cell outside $D$ has height one throughout and misses such a
neighborhood. The truncated ambient dual blocks indexed by cells
in $D$ lie over the corresponding dual blocks in $D$. They form
normal blocks: their ball pairs are the local link pairs of the
proper locally flat inclusion of $D$. This is the dual-block proof
of \cite[Theorem 4.3, printed pp.16--19; PDF pp.17--20]{rourke1968block},
in the cellular form used in
\cite[Theorem 4.4 and its following remark, printed pp.285--286;
PDF pp.18--19]{mccrory1975cone}. The proper boundary blocks are included
by its relative construction
\cite[Lemma 4.2, printed p.284; PDF p.17]{mccrory1975cone}.

For this \emph{one} normal block bundle $\xi$, each mark $X$ satisfies
\[
 X\cap E(\xi)=E(\xi|_{X\cap D}).
\]
Indeed both sides are the union of exactly its selected dual blocks.
This is also the construction behind the equivalence of embedded
mock and block transversality in
\cite[II, Theorem 4.4, printed pp.34--35; PDF p.21]{buoncristiano1976geometric}.
It is not obtained by intersecting separately chosen normal bundles.
Sliding between two short heights keeps the same flag and dual
block; after a finite prism subdivision these are PL collars
retaining the marks and the old cells. The preceding collar-prism
description checks their boundary corners.

Choose a compact inner disk $D_0\subset\Int D$ by deleting a short
collar of $\partial D$. The transverse-star construction used in
Lemma~\ref{app:criterion} is a product over a neighborhood of
$D_0$, with the incident cofaces as angular faces. Keep that product
on a slightly larger inner disk while constructing $J$. Its normal
blocks there are the product blocks. Since $D$ collapses to $D_0$,
their trivialization extends over $D$ by
\cite[Theorem 1.6, printed p.6; PDF p.7]{rourke1968block}.
A block isomorphism respects every base block, hence every marked
base subcomplex. It also respects the proper boundary blocks.
This proves the asserted global marked product. Its extension at
the ends need not retain the entire angular coface fan; no such
assertion is used.
The relative data here are the already constructed geometric collars
and the central product, not arbitrary independently prescribed
normal framings at both ends. The extension theorem is applied to
this one block bundle and the inner subdisk.
Trivialize the innermost neighborhood first, and extend through the
successive height collars. This gives one product parameterization
for all nested levels, rather than unrelated trivializations.
The $G$-labels also pull back in this product: each normal disk is
contractible, and the wall germs and their coorientations are
retained by the common marked blocks.
\end{proof}

\begin{lem}[Controlled continuation; Main text, Lemma~\ref{App: Filling: algorithm can continue}]
\label{app:algorithm}
Let $\mathcal B$ be a bulk vector and let $K$ be a transverse
triangulation with full true boundary. Choose any
face-first order in the algorithm in the main text. Its neighborhoods and fillings
can be chosen so that every literal core
$D_s=|\sigma_s|\cap B_{s-1}$ has the required $G$-stratified product
neighborhood, every subsequent step is legal, and the final pieces
are balls. All summands can use the same underlying cuts.
The filling can equal the original structure near the boundary of
the core and outside the prescribed inner buffer in the definition of the $\mathcal S^{k-1}$-relations.
\end{lem}
\begin{proof}
Initially the transverse triangulation has simultaneous filtered
whole-cell collars by Lemma~\ref{app:criterion}. Choose the
projection-compatible collars of Lemma~\ref{app:projection-collars};
this is a cell-preserving preparation, which may be pulled back.
The induction retains the truncated original cells and their
cut-frontier faces, with common filtered collars. The ordinary
comparison in Lemma~\ref{app:ordinary-cores} makes the next literal
core a proper locally flat disk. Apply Lemma~\ref{app:marked-core}
to obtain its filtered normal product and nested buffer levels.
Its central product retains all the future angular faces.

Write $\mathcal D_1^k$ for the current core type and choose a
summand with type $\mathcal D_2^k$. They have the same identified
boundary $\mathcal S^{k-1}$. Using their filtered boundary collars,
realize $\mathcal D_2^k$ so that it agrees with $\mathcal D_1^k$
in a fixed short end collar. Transport the label with this
identification. Corollary~\ref{app:sector-cylinder} supplies
$\mathcal C^{k+1}$ on $D^k\times I$ with the prescribed ends and
side, unchanged in that end collar. For $k=1$ use only the matching
transport sector; for $k=0$ take the interval of transport $1$.
Insert
\[
 \mathcal C^{k+1}\times S^{d-k-1}
\]
in the radial annulus. Retain an outer product collar, so that the
structure is unchanged outside $W$ in the definition of the $\mathcal S^{k-1}$-relations. Place the
small correction carrying the $\alpha_k$ label inside a
top-dimensional angular face. It misses the finite angular faces
for all later nonterminal cores. Its filling is the ball filling
of Lemma~\ref{app:ball-filling}, with its boundary collar fixed.
For $k=d-1$ use the two dual-basis insertions; there is no later
nonterminal core to protect.

Consider the remaining cell collars. In the central region every
future coface is the base times an angular face. After replacement
its relevant pieces are $\mathcal C^{k+1}$ times that angular face,
and its inner frontier is $\mathcal D_2^k$ times the angular face.
Their collars are the prescribed common side collars and the
angular collars. Outside this central region the filling is
unchanged, so the collar-prism construction of
Lemma~\ref{app:marked-core} still applies. At their joins use the
same side collar and the same PL rectangle identification.
Thus the whole boundaries, not merely individual boundary pieces,
have simultaneous filtered collars for the next stage.

The persistent cells here are the truncated original cells and
their frontier faces. The auxiliary dual subdivision used to build
the neighborhood is discarded after this step. In particular we
do \emph{not} require the arbitrary new type $\mathcal D_2^k$ to be
transverse to that auxiliary base subdivision. At the next stage
choose a new dual subdivision retaining the new marks.

There are finitely many summands at every step, by the algebraic
finiteness assumptions. Prescribe the same ordinary side and
angular collars in all of them. Apply the marked-neighborhood
construction to the finite family of all their filtration terms.
This produces a single underlying cut and preserves the common
collars. The induction therefore applies to the full finite state
tree and justifies the common support $B_s$ in the algorithm.
Finally Lemma~\ref{app:ordinary-cores} identifies the terminal
pieces as disjoint balls, each with its complete labeled boundary.
\end{proof}

\subsection{Proposition~\ref{App: order}: choices and face order}
\label{appsec:order}

We first record precisely the filling comparison used in the main
text. It concerns a fixed legal cutting scheme; existence of such a
scheme was proved in Lemma~\ref{app:algorithm}.

\begin{lem}[Filling comparison for a fixed legal scheme]
\label{app:filling-comparison}
Fix the underlying cutting neighborhoods, their boundary
identifications, and the representatives at every node of a finite
algorithmic state tree. Suppose two choices of fillings obey
the definition of the $\mathcal S^{k-1}$-relations and make this same scheme legal.
Then the two complete sums agree. The fillings need not be the
particular cylinders used in the continuation proof.
\end{lem}
\begin{proof}
Match the initial boundary diagrams and proceed along each branch.
At the next cut the link boundary is made from the original boundary
and the already inserted boundary types. These pieces, their labels
and their identifications are the same. The only additional sector
datum occurs for Move $1$. Changing a preceding vertex filling does
not change that transport: the radial transport required in
the definition of the $\mathcal S^{k-1}$-relations is $1$.

Here is a path formulation that also covers changes of a whole
permitted shell filling. A shell has support
$D^k\times S^{n-k}\times I$. For $k\ge1$, each component has
connected boundary, and its fundamental group is generated by loops
in that boundary: the boundary contains a copy of the sphere factor,
and the disk and interval factors are contractible. Thus any path
whose ends lie on the boundary is homotopic, relative its ends, to a
boundary path. For $k=0$ there are two boundary components; add one
radial connector, whose transport is prescribed to be $1$. These
boundary paths and the connector generate the relative fundamental
groupoid. For $k=n$, apply the argument to the two ball components
separately. Codimension-two flatness makes wall transport invariant
under these homotopies, by the path argument in
Lemma~\ref{app:ball-filling}. Consequently the identified boundary
data and the specified radial transport determine every transport
used by a later cut. This is a statement about $G$-labels, not about
the numerical gauge invariance to be proved later.

The next relation therefore has the same index set, coefficients
and inserted boundary diagrams in the two computations. At the
highest index choose the same paired bases. Induction matches the
complete terminal sphere diagrams and all coefficients. Move
$(n+1)$ evaluates only those diagrams, and hence gives the same sum.
\end{proof}

\begin{prop}[Main text, Proposition~\ref{App: order}, with auxiliary choices]
\label{app:algorithm-choices}
For fixed representative choices, the controlled construction in
Lemma~\ref{app:algorithm} has the same scalar value for different
auxiliary collars, sufficiently short cut levels, permitted coupon
carriers and linear extensions of the nonboundary face order.
Matched legal fillings can be changed as in
Lemma~\ref{app:filling-comparison}. The construction is natural under
PL homeomorphisms transporting all data. In particular a
cell-preserving change of coordinates preserves its value.
\end{prop}
\begin{proof}
First fix the top-dimensional angular cell carrying each coupon.
Compare collars over the old cells in increasing dimension. Match
the already constructed boundary caps, use simultaneous collar
uniqueness (Lemma~\ref{app:collars}) to put their extensions in one
collar, and match the finitely many short levels by a piecewise
affine change of its interval coordinate. Keep the outer end and a
prescribed inner core fixed. The maps agree on common faces and
preserve all marks. Compare the central transverse-star parameters
by cellwise ball isotopy in the normal factor. At the ends the old
and new filtrations coincide and there is no coupon; changes of
the extended normal trivialization there do not change the physical
filled diagram. Move the coupon within its fixed carrier by an
isotopy supported in that carrier. These identifications transport
the states, their boundary parameters and their weights. Different
controlled cylinders are compared by choosing common cuts for their
finite marked families and applying
Lemma~\ref{app:filling-comparison}.

Next interchange two consecutive incomparable faces. If their
remaining cores are disjoint, use disjoint short neighborhoods;
the operations and their finite sums commute. Otherwise all
nonboundary common proper faces have already been removed.
By Lemma~\ref{app:ordinary-cores}, the surviving intersection lies
in the relative interior of a common true-boundary facet $\beta$.
In particular the two cores have the same dimension. In a marked
normal chart over the remaining part of $\beta$ they are two inward
rays in the half-space cone $c\lk_K\beta$.

Take a short full cap $C$ in this cone, and disjoint thin ray tubes
$T_1,T_2$ outside $C$. The first core uses $C$ and its own tube.
The other core is shortened by a boundary collar and starts at the
cap frontier. Reversing the order removes the same set
$C\cup T_1\cup T_2$. The construction over $\beta$ uses its common
marked collar; at its boundary use the same corner collar as in
Lemma~\ref{app:marked-core}. Thus both orders are legal.
Keep both fillings unchanged in the common cap and in their end
collars. Their changed central regions and coupons are disjoint.
The new core types differ only by the unchanged boundary collars;
use these collars to identify the parameters and labels. At index
$1$ the extra short collar has transport $1$. Therefore the weights
and final states match for each pair of labels, with their coupon
carriers fixed. The two finite sums commute. This argument uses a
shared cap, not disjoint neighborhoods of the entire original cores.
Any two linear extensions are connected by these adjacent exchanges.

It remains to compare different angular carriers. This last
comparison is made after summation, not term by term. By the order
comparison just proved, use an order increasing in dimension.
Fix the lower-dimensional labels. On the compact central part of
a $k$-core, $k<n$, the normal angular sphere is $S^{n-k}$ and the
filtration is a product in its angular direction. Join the two
carrier cells by a path crossing their codimension-one faces and
missing the codimension-two skeleton. Such a path exists since the
facet-adjacency graph of this sphere is connected. The crossed
angular faces correspond to $n$-dimensional cofaces of the original
core, precisely the last nonterminal cuts. Keep the core's end
collars fixed throughout this local comparison.

For one crossing let $V_S,V_T$ be the finite-dimensional local disk
spaces on its two sides. A local coupon tangle gives a linear map
$A:V_S\to V_T$, and reverse gluing gives the dual map
$A^\vee:V_T^\vee\to V_S^\vee$. These are spaces of the original
$S^n$-algebra, not the Time Boundary state spaces constructed later.
Write $\Omega_S=\sum_i e_i\otimes e^i$ and similarly for $T$.
The inverse-pairing tensors satisfy
\begin{equation}\label{app:coupon-tensor}
 (A\otimes\id)\Omega_S=(\id\otimes A^\vee)\Omega_T
       \quad\text{in }V_T\otimes V_S^\vee.
\end{equation}
Indeed contraction against $\lambda\in V_T^\vee$ and $v\in V_S$
gives $\lambda(Av)$ on both sides; these contractions separate
tensors. The highest nonterminal relation inserts exactly this
inverse-pairing tensor. Consequently moving the coupon through
one such cut leaves the complete contraction unchanged. This
includes arbitrary labels and sphere functionals attached to its
other slots. The local tangle and its reverse carry the given
$G$-data, so no gauge transformation or unproved gauge invariance
is invoked. Apply the identity at each crossed face and then sum
over the other, fixed, lower labels. Finally return to the original
face-first order. The case $k=n$ is the same basis-independent
inverse-pairing tensor itself.

Naturality follows by transporting the entire geometric
construction and the hypersphere evaluations. The proof has used
neither global relations nor isotopy invariance for a fixed
triangulation. In particular different carrier choices have been
compared by their complete sums; no assertion that all their
individual terminal diagrams coincide is needed.
\end{proof}

\subsection{Lemma~\ref{App: Shellable}: evaluation of a complete shellable tail}
\label{appsec:shell-tail}

Write $E(B,K)$ for the algorithmic sum and reserve $Z$ for the given sphere functional. The only numerical input in this subsection is the adjacent identity in the main text. In particular, global relations and general triangulation independence are not used.

A \textbf{shelling} orders the top simplices of a ball so that each simplex after the first meets the preceding union in a nonempty union of facets forming a ball. For a sphere the final intersection is its entire boundary. An elementary collapse removes a simplex together with a facet contained in no other remaining simplex.

\begin{lem}[A relative face order for a shellable ball]
\label{app:shell-order}
If $Q$ is a shellable $d$-ball, it has a top simplex $\Omega$ such
that $Q\setminus\Int\Omega$ collapses to $\partial Q$. Reversing
the collapse orders its interior faces in consecutive pairs
$(\sigma,\tau)$, where $\sigma$ is a facet of $\tau$, followed by
$\Omega$. This order extends the face order relative to $\partial Q$.
\end{lem}
\begin{proof}
Take $\Omega$ to be the first simplex of a shelling and remove its
open top simplex. Suppose the earlier shelling part has been
reduced to its boundary. For the next simplex $F$, partition its
vertices into nonempty simplices $R,T$ so that its attaching part
is $\partial R*T$ and the other part is $R*\partial T$.
Choose a vertex $v\in R$. For faces $\eta\subset R\setminus\{v\}$,
in decreasing dimension, remove
\[
 (T*\eta,\ T*v*\eta).
\]
The low face is in the relative interior of the attaching part.
It has become interior in the enlarged shelling ball, so no later
top simplex contains it. Its other cofaces in $F$ have already
been removed. It is therefore free. These pairs remove $F$ and
the interior of its attaching part, leaving $R*\partial T$.
Induction leaves precisely $\partial Q$. Reversing the list gives
the asserted face order: all proper faces of a low face, and all
faces of its high face except the low face, are already present.
\end{proof}

\begin{lem}[The actual shellable tail]
\label{app:shell-tail}
In a controlled algorithm of Lemma~\ref{app:algorithm}, suppose a
prefix has separated a component $Q_0$ supported on a ball. Suppose
its remaining face list is the relative interior face list of a
shellable ball $Q$, with the inherited truncated-cell geometry.
Then the complete remaining sum on $Q_0$ is $Z(\partial\mathcal Q_0)$.
It has this value for every permissible choice of representatives
in that tail.
\end{lem}
\begin{proof}
Use the order from Lemma~\ref{app:shell-order}; the order comparison
in Proposition~\ref{app:algorithm-choices} permits it. If a top ball
becomes isolated, its final scalar evaluation may be performed
immediately: it has no remaining cut and its scalar commutes with
the other finite sums. This is only a rearrangement of the original
algorithm, in which all Move-$d$ evaluations are written last.

We check the geometry of one consecutive free-facet pair. Before
that pair all faces of $\tau$ except $\sigma$ and $\tau$ have
already been processed or belong to the fixed boundary. The
ordinary neighborhood of the pair is the elementary cancelling
block. Its incoming part is
\[
 \bigl((S^{k-1}\times I)\cup_{S^{k-1}}D^k\bigr)
       \times D^{d-k-1},\qquad k=\dim\sigma,
\]
a $(d-1)$-ball; its outgoing part is another such ball, with their
outer collar identified. One obtains this model by writing the
high cell as a ball with the low facet and its complementary
boundary ball, then using the shared boundary collars. It is the
elementary step behind the regular-neighborhood collapse theorem
\cite[Theorem 3.26, printed pp.41--43; PDF pp.49--51]{rourke2012introduction}.

Use the same collars in the marked construction. Centrally they
give the base times the normal angular disk. At an attaching end
the filling is unchanged, so the cap is absorbed by the same
product-collar reparameterization. After the low cut, the boundary
of the high core is the inserted $\mathcal D_2^k$ joined to the
complementary old $k$-disk. Its normal factor is $D^{d-k-1}$.
These are exactly the two cuts in the definition of the adjacent move, with the
orientation giving $F^k_+$ or $F^k_-$. The argument straightens a
single free-facet pair; it does not straighten an entire possibly
knotted coface fan. The choices are available in every preceding
state by the common-collar construction. The $k=0$ transport is
the specified one, since the inserted radial interval has
transport $1$.

The ordinary pair changes the remaining ball by a boundary collar
bulge. For the last-dimensional pair it can first isolate a ball,
which is evaluated by Move $d$. Thus the complete support before
and after each paired stage is again a ball, after these isolated
evaluations.

Now induct \emph{backwards} on the complete tail. The last ball is
evaluated by the defining Move $d$, so its value is $Z$. Assume
that the entire tail after a pair equals $Z$ for every current
complete boundary diagram. The preceding sum is then the two
adjacent relations followed by $Z$. The adjacent identity in the main text says that it equals
the old $Z$ again. This proves the assertion. In particular we
never insert a null relation into an unevaluated tail. MI is used
exactly through the $k=0$ case of the adjacent identity in the main text.
\end{proof}

\begin{rmk}
The remaining cells in this statement are inherited from a legal computation on the original full triangulation. They are not assumed to form a full triangulation relative to the boundary of the separated ball.
\end{rmk}

\subsection{Proposition~\ref{App: p-subdivision invariance}: invariance under a p-subdivision}
\label{appsec:stellar}

\begin{lem}[Common interface elimination]
\label{app:interface}
Let $K'=s_A^pK$ be a stellar subdivision of a full triangulation,
and suppose both triangulations are transverse to the same bulk.
Put $L=\lk_K A$, $Q=A*L$ and $J=\partial A*L$.
The nonboundary faces of $J$ can be processed as the same prefix
on both sides. After this prefix the inside and outside supports
are disjoint compact manifolds; the inside support is a ball.
For each prefix state its complete inside boundary diagram and
its weights are the same on both sides.
\end{lem}
\begin{proof}
The interface simplices are unchanged by the subdivision. Construct
their neighborhoods using the fine cells of $K'$ and their common
collars; the coarse cells are unions of these cells. The flag
neighborhoods preserve those unions, so the same cuts are legal
for both triangulations. Amalgamation of the fine collars gives
the coarse collars. In the central product a coarse angular face
is a union of fine angular faces. Put new coupons in the fine
top-dimensional angular faces. Hence the entire prefix, including
all its labels and buffers, can be shared.

We include the boundary-seam check. Any face in the intersection
of the closed inside and outside top-cell unions is contained in
an internal interface facet. Indeed, the top cells containing it
are the facets of a manifold link, whose facet-adjacency graph is
connected; a path from an inside to an outside cell crosses such
a facet. That facet is a nonboundary face of $J$ and has been
processed. Its whole remaining core, including its true-boundary
seam, is removed. The two closed remaining pieces are therefore
disjoint. The width at that seam is positive; no zero-width collar
is being used. Each is a union of components of the remaining
manifold and has the induced filtered collars.

To see that the inside is a ball, use the ordinary convex model
of Lemma~\ref{app:ordinary-cores}. Choose $c\in\Int A$, so that
$Q=c*J$. For every processed interface face $\sigma$, one has
$f_\sigma(c)>0$. Choose its level smaller than this value. On
each ray from $c$ the retained inequalities leave an initial
interval, of positive length uniformly near $c$. The finite
polyhedral radial decomposition is a cone on its outer frontier;
triangulate its faces and cone to $c$. It identifies the retained
set with a PL ball. When $A$ is a boundary face, $J$ is a ball and
$c$ is on the true boundary; the same construction is a cone on
a ball. The cellwise cap comparison transports this conclusion
to the actual cuts.

Finally the prefix uses the same physical cores, link boundaries,
transport sectors and representatives. Its boundary types are
inserted with the same parameters. Keep the paired dual-basis
indices when an interface facet is cut. The complete resulting
inside diagram is consequently identical. Finite distributivity
now gives
\[
 E(B,K)=\sum_s w_s E_{\rm in}(s,K)E_{\rm out}(s),
\]
and the analogous formula for $K'$, with the same $s,w_s$ and
outside factors. This formula is used only \emph{after} the
geometric separation has been proved.
\end{proof}

\begin{prop}[A stellar move with shellable link]
\label{app:shellable-stellar}
If the link of $A$ is shellable, the two algorithmic sums in
Lemma~\ref{app:interface} are equal. Choices inside the star may
be arbitrary; use the same choices in the common prefix and outside.
\end{prop}
\begin{proof}
Both $A*L$ and $p*\partial A*L$ are shellable balls. After the
common prefix, their remaining faces are their relative interior
faces. Lemma~\ref{app:shell-tail} makes both complete inside sums
equal to $Z$ of the same boundary diagram. Multiply by the common
outside factors and weights and sum over all prefix states.
Boundary stellar moves are included; only actual nonboundary
faces of the interface are processed.
\end{proof}

\begin{lem}[Factoring a stellar subdivision]
\label{app:stellar-factor}
Every stellar subdivision of a finite combinatorial manifold is
a finite composition of stellar subdivisions and their inverses
whose centers have shellable links. The factorization fixes the
old interface and exterior. If the initial triangulation is full,
all intermediate triangulations in this factorization are full.
\end{lem}
\begin{proof}
Links and joins of shellable balls or spheres are shellable.
Stellar subdivisions preserve shellability. For the latter claim,
replace a facet containing a center $A$ by the facets
$b*(F\setminus\{v\})$, $v\in A$, placing first the $v$ whose
opposite facet belonged to the old attaching part. Their attaching
intersections are precisely the old attaching facets and the
facets shared with earlier new simplices. They satisfy the ball
or sphere shelling condition.

Prepare a ball or sphere $L$ by induction on dimension. For each
original vertex $v$, prepare $\lk_Lv$ by stellar subdivisions at
centers with shellable links, using the induction. Lift a center
$C$ to $v*C$. Its link is $\lk_{\lk_Lv}C$. Previously prepared
vertex links undergo stellar subdivisions or stay unchanged; a
new vertex has shellable link $\partial(v*C)*\lk_L(v*C)$.
After all original vertices have been treated, every vertex link,
and hence every nonempty face link, is shellable. Now apply
\cite[Corollary I.3.10, 2013 preprint, printed/PDF p.11]{adiprasito2017subdivisions}:
enough derived subdivisions make $L$ shellable. Implement them
by stellar subdivisions in decreasing face dimension. All their
centers already have shellable links. The empty link is the
$(-1)$-sphere, the join unit, and causes no exception.

For $L=\lk_K A$ lift this preparation to both sides of the desired
move. For a center $C\subset L$ and new vertex $b$, the exchange
identity is
\[
 s_A^p s_{A*C}^bK=s_{p*C}^b s_A^pK.
\]
If $H=\lk_LC$, both sides replace the affected region by
$\partial A*b*(C\cup p*\partial C)*H$. Their lifted centers
have links $H$ and $\partial A*H$, respectively. Prepare on the
left, perform $s_A^p$ when its link is shellable, then reverse the
preparation on the right. All centers have the required links.
Each forward preparation is a subdivision of its full endpoint,
so the reversed right-hand path is full as well.
\end{proof}

\begin{prop}[Main text, Proposition~\ref{App: p-subdivision invariance}]
\label{app:p-subdivision}
The controlled algorithm on transverse, boundary-full triangulations
is invariant under $p$-subdivision, with the common exterior and
interface data retained.
\end{prop}
\begin{proof}
Take the finite factorization in Lemma~\ref{app:stellar-factor}
and a common subdivision $H$ of all its members. The fine endpoint
$K'$ is already transverse. Theorem~\ref{App: transversal subdivision preparation} gives a $K'$-cell-preserving
isotopy making the bulk transverse to $H$, hence to every member
of the factorization by amalgamation. This also preserves every
$K$-cell. Transport neighborhoods, representatives and labels at
the endpoints. Proposition~\ref{app:algorithm-choices} identifies
their sums with the original sums; this is cell-preserving
naturality, not general isotopy invariance. Compare each middle
step by Proposition~\ref{app:shellable-stellar} and transport back.
All intermediate triangulations are full. No general ball
evaluation has entered the proof.
\end{proof}

\paragraph{Relative computations after a cutting prefix.}
The local comparison also applies to the inherited cell structures left
by the algorithm. We give this technical extension here. Its geometric
and local numerical proof does not use global relations. Identifying
the resulting full-triangulation sum with a partition function uses
the main text's triangulation independence; this consequence is used
only after that result is available.

The order of the arguments is important. The absolute stellar comparison
above uses the adjacent identity and the shellable-tail calculation,
not the restarting result below. Subdivision invariance, ball evaluation,
and isotopy invariance then give the main text's triangulation independence.
Only after that theorem may Proposition~\ref{app:restart} identify a
remaining sum with a new partition function. Neither
Proposition~\ref{app:restart} nor Proposition~\ref{app:expose-interface}
is an input to any of those earlier invariance proofs.

\begin{lem}[Derived refinement of residual cells]
\label{app:relative-starring}
Let $C$ be the residual cell complex of a controlled cutting prefix.
There is a finite sequence of controlled cellular stellar comparisons
from $C$ to a transverse triangulation with full boundary, preserving
the complete remaining scalar sum and its entire boundary diagram.
The terminal subdivision may be obtained by starring all original
positive-dimensional cells, including boundary cells, in decreasing
dimension.
\end{lem}
\begin{proof}
Use the ordinary truncated-polytopal model and the simultaneous
marked whole-cell collars supplied by the controlled construction.
Do not first subdivide boundary faces while retaining their
nonboundary cofaces. Instead perform one descending-dimensional
derived subdivision of the entire residual complex. This is the order
of auxiliary refinements; every algorithm compared below still
processes nonboundary faces in face-first order.

At the turn of an original cell $D$, its proper faces have not yet
been subdivided, whereas all its proper cofaces have been coned.
The two closed stars and their common interface are therefore
\[
 Q=D*L,\qquad Q'=p_D*\partial D*L,\qquad J=\partial D*L,
\]
where $L$ is the simplicial complex of chains of strict cofaces.
It is a PL sphere when $D$ is nonboundary and a PL ball when $D$
is a boundary cell. The boundary of $D$ is still its original convex
polytopal boundary. Coning it and extending by the join with $L$
is a face-to-face polytopal subdivision in the ordinary model.
No shellability assertion for an arbitrarily pre-refined boundary
cell structure is required.

We verify the controlled geometry. An unaltered nonboundary cell
retains the inherited core and collars. Starring a nonboundary
center introduces an ambient-interior vertex. At a boundary center,
every nonboundary coface already contains the interior center of an
original nonboundary coface, since higher cells were coned first.
Boundary starring retains these vertices. Every new nonboundary
cell thus has an interior cone vertex $v$. A boundary face $\beta$
other than the entire opposite facet belongs to the proper
nonboundary face $v*\beta$, which is processed earlier. Its cut
removes the corresponding true-boundary seam. Surviving contact
with the true boundary is consequently in the relative interior of
an opposite boundary facet, with the half-space and inward-ray
model of Lemma~\ref{app:ordinary-cores}, or there is no such contact.

Extend the short boundary caps through the existing whole-cell
collars, using the PL corner rectangle of Lemma~\ref{app:marked-core}.
A short neighborhood of a proper face of a convex polytope is a
boundary cap, with the usual truncated-cell complement. Successive
caps have the corresponding orthant collars. The new literal cores
are therefore proper locally flat disks, and all remaining cells
retain simultaneous marked whole-boundary collars. The marked normal
neighborhoods and buffered fillings of Lemmas~\ref{app:marked-core}
and~\ref{app:algorithm} apply, including simultaneous choices for
finitely many states. This concerns these inherited and coned
polytopal models, not arbitrary regular CW decompositions.

Order and choice comparisons retain the same control. Two
incomparable current cores meet only in the current boundary,
since their common nonboundary proper faces were processed earlier.
Their intersection is a face, or its inherited truncation, in the
ordinary polytopal model. Take a common short derived neighborhood
of it in the whole-cell collars. In each core this is a boundary
cap. Whichever core is processed first removes the common cap;
off it the compact cores have disjoint short tubes. Both fillings
are unchanged there. Interchanging the cuts changes only marked
boundary-collar parameterizations, so the proof of
Proposition~\ref{app:algorithm-choices} applies. Its one-dimensional
transport check is unchanged: earlier central changes are vertex
fillings with radial transport $1$. Keep coupon carriers fixed
in these exchanges. Their subsequent comparison uses that
proposition's inverse-pairing tensor: the angular cell sphere has
connected facet-adjacency graph, and its codimension-one faces index
the last nonterminal contractions. This compares the complete
contractions, not individual carrier-dependent terms.

For the numerical argument use recursive cellular shellings:
the first top-cell boundary is shellable, and each later attaching
part begins a shelling of the new top-cell boundary
\cite[Definition 2.3, author's preprint, printed/PDF pp.2--3]{ziegler1998shelling}.
Convex polytopal boundaries have line shellings \cite{bruggesser1971shellable}.
Joins preserve these recursive shellings. To check this, order
joined top cells lexicographically. In a cell $E*F$ the attaching
part is $(A*F)\cup(E*B)$, where $A,B$ are the factor attaching parts.
Shellings of $\partial E,\partial F$ beginning with $A,B$ induce a
shelling of $\partial(E*F)$ beginning with this union: shuffle their
induced facets, placing those from $A,B$ first. Induction on the sum
of factor dimensions checks each attaching intersection, which has
the same displayed form with a factor of lower dimension. Empty
and whole factor attaching parts are allowed.

The cellular version of Lemma~\ref{app:shell-order} follows as well.
For a recursively shellable ball $Q$, delete its open first top
cell $\Omega$. Suppose the preceding shelling part has been reduced
to its boundary, and let $F$ be the next top cell, with attaching
ball $A$. By the recursive definition $A$ is shellable. Induction
on dimension supplies a facet $\theta$ such that
$A\setminus\Int\theta\searrow\partial A$.
The pair $(\theta,F)$ is free: $\Int A$ has become interior to the
enlarged shelling ball, so no later top cell meets it. Remove this
pair and then collapse $A\setminus\Int\theta$ relative to
$\partial A$. These pairs remain free for the same reason.
Induction leaves $\partial Q$. The zero-dimensional attaching case
is the single vertex--edge pair. Reversing the collapses gives a
face-first list of consecutive free-facet pairs, followed by
$\Omega$. A convex cell with one facet and its complementary
boundary ball, with their collars, has the cancelling-block model
of Lemma~\ref{app:shell-tail}. The adjacent identity and backward
induction on the complete tail give that lemma's evaluation for
these cellular shellings.

If $L$ is shellable, both $D*L$ and $p_D*\partial D*L$ are shellable
cell balls. Apply the common-prefix construction of
Lemma~\ref{app:interface} to the nonboundary faces of $J$.
Its marked-collar and radial arguments apply to the convex cell
$D$ in place of a simplex. When $D$ is a boundary cell,
\[
 \partial Q=(\partial D*L)\cup(D*\partial L),
\]
and $D*\partial L$ is in the true boundary. Only nonboundary
interface faces are processed. Each separated inside tail is
therefore relative to the entire boundary of its local ball,
including this true-boundary part. The inside boundary diagrams,
paired basis indices and prefix weights coincide. The cellular
shellable-tail calculation makes both complete inside sums equal
to $Z$ of that diagram. Their common outside tails give equality
of the complete sums. This argument also covers shellable links
which are polytopal joins of the factors just described.

For general simplicial $L$, use the good-link preparation of
Lemma~\ref{app:stellar-factor}. A center $A$ with shellable link
$H$ lifts to $D*A$ and $p_D*A$ on the two sides. Their links are
$H$ and $\partial D*H$, both shellable by the join argument; the
second need not be simplicial. The first center is the convex join
of $D$ and a simplex and also has a polytopal boundary. Each lifted
comparison thus uses the shellable case already proved. The join
exchange identity of Lemma~\ref{app:stellar-factor} holds with the
cell $D$ in place of its simplex factor. Prepare on both sides,
compare when the middle link is shellable, and reverse the second
preparation. These auxiliary operations fix $\partial D$ and all
other unprocessed original cells. A nonboundary auxiliary center
introduces an interior apex; a boundary center retains an interior
vertex in every nonboundary coface. Their controlled geometry is
covered above. No induction on a possibly larger second link,
no restarting assertion and no global relation is used. The
empty-link case is the single top cell and is covered directly
by the shellable calculation.

All constructions have finite common subdivisions. Apply the
cellwise relative positioning of Lemma~\ref{app:relative-position}
to their ball cells, on their boundaries first and then on their
interiors, retaining every original cell set and smaller marked
collars. Inverse motions position the subdivisions while the
labeled manifold stays fixed. The comparison under these coordinate
changes is cell-preserving naturality, not general isotopy invariance.

Finally the simplices of the derived subdivision correspond to
chains of original cells. Their vertices are the corresponding
centers, with original vertices retained. If all vertices are in
the true boundary, the chain's largest cell is a boundary cell,
so the entire simplex is in the boundary. This proves fullness.
The preceding positioning makes the triangulation transverse, and
all comparisons retain the entire boundary diagram and scalar sum.
\end{proof}

\begin{prop}[Restarting the computation]
\label{app:restart}
Assume that the algorithmic sum is independent of full transverse triangulations. After any controlled cutting prefix, the complete remaining sum
on each separated component equals its partition function computed
from a new full transverse triangulation. In particular, if the
prefix has states $s$, weights $w_s$, and disjoint remaining bulk
vectors $B_{s,1},\ldots,B_{s,r_s}$, then
\[
 E(B,K)=\sum_s w_s\prod_{j=1}^{r_s}E(B_{s,j}).
\]
All current boundary diagrams, including paired basis indices, are
retained in this formula.
\end{prop}
\begin{proof}
Fix a prefix state $s$ before making any comparison. The operations
remaining in distinct components have disjoint supports. Finite
distributivity therefore gives their product of complete inherited
tail sums; at this point no tail has been identified with a new
partition function.

For each component, Lemma~\ref{app:relative-starring} supplies a
full transverse triangulation and equates its algorithmic sum with
that particular inherited tail. Its comparison retains the complete
boundary diagram, including coupon and paired-basis data. The
assumed triangulation independence, and only that already proved
result, identifies this full-triangulation sum with $E(B_{s,j})$.
Multiply these equalities and sum with the original coefficients
$w_s$. This is the displayed formula. No global relation is used.
\end{proof}

\begin{prop}[Exposing a collared interface]
\label{app:expose-interface}
Let $T$ be a compact separating PL hypersurface with a filtered
two-sided collar, allowing a boundary in the true boundary with the
specified corner collar. There is a full transverse triangulation
of the bulk having $T$ as a subcomplex. Its nonboundary faces can
be processed first. The resulting components are the two sides
with their $T$-collars shortened, and the restarting formula of
Proposition~\ref{app:restart} applies.
\end{prop}
\begin{proof}
Choose a transverse triangulation of $T$ which is full relative to
$\partial T$. Extend it over each side and make each extension full
relative to that side's whole boundary. Their union is full relative
to the true boundary of the bulk. Indeed an all-true-boundary
simplex lies in a boundary face of one side; if this face belongs
to $T$, fullness of $T$ puts it in $\partial T$. Relative positioning
retains the common triangulation. Match the two filtered collars
in the construction.

The nonboundary faces of $T$ form a lower ideal in the nonboundary
face order, so they can be processed first. The union of their
short controlled neighborhoods is the regular neighborhood of $T$
in its given two-sided corner collar. One can see this cell by cell:
the cut on the boundary of a $T$-cell is extended by its collar,
and its interior cut completes the normal interval over that cell.
On common faces the intervals agree. The corner caps have positive
width at $\partial T$. Their removal therefore separates the sides
and shortens their collars, including at the seam. The arbitrary
coupons change the complete boundary diagrams, not these supports.
Apply Proposition~\ref{app:restart} separately to each prefix state.
\end{proof}

\subsection{Lemma~\ref{App: p-subdivision sequence}: a full-preserving stellar connection}
\label{appsec:full-connection}

\begin{lem}[A full-preserving stellar connection]
\label{app:full-connection}
Let $\Delta$ be a finite triangulation with full boundary, and let
$\Gamma$ be a subdivision of $\Delta$ on the same underlying
polyhedron. They are connected by $p$-subdivisions and their
inverses through full triangulations, with the given underlying
manifold identification retained.
\end{lem}
\begin{proof}
Realize $\Delta$ linearly in a Euclidean space and realize
$\Gamma$ as its prescribed subdivision on the same polyhedron.
By \cite[Theorem 1.1, author's version of 21 April 2026, printed/PDF p.1]{adiprasito2026all}, there is a common iterated stellar
subdivision $R$. Thus there are finite sequences
\[
\Delta \longrightarrow \cdots \longrightarrow R
\longleftarrow \cdots \longleftarrow \Gamma,
\]
where each arrow is a stellar subdivision, equivalently a
$p$-subdivision.

Every subdivision of a full triangulation is full. Indeed,
for each old simplex $\sigma$, fullness implies that
$|\sigma|\cap\partial B$ is the face spanned by the boundary
vertices of $\sigma$ (possibly the empty face). Any new simplex
contained in $|\sigma|$ whose vertices all lie in $\partial B$
is therefore contained in that face, and hence in $\partial B$.

Consequently, $\Gamma$ and all intermediate triangulations
in both sequences are full. Reading the second sequence backwards
gives the required sequence of $p$-subdivisions and their inverses.
All operations retain the underlying polyhedron and its given
identification with $B$.
\end{proof}

\subsection{Lemma~\ref{App: Filling: global relation can cut second for any general filling}: a second cut through a general filling}
\label{appsec:second-cut}

\begin{lem}[A product disk through a general shell; Lemma~\ref{App: Filling: global relation can cut second for any general filling}]
\label{app:second-cut}
Let $0\le k<n$, $q=n-k$, and let a bulk vector have underlying
manifold $(D^k\times I)\times S^q$. Suppose its two end faces carry
$\mathcal D_2^k(\alpha_2)\times S^q$ and
$\mathcal D_1^k(\alpha_1)\times S^q$, and its side carries
$\mathcal S^{k-1}\times I\times S^q$, with the boundary collars
required of the bulk vector. There is an ambient PL isotopy fixed
pointwise on its entire boundary and a small closed angular disk
$C\subset S^q$, avoiding both coupon carriers, such that the
transported $G$ structure on $(D^k\times I)\times C$ is
$\mathcal D^{k+1}\times C$ for a $G$-labeled disk
$|\mathcal D^{k+1}|=D^k\times I$.
The product holds on a slightly larger disk, so its outgoing face
has a two-sided filtered collar. Removing its smaller interior
leaves a PL $(n+1)$-ball.
\end{lem}
\begin{proof}
Put $H=D^k\times I$ and project $H\times S^q$ to $S^q$.
By the small-carrier convention of the main text, there is a
nonempty open angular disk $C_0$ disjoint from both coupon carriers.
The prescribed boundary is an angular product over this disk. Triangulate the bulk filtration, its boundary faces and the
carriers simultaneously, and make the projection simplicial.
Choose $p$ in the relative interior of a top target simplex inside
$C_0$ and a small closed disk $C^+\subset C_0$ around it.
Lemma~\ref{app:williamson} gives a filtration-preserving PL
trivialization
\[
 h:H\times C^+\longrightarrow\pi^{-1}(C^+),\qquad
 h(x,y)=(h_y(x),y),\quad h_p=\id.
\]
The fiber filtration is a recursive stratification of a PL
$(k+1)$-disk by Lemma~\ref{app:slice}. Its boundary is the prescribed
union of $\mathcal D_1^k$, $\mathcal D_2^k$ and the side cylinder,
with the dimensions shifted by $q$.

We spell out why the isotopy fixes the \emph{whole} boundary, including
the labeled end faces. On $C^+$ the boundary structure is a product,
so $b_y=h_y|_{\partial H}$ is a filtered family of boundary
homeomorphisms, $b_p=\id$. Choose a filtered collar of $\partial H$
in the central fiber and a jointly PL contraction parameter
$r:[0,1]\times C^+\to C^+$ with
$r(0,y)=y$, $r(1,y)=p$, $r(s,p)=p$. On the collar define
\[
 a(x,s,y)=(b_{r(s,y)}^{-1}(x),s,y),
\]
and use the identity off it. As in Lemma~\ref{app:product-region},
$g=h\circ a$ is a filtered trivialization with $g_p=\id$ and
$g_y|_{\partial H}=\id$ for every $y$.

Extend the identity of $C^+$ to a PL map $\rho:S^q\to C^+$ and
choose a jointly PL homotopy $\rho_t$ from the constant map $p$
to $\rho$. Then
\[
 H_t(x,y)=(g_{\rho_t(y)}(x),y)
\]
is an ambient PL isotopy fixing $\partial H\times S^q$ pointwise.
Take its inverse to obtain the pushforward convention. On $C^+$
the new pattern is the required product; the actual coupons elsewhere
on the boundary have not moved. Labels agree by their restriction to
the central fiber. Choose $C\subset\Int C^+$ so that its sphere
boundary has room for a collar. The product with that ordinary
angular collar is a filtered collar of $H\times\partial C$,
including its corners at $\partial H$.

The closed complement of $C$ in $S^q$ is a PL $q$-ball. Thus the
remaining support is $H\times D^q\cong D^{n+1}$. No global relation
and no gauge invariance was used in this geometric construction.
The case $k=n$ is not included: its normal sphere is $S^0$, and the
two local components are already separated by the top pairing move.
\end{proof}

\subsection{Lemma~\ref{App: chamber divided by balls}: localizing chamber gauges in collared balls}
\label{appsec:chamber-balls}

\begin{lem}[Compactly supported refinement into small chambers]
\label{app:local-refinement}
Let $P$ be a recursive PL stratification on a compact PL manifold
$B$, and let $K\subset\Int B$ be compact. There is a refinement,
unchanged on a neighborhood of $\partial B$, which agrees near $K$
with the open-simplex stratification of a finite triangulation
adapted to $P$. The closed simplicial neighborhoods of the top simplices meeting $K$
can be required to lie in the region where this description holds.
Here the simplicial neighborhood of $\sigma$ is the union of the
closed simplices meeting $\sigma$, including their faces.
\end{lem}
\begin{proof}
Choose open sets $K\subset W_0\Subset W_1\Subset\Int B$.
Triangulate all closed filtration terms of $P$ simultaneously and
subdivide sufficiently finely that the closed simplicial neighborhoods
of simplices meeting $\overline W_0$ lie in $W_1$. After another subdivision,
choose a subcomplex $L$ whose support lies in $\Int B$ and contains
these neighborhoods in its interior. Such choices follow from a Lebesgue
number for a finite cover of the compact sets in question. The set $K$ need not itself be a subcomplex; the construction only
uses a finite simplicial neighborhood of it.

Inside $|L|$ use the relative interiors of its simplices as strata.
Outside $|L|$ use the connected components of $S\setminus|L|$ for
old strata $S$. Each of these is an open subset of a PL manifold;
inside $L$ the strata are open simplices. Each new stratum is
contained in one old stratum because the triangulation is adapted.
Their dimension filtration is closed: a new open simplex has its
boundary in smaller simplices, while the boundary of an exterior
piece consists of old lower strata and simplices of $L$ of smaller
dimension. The latter assertion follows by looking in a simplex of
the adapted triangulation; an equal-dimensional face in its closure
would be part of the same old stratum, and its boundary has smaller
dimension.

We verify the recursive conical condition at the seam, which is the
only point not immediate from the old or simplicial charts. At a
point in the relative interior of a face $\sigma\subset L$, the
triangulated star has its standard form
$\mathbb R^{\dim\sigma}\times c\lk(\sigma)$, preserving all old
filtration terms and $L$. Membership in these subcomplexes is
independent of the first coordinate. The new strata inside $L$
are the corresponding simplicial product strata; outside $L$ they
are the old strata restricted to this open complement. The induced
strata on a small normal link are PL manifolds: on the old part this
follows by removing the translation and radial product directions
using Lemma~\ref{app:slice}, and on the simplicial part it is the
simplex-link description. Repeating the same construction in a
proper link decreases dimension, and gives its recursive charts.
The normal link has its standard PL sphere support. Thus these are
the charts of Definition~\ref{stratified PL manifold}, including the seam. Equivalently this
is the relative simplicial refinement of a conically filtered
polyhedron, verified here with its marked subcomplex $L$.

The construction is finite, so local finiteness holds. It leaves an
open boundary neighborhood completely unchanged; in particular the
old filtered boundary collar survives. The chosen simplicial-neighborhood containment
ensures the last assertion.
\end{proof}

\begin{lem}[Collared balls around small chambers; Lemma~\ref{App: chamber divided by balls}]
\label{app:chamber-balls}
Let $(B,P)$ be a compact $G$-labeled $d$-manifold and
$K\subset\Int B$ compact. There is a refinement $P'\succcurlyeq P$,
unchanged near $\partial B$, such that for every new chamber $D$
contained in $K$ there is a PL ball $U_D\subset\Int B$ with
\[
 \overline D\subset\Int U_D,
\]
and $\partial U_D$ has a two-sided $G$-stratified collar in $(B,P')$.
In particular $U_D$ with its restricted filtration is a stratified
submanifold in the main text's definition of a stratified submanifold.
\end{lem}
\begin{proof}
Use Lemma~\ref{app:local-refinement}. Label a new wall in an old
chamber by $1$ and a piece of an old wall by its inherited label,
with the inherited coorientation. This is a $G$-labeled refinement:
a new codimension-two small circuit either has no nontrivial
crossing, crosses an old wall twice in opposite directions, or
is the subdivision of an old flat circuit. Its product is $1$.

Every chamber $D\subset K$ is the relative interior of a top simplex
$\sigma$ of the chosen triangulation $T$. Its closed simplicial
neighborhood lies in the interior region carrying the simplicial
stratification. The simplex $\sigma$, with all its faces, is a full
subcomplex of $T$. Define the simplicial function $f:T\to[0,1]$ by
$f(v)=0$ on vertices of $\sigma$ and $f(v)=1$ on all other vertices.
Then $f^{-1}(0)=|\sigma|$. For $0<t<1$ set
\[
 U_D=f^{-1}([0,t]).
\]
This lies in the indicated neighborhood and contains $\sigma$ in its
interior. It is the $t$-neighborhood of the full subcomplex in the
terminology of \cite[printed p.32; PDF p.40]{rourke2012introduction}. Derive each mixed
simplex at a point on $f^{-1}(t)$; the resulting derived neighborhood
has exactly this support, by the construction on that page. Thus
$U_D$ is a regular neighborhood of $\sigma$.
The closed simplex collapses to a point. A regular neighborhood of
it is therefore a PL $d$-ball by the regular-neighborhood theorem
\cite[Corollary 3.27, printed p.41; PDF p.49]{rourke2012introduction}.

Here is the needed \emph{stratified} collar, which does not follow
just from that ordinary ball statement. Let $f$ be the above
simplicial height function. On a sufficiently short interval
$(t-\delta,t+\delta)$ there are no vertex levels. The simultaneous
simplicial trivialization of Lemma~\ref{app:williamson} identifies
this strip with
\[
 f^{-1}(t)\times(t-\delta,t+\delta),
\]
preserving every closed simplex of the original adapted
triangulation. It consequently preserves their relative interiors,
which are precisely the strata of $P'$ here. After reparameterizing
the interval this is the asserted two-sided collar, with its zero
level fixed. Its G labels are constant along each product wall.
Product slicing gives the recursive boundary charts, and restricting
to the inner half gives the boundary collar of $U_D$ itself.
Thus $U_D$ is the required stratified submanifold.
\end{proof}